\documentclass[10pt,a4paper,envcountsame]{amsart}

	\expandafter\let\csname ver@amsthm.sty\endcsname\relax
	\let\theoremstyle\relax

\usepackage{amsaddr}
 \usepackage[colorlinks,citecolor=blue,urlcolor=blue, linkcolor=blue, backref]{hyperref}
\usepackage{amsthm,amssymb} 
\usepackage[capitalise]{cleveref}

\usepackage[usenames,dvipsnames]{color}
\usepackage{orcidlink}

 \usepackage[all]{xy}
\usepackage{amssymb,enumitem}

\usepackage{tikz}
\usetikzlibrary{arrows,snakes,positioning,backgrounds,shadows}

\newtheorem{theorem}{Theorem}[section]

\newtheorem{thm}[theorem]{Theorem}

\newtheorem{proposition}[theorem]{Proposition}

\newtheorem{prop}[theorem]{Proposition}

\newtheorem{notation}[theorem]{Notation}

\newtheorem{criterion}[theorem]{Criterion}

\newtheorem{cor}[theorem]{Corollary}

\newtheorem{questions}[theorem]{Questions}
\newtheorem{fact}[theorem]{Fact} 

\newtheorem{claim}[theorem]{Claim}
\theoremstyle{definition}
\newtheorem{definition}[theorem]{Definition}
\newtheorem{example}[theorem]{Example}

\newtheorem{remark}[theorem]{Remark}

\newtheorem{lemma}[theorem]{Lemma}

\newcommand{\NN}{{\mathbb{N}}}
\newcommand{\RR}{{\mathbb{R}}}

\newcommand{\QQ}{{\mathbb{Q}}}
\newcommand{\ZZ}{{\mathbb{Z}}}
\newcommand{\FF}{{\mathbb{F}}}
\newcommand{\sub}{\subseteq}
\newcommand{\sN}[1]{_{#1\in \NN}}
\newcommand{\sNp}[1]{_{#1\in \NN^+}}
\newcommand{\uhr}[1]{\! \upharpoonright_{#1}}

\newcommand{\bi}{\begin{itemize}}
\newcommand{\ei}{\end{itemize}}
\newcommand{\bc}{\begin{center}}
\newcommand{\ec}{\end{center}}

\newcommand{\ES}{\emptyset}

\newcommand{\tp}[1]{2^{#1}}
\newcommand{\ex}{\exists}
\newcommand{\fa}{\forall}

\newcommand{\la}{\langle}
\newcommand{\ra}{\rangle}

\newcommand{\n}{\noindent}

\newcommand{\sss}{\sigma}
\newcommand{\aaa}{\alpha}

\renewcommand{\S}{S_\infty}

\newcommand{\lland}{\, \land \, }

\newcommand{\Tree}[1]{\mathit{Tree}(#1) }

\newcommand \seq[1]{{\langle{#1}\rangle}}

\newcommand\+[1]{\mathcal{#1}}

\newcommand{\wt}{\widetilde}
\newcommand{\ol}{\overline}
\newcommand{\ul}{\underline}
\newcommand{\ape}{\hat{\ }}
\renewcommand{\ape}{}

\newcommand{\lra}{\leftrightarrow}
\newcommand{\LR}{\Leftrightarrow}
\newcommand{\RA}{\Rightarrow}
\newcommand{\LA}{\Leftarrow}

\newcommand{\rapf}{\n ``$\RA$":\ }
\newcommand{\lapf}{\n ``$\LA$":\ }

\newcommand{\sssl}{\ensuremath{|\sigma|}}

\newcommand{\range}{\ensuremath{\mathrm{range}}}
\newcommand{\dom}{\ensuremath{\mathrm{dom}}}

\DeclareMathOperator \SL{SL}
\DeclareMathOperator \GL{GL}
\DeclareMathOperator \PGL{PGL}

\renewcommand \hat \widehat

\newcommand{\Op}{\text{\it Mult}}
\newcommand{\Inv}{\text{\it Inv}}

\newcommand{\Aut}{\text{\it Aut}}
\renewcommand{\epsilon}{\varepsilon}
\newcommand{\andre}[1]{{\textcolor{red}{O: #1}} }
\newcommand{\sasha}[1]{{\textcolor{blue}{H: #1}} }
\renewcommand{\andre}[1]{{\textcolor{red}{}} }
\renewcommand{\sasha}[1]{{\textcolor{blue}{}} }
\begin{document}

\title{Computably  totally disconnected   locally~compact groups}

    \author{Alexander Melnikov  \orcidlink{0000-0001-6878-7493}}
    \address{School of Mathematics and Statistics, \\
Victoria University of Wellington, New Zealand}
\email{sasha.melnikov@vuw.ac.nz} 
    
      \author{Andre Nies  \orcidlink{0000-0002-0666-5180}}
          \address{School  of Computer Science, \\
The University of Auckland, New Zealand}
\email{andre@cs.auckland.ac.nz}

\maketitle

\begin{abstract}   We study  totally disconnected, locally compact (t.d.l.c.) groups from an algorithmic perspective. We give various approaches to defining computable presentability of a   t.d.l.c.\ group, and   show their equivalence. In the process, we  obtain an algorithmic Stone-type duality between t.d.l.c.~groups and certain countable ordered groupoids given by the compact open cosets. Several natural groups, such as $\mathrm{Aut}(T_d)$ and $\mathrm{SL}_n(\mathbb Q_p)$, have computable presentations. We provide  a criterion  based on the duality when a  computable presentation of a t.d.l.c.~group is unique up to computable isomorphism.   We show that many constructions leading from t.d.l.c.\ groups to new t.d.l.c.\ groups have algorithmic versions that stay within the class of computably presented t.d.l.c.\ groups; most prominently, quotients by computable closed normal subgroups.
 We study whether objects associated with computably t.d.l.c.\ groups are computable: the modular function, the scale function, and Cayley-Abels graphs in the compactly generated case.  
%
 \end{abstract}
 \setcounter{tocdepth} 1
\tableofcontents
\section{Introduction}
In the past 10 years, research  on  totally disconnected, locally compact  (t.d.l.c.) groups such as Willis~\cite{Willis:17} has increasingly focussed on their algorithmic aspects.    We  develop  a general  algorithmic theory of  t.d.l.c.\  groups. This   establishes  a   theoretical framework, and also  enables us  to prove non-computability results, such as  the existence  of an algorithmically t.d.l.c.\ group with a non-computable scale function.  

    In this paper we assume that  all topological groups  have a countable basis.  
       Our theory will  address the following:

 \begin{questions} \label{qu:quest}  \  \nolinebreak \nopagebreak
{\rm  \begin{enumerate}[label=(\alph*)] \item How can one define a computable presentation of  a t.d.l.c.\ group?     
 
\item  Which  t.d.l.c.\ groups have such a presentation?

   \item  Given a computable  presentation, which associated  objects  (such as the rational valued Haar measures) are   computable?

%
\item Which    constructions that lead from given  to new t.d.l.c.\ groups have algorithmic versions?
    \item When is      a computable presentation of a t.d.l.c.\ group unique up to computable isomorphism?

 \end{enumerate} }
 \end{questions}

%

We begin with  an informal  summary of our main results, and the extent to which they answer these questions.   Then we   proceed to some background on t.d.l.c.\  groups and on computability. We will  discuss  the five questions and results answering them  in more detail   in the corresponding  Subsections~\ref{s: types of presentations}--\ref{ss:auto} of the introduction. 
\subsection{Informal summary of answers to the questions}  \label{s:backgroundQ} \

\n (a)  We will introduce   two approaches to   defining a computable presentation of a t.d.l.c.\ group. Each of them addresses, in   different ways,  the problem that the domain is generally uncountable.  One approach, which we call  \emph{Baire presentation}, employs computation  based on approximations,  as is done in   the field of   computable analysis. The other   approach uses  a duality between t.d.l.c.\ groups and certain countable structures  we call \emph{meet groupoids}. We will show that the two approaches are equivalent,  in the sense that a group has a computable Baire presentation if, and only if,  it has a computable presentation via its meet groupoid. This duality is a core technique towards answering the remaining questions. 
 It  also shows that Baire presentations can be ``enhanced" to presentations as closed subgroups of $\S$: then  the group operations are canonically computable, because they are given by the   operations of composition and inverse on permutations of the natural numbers. 
 
\n   (b) Our thesis   is  that \emph{all} ``natural" countably based  groups that are considered in the field of   t.d.l.c.\ groups have computable  presentations.  We use our methods to collect evidence for this thesis. For instance, we   show that the groups of the form   $\mathrm {PGL}_n(\mathbb Q_p)$ are computable.

\n  (c)   Haar measure, modular function, and Cayley-Abels graphs (for the compactly generated case) are computable, while the scale function can fail to be computable.

\n  (d) Many    constructions of this sort   have algorithmic versions. This is  expected for constructions such as taking closed subgroups  and forming   local  direct products. It is   more surprising for taking    quotients by   closed normal subgroups, given that in the uncountable case there is no canonical    set of coset  representatives. The full power of the duality will be  needed.

\n  (e) While we stop short of giving a   full answer to the question, we  provide  a criterion,  based on the duality, of when a  computable presentation of a t.d.l.c.~group is unique up to computable isomorphism.  We apply this to show that the split extension of $\mathbb Q_p$  by $\ZZ$ with the action  of a generator given by $x \mapsto xp$ has a unique computable presentation.

  We hope that the first five  sections of the paper will  be  accessible to readers with only   basic knowledge of   computability theory; as we progress, we will explain some     notions from computability  theory  that are  more advanced.
  
 Future work  could use   our  framework to address the resource-bounded level, and even the practice of computation in t.d.l.c.\ groups.   Several recent works study these topics for particular groups. We mention Matthew Conder~\cite[4.1]{Conder:20}: given a non-Archimedean local field  $K$ such as $\QQ_p$,  an algorithm is described to determine whether two input elements $x,y$ of $\mathrm {SL}_2(K)$ (or $\mathrm {PSL}_2(K)$) generate a discrete free subgroup. Conder subsequently discusses  an implementation in the computer algebra system MAGMA. Such systems can only work with approximations to the elements in the uncountable domain.  Thus,    theoretical underpinning is needed for   claims that such algorithms run efficiently.  Future work could also explore the possibility to directly compute with the countable dual structure in computer algebra systems.

\subsection{Background on t.d.l.c.\ groups}  \label{s:background}
  Van Dantzig~\cite{Dantzig:36}   showed that each t.d.l.c.\ group has a neighbourhood basis of the identity consisting of compact open subgroups.  
With Question~\ref{qu:quest}(a) in mind,   we discuss  some   examples of t.d.l.c.\ groups.  We will return to them  repeatedly during the course of   the   paper.

\n  (i)  All countable discrete groups  and all profinite   groups are t.d.l.c.
 
\n  (ii) $  ( \mathbb Q_p,+)$,  the additive group of   $p$-adic numbers   for a prime~$p$,  is a t.d.l.c.\  group    in neither of the two classes above.
   
\n (iii) The    semidirect product $\ZZ \ltimes  \QQ_p$  corresponding to the  automorphism  $x \mapsto px$ on $\QQ_p$ is t.d.l.c.  
 
 \n  (iv) Algebraic groups over local fields, such as $\SL_n(\QQ_p)$ for $n \ge 2$, are t.d.l.c.

\n (v) Given a connected countable undirected graph such that each vertex has finite degree, its   automorphism group is t.d.l.c. The stabiliser of any   vertex forms a compact open subgroup. 
By an  \emph{undirected  tree} we mean  a  connected graph without  cycles.   For $d \ge 3$, by   $T_d$   one denotes the undirected  tree  where each vertex has degree $d$. The   group $\mathrm{Aut}(T_d)$ was    first studied by Tits~\cite{Tits:70}.

Towards Question~\ref{qu:quest}(b), we  will  review some    objects  that are associated with a locally compact group $G$.  

\medskip
\n  \emph{Modular function.} The left and right Haar measures on $G$ are treated in standard textbooks such as~\cite{Hewitt.Ross:12}.  Recall that for any left  Haar measure $\mu $ on $G$ and any $g \in G$, one obtains a further  left  Haar measure $\mu_g$ by  defining $\mu_g(A) = \mu(Ag)$. By the  uniqueness up to a multiplicative constant of   the left Haar measure, there is a real $\Delta(g)>0 $ such that $\mu_g(A) = \Delta(g) \mu(A)$ for each measurable $A$.  The function $\Delta \colon G \to \RR^+$,  called the  {modular function} for $G$,   is a group homomorphism.

 \medskip

\n  \emph{Scale function}. Willis~\cite{Willis:94} introduced the   scale function $s \colon G \to \NN^+$.  Let $g\in G$. For a compact open subgroup~$V$ of $G$, let $m(g,V) = |gVg^{-1}\colon V \cap gVg^{-1}| $. Let  $s(g)$  be  the minimum value of $m(g,V)$ over all $V$. It is not hard to show that $\Delta(g) = s(g)/s(g^{-1})$.      If a group has  a  \emph{normal} compact open subgroup,  as for  Examples (i)--(iii) above,   the scale function is constant of value 1. The group $\ZZ \ltimes \QQ_p $  for a prime $p$ is  among  the simplest examples of a t.d.l.c.\ group with a nontrivial scale function:     $s(t)=p$ for the   generator    $t$    of $\ZZ$ such that $ t \aaa t^{-1} =  \aaa/p$ for  each $\aaa \in \QQ_p$.

  \medskip
  
\n \emph{Cayley-Abels graphs.} We first  discuss a guiding principle in the study of a t.d.l.c.\ group $G$:  it has a  topological aspect that captures the small-scale (or local) behaviour, and a geometric aspect that captures the large-scale (or global) behaviour. The small-scale aspect  is given by a compact open subgroup $U$ (provided by van Dantzig's theorem). Note that any two such subgroups   are commensurable.  
  
For the large-scale aspect,   recall that  in the discrete setting,  Gromov and others  initiated a geometric theory of   finitely generated groups via their  Cayley graphs.  Given two generating sets, the corresponding  graphs are  quasi-isometric: there is a map $\phi $ from one vertex set to the other that only distorts distances affinely, and for some constant $c$, each vertex has  at most   distance $c$ from the range of $\phi$.
  For a  t.d.l.c.\ group $G$,  the generalisation of being f.g.\   is  being  algebraically generated by a compact subset.    Given     such a  group $G$ and a compact open subgroup $U$, note that   $G$ is algebraically generated by $U \cup S$ for some finite   set $S$ closed under inversion.  The Cayley-Abels graph $\Gamma_{U,S}$ (see, e.g., Kr\"on and M\"oller  \cite[Definition~3]{Kron.Moller:08})   is a graph on  the left cosets of $U$ generalising the Cayley-graph for discrete $G$ (where one can take $U= \{ e\}$). As in the discrete  case, any two such graphs  for the same group are quasi-isometric.
   So, one can think of any  Cayley-Abels graph as capturing  the large-scale aspect of a compactly generated group.

  Wesolek \cite{Wesolek:18}  provides   further  background and references on the  topics discussed above.   
 Also see   the record of the   2022 \href{https://zerodimensional.group/conferences/catdlc/}{conference on computational aspects
of
t.d.l.c.\ groups} in Newcastle, Australia. 

The following will be needed in several places.
\begin{remark} \label{rem: open mapping} The open mapping theorem for Hausdorff  groups says   that every surjective continuous homomorphism from a $\sss$-compact   group $G$    onto a Baire group $H$ is   open. This applies in particular when $G,H$ are countably based  t.d.l.c.\ groups. \end{remark}


  \subsection{Background on computable mathematics}   \label{s:compMath}
 A general goal of compu\-table mathematics is to study the algorithmic content of areas such as algebra \cite{AshKn,ErGon}, analysis  \cite{PourElRich,Wei00}, or topology~\cite{Schroeder:21}.  A first step   is invariably to define what  a computable presentation of a mathematical structure in that area is, such as a countable    group, a complete metric space, or a separable  Banach space.   (Note that ``computable" is the    commonly accepted adjective used with presentations, rather than   ``algorithmic".) One also  introduces and studies   computability for   objects related to the structure. 
For instance,
 in computable analysis one  uses rational approximations to reals in order to  define what it means for a function $f \colon \mathbb R \to \mathbb R$ to be computable (see   below for further detail). 
 %
 A large body of results addresses the algorithmic content of classical results.  
 There are also interesting new questions with no pendant in the classical setting. What is the complexity of  recognizing that two computable presentations present the same structure? How hard is it to determine whether the structure presented has a certain property?  (For instance,     determine whether a computably presented group is torsion-free.) The basic distinction is ``decidable/undecidable". Mathematical logic provides a bevy of descriptive complexity classes, with corresponding completeness notions,  for a more detailed answer in the undecidable case. For instance, torsion-freeness  is of maximum complexity within the    co-recursively enumerable properties.

 Towards defining computable presentations, we first recall the    definition of a computable function on natural numbers,  slightly adapted to our purposes in that we allow the domain not only $\NN^k$ but also   any computable subset of $\NN^k$. 
  \begin{definition} \label{def:computable} Given a   set $S \sub \NN^k$, where $k\ge 1$,  a  function $f \colon S \to \NN$ is called computable if there is a Turing machine that on inputs $n_1, \ldots, n_k $  decides whether the tuple of  inputs   $(n_1, \ldots, n_k) $ is in $S$, and if so outputs $f(n_1, \ldots, n_k)$.  We say that $S$ is computable if the function with domain $S$ and constant value $0$  is computable. \end{definition} One version of the  Church-Turing thesis    states that  computability in this sense   is the same as being computable by some algorithm.
 We note that without the restriction that the TM can decide membership in~$S$, the function is called partial computable. The domain of   a partial computable function is a recursively enumerable relation.

A structure in the  model theoretic sense consists of  a nonempty      set $D$, called the domain,  with relations and functions defined on it.  The following definition was first formulated in the 1960s by Mal'cev~\cite{Ma61} and Rabin~\cite{abR},  independently.
\begin{definition} \label{compStr}  A \emph{computable   structure}  is a structure such that the domain is a 
 computable  set   $D\sub  \NN$,  and the    functions and relations of the structure are computable.
A  countable structure $S$ is called \emph{computably presentable} if some  computable structure $W$ is  isomorphic to  it. In this context we call $W$ a \emph{computable copy} of $S$.  \end{definition} 
Next, we discuss how to define  that an  uncountable structure has a computable presentation.
 In the field of computable analysis,    one    represents all the elements of the structure by ``names" which are directly accessible to computation, and requires that the functions and relations are computable on the names. For   detail  see e.g.\ Pauly~\cite{Pauly:16} or Schr\"oder~\cite{Schroeder:21}.  Names usually are elements of  the set $[T]$ of  infinite paths on some computable subtree $T$ of $\NN^*$, the tree of strings with natural number entries. (Henceforth all paths will be infinite paths on  rooted trees, starting at the root.)     For instance, a  standard name of a real number  $r$ is  a path coding a      sequence of rationals $\seq{q_n} \sN n$ such that $|q_n - q_{n+1}| \le \tp{-n}$ and $\lim_n q_n = r$. Using  so-called ``oracle Turing machines",  one can define computability of functions on $[T]$; we will provide detail in  Section~\ref{s:comp notions}.  This indirectly defines computability on spaces relevant to computable analysis; for instance,  one can define that   a function on $\RR$ is computable.
  The example above shows that names and the object they denote can be of quite a different kind. In contrast, each  totally disconnected Polish space is homeomorphic to $[T]$ for some subtree $T$ of $\NN^*$, which  is advantageous because in principle there is no need to distinguish between names and objects  in our setting.

 An  \emph{ad hoc}  way to define computability  often  works  for  particular classes  of uncountable structures:     impose   algorithmic constraints on   the definition of the class. For instance, 
   the following is due to  Smith~\cite{Smith:81} and La Roche~\cite{LaRoche:81}. 
  
  \begin{definition} \label{df:profgroups}  A profinite group $G$ is computable if  $G= \varprojlim_i (A_i, \psi_i)$ for a computable diagram  $ (A_i, \psi_i)\sN i$ of  finite groups and epimorphisms~$\psi_i \colon A_i \to A_{i-1}$ ($i>0$).   \end{definition}
The aforementioned approach to profinite groups of  Smith and la Roche  admits  equivalent formulations 
that work beyond this class. One such reformulation, in terms of effectively branching subtrees of $\NN^*$, already is in  Smith~\cite{Smith:81}.  For the other, in terms of computably compact metric spaces, see Thm.\ 4.33 of the  recent work~\cite{EffedSurvey}. Results of this sort indicate that the respective notion of computable presentability is  {robust} for  the class.


 The remainder of the introduction      discusses the questions posed at   the beginning of the paper in more detail.

\subsection{Computable presentations of t.d.l.c.\ groups}   \label{s: types of presentations}
We aim at a   robust definition of the class of (countably based)   t.d.l.c.\ groups with  a computable presentation. We want this class to have  good algorithmic closure properties, and also ask that our definition extend the existing definitions for  discrete, and for profinite groups.  We  provide three types of computable presentations: type $S$ and  the more general type $B$ are based on computation with approximations, while  the separate type $M$ is based on a reduction to  countable structures via a duality. They  will all  turn out to be equivalent: a t.d.l.c.\ group has a computable presentation of one type iff it has one of the other type. The equivalences are algorithmically uniform. 

\emph{Type S} (for ``symmetric group")   is based on the fact that each t.d.l.c.\ group $G$ is isomorphic to a closed subgroup of $\S$. We represent such subgroups by subtrees of~$\NN^*$ in a particular  way to be described below. We impose algorithmic conditions on the tree to define when the presentation is computable. This approach is consistent with 
  earlier work~\cite{Greenberg.Melnikov.etal:18} on computable subgroups of $S_\infty$.

\emph{Type M} (for ``meet groupoid") is based on  an  algebraic structure  $\+ W(G)$ on the countable set of compact open cosets in $G$. (Note here that every left coset $A= gU$ of a subgroup $U$  is a right coset of the conjugate subgroup $V= gUg^{-1}$, namely, $A= Vg$.)  This  structure is   a  partially ordered groupoid, with the usual set inclusion and multiplication of a left coset of a subgroup $U$ with a right coset of    $U$.  The intersection of two compact open cosets is such a coset itself,  unless it is empty. So, after  adjoining  $\ES$ as a least element  (which only  interacts trivially with the groupoid structure), we obtain a meet semilattice.   A Type M computable presentation of    $G$ is a computable copy of the meet groupoid of $G$ such that the index function on compact open subgroups, namely $U,V \mapsto | U \colon U \cap V|$,  is also computable. 
  The idea to study appropriate Polish groups via an algebraic structure on their open cosets     appeared in~\cite{Kechris.Nies.etal:18}, and  was further elaborated in a paper  on the complexity of the  isomorphism problem for oligomorphic groups~\cite{Nies.Schlicht.etal:21}. There,       approximation     structures  called ``coarse groups" are used that are given by   the       ternary relation  ``$AB \sub C$", where $A,B,C$ are certain open cosets.   
In the present work, it will be important that we  have explicit access to  the combination of the groupoid and the meet semilattice structures (which coarse groups fail to  provide).  As a further example of the usefulness of meet groupoids,  in \cref{prop: Braconnier} we will show that the meet groupoid of  $G$  can be used to understand  the topological  group of  continuous automorphisms of $G$.

\emph{Type B}  (for ``Baire") of computable presentation   generalizes Type S. For   computable Baire presentations,  one asks that the domain of $G$ is what we call a computably  locally compact subtree of $\NN^*$ (the tree of strings with natural number entries), and the operations are computable in the sense of oracle Turing machines.   This    generalizes the approach in \cite{Smith:81} 
from profinite groups to   t.d.l.c.~groups. However,  it also works   for t.d.l.c.\ structures other than groups. We postpone this approach    until Section~\ref{s:Baire},   because    it  requires  more advanced notions from computability theory. In particular, it relies on    computable functions on the set of paths of a computable tree. These notions will be provided in   Section~\ref{s:comp notions}.  We note that Block and Miller~\cite[Def.\ 1.1]{BlockMiller2025} have recently studied representations of general $0$-dimensional  functional structures via path spaces of trees.

Among our    main results is that  the various approaches  to computable presentations  are equivalent.
This  will be stated formally   in Theorem~\ref{thm:main} and its extension Theorem~\ref{thm:main2}. 
Indeed,   the approaches are equivalent in the strong sense that from a presentation of  one type, one can effectively obtain a presentation of the other  type for the same t.d.l.c.\ group.  Below, we will initially say that a t.d.l.c.\ group is computably t.d.l.c.\ of  a particular type, for instance via a closed subgroup of $\S$, or via a Baire presentation. Once the equivalences have been established, we will often omit this. 
These results    suggest that 
our approach to computability for t.d.l.c.\ groups is  {natural} and  {robust}. We will later support this thesis with examples that  show that many widely studied t.d.l.c.\ groups are computably t.d.l.c. Evidence for the  robustness is also given in~\cite[Thm.\ 1.2]{Melnikov.Ng:23}, where it is  shown that the  notion of computable presentability for general locally compact Polish groups in their Definition 2.16 restricts to our notion in the totally disconnected case.

Baire presentations  appear to be the simplest and most elegant notion of computable presentation for general totally disconnected Polish groups.  However,     computable Baire presentations are hard to study because the domain is usually uncountable (while the  meet groupoids are countable), and there are no specific combinatorial  tools available (unlike the case of  permutations of~$\NN$).  In the proofs of   several results, notably \cref{thm:closure normal},  we will work around this by replacing    a Baire presentation by   a more accessible presentation of Type M or S.

  The   operation    that leads from a t.d.l.c.\ group to its meet groupoid has an inverse operation. Both operations are functorial for the categories with isomorphisms. This yields   a    duality between t.d.l.c.\ groups and a  certain  class of countable meet groupoids  (similar to the duality  in \cite{Nies.Schlicht.etal:21} between oligomorphic groups and the ``coarse groups").  This class  of meet groupoids  can be described axiomatically.  The equivalence of computable presentations of Type B and Type M can be extended to  a computable version of that duality. We will elaborate on this in \cref{rem:duality}. 

 \subsection{Which t.d.l.c.\ groups $G$ have computable presentations?} Dis\-crete groups,  as well as   profinite groups, have  a computable presentation as   t.d.l.c.\ groups if and only if they have one in the previously established senses of \cref{def:computable} and \cref{df:profgroups}, respectively.   
 For discrete groups this will be shown in \cref{ex:discrete computable}; for profinite groups it {will} be obtained  by combining  \cref{ex:prof Baire} and \cref{lem:proc}. 
We     provide several examples of computable presentations for     t.d.l.c.\ groups  outside these two classes. The various  equivalent approaches to computable presentations will be useful for this,  because they allow us  to construct  a   presentation of  the  type most appropriate for a given group. For a group of automorphisms such as  $\Aut(T_d)$, we   use Type S (presentations as closed subgroups of $\S$). For $(\QQ_p, + )$,  we    use   Type M (meet groupoids). For $\SL_n(\QQ_p)$,   we use Type B (computable  Baire presentations).   The first author has shown how to give
an equivalent definition of a computable t.d.l.c.~group in terms of an `effectively $\sigma$-compact metric'. Using this one   can  generalize the latter example to other algebraic groups over~$\QQ_p$. For an outline  of this  work,  see \cite[Section~4]{LogicBlog:21}.
  Neretin's groups~$\+ N_d$ of almost automorphisms of $T_d$, for $d \ge 3$  \cite{Kapoudjian:99},  are computable by upcoming work of Ferov, Skipper and Willis.

 %


\subsection{Associated  computable objects}  Recall that    to a t.d.l.c.\ group $G$ we associate  its meet groupoid $\+ W(G)$, an algebraic structure on its compact open cosets. If $G$ is given by a computable Baire presentation, then we   construct a copy $\+ W = \+ W_{\text{comp}}(G)$ that is computable in a strong sense, essentially including the condition  that some  (and hence any) rational valued  Haar measure on $G$ is computable when restricted to   a function   $\+ W \to \RR$.
We will   show in \cref{prop: comp isom} that   the left,  and hence also the right,  action of $G$ on $\+ W$ is  computable. 
We conclude  that the  modular function on $G$  is computable.

 If $G$ is compactly generated,  for each  Cayley-Abels graph  one can determine a  computable copy,  and any two     copies of this type  are computably quasi-isometric (\cref{prop:CA graph}).  Intuitively, this means that the large-scale structure of $G$ is a computable invariant.

Assertions that the scale function  is computable have been made   for particular t.d.l.c.\ groups in works such as Gl\"ockner~\cite{Glockner:98} and Willis~\cite[Section~6]{Willis:01}; see the   survey~\cite{Willis:17}.
 In these particular cases, it was generally clear what it  means that one can compute the scale $s(g)$:  provide an algorithm that  shows it.  One has to declare what  kind of  input    the algorithm takes; necessarily it has to  be some  approximation to $g$, as $g$   ranges over a potentially  uncountable domain. Our new framework allows us to give a precise meaning to the question whether  the scale function is computable for a particular computable presentation of  a     t.d.l.c.\ group, thus also allowing  for a   negative answer.    In \cref{th:scale noncomp}, which is joint with George Willis, we provide a computable presentation of    a t.d.l.c.\ group $G$ such that the  scale   function fails to be  computable for this presentation.
 
%

%

It remains open whether  for some computably presented  t.d.l.c.\ group~$G$,  the scale is non-computable for \emph{each} of its computable presentations. An even stronger negative result  would be that such a  $G$ can  be chosen to have a unique computable presentation  (see the discussion in \cref{ss:auto} below).


\subsection{Algorithmic versions of constructions that lead from given to new t.d.l.c.\ groups}
Section~\ref{s:closure} 
  shows that for many constructions that have been studied in the theory of t.d.l.c.\ groups,  the class of computably t.d.l.c.\ groups   is closed under suitable algorithmic versions. In particular,    the constructions (1), (2), (3) and (6) described  in  Wesolek~\cite[Thm.\ 1.3]{Wesolek:15}    can be phrased algorithmically  in such a way  that they stay within the class of  computably t.d.l.c.\ groups; this provides   further evidence that our class is robust. These  constructions  are suitable versions, in our algorithmic topological setting, of passing to closed subgroups, taking group extensions via continuous actions,  forming ``local" direct products, and taking quotients by closed normal subgroups (see~\cite[Section~2]{Wesolek:15}  for   detail on these constructions). 
  The algorithmic  version of taking quotients (\cref{thm:closure normal}) is the most demanding; it uses  extra insights, provided in \cref{prop: comp isom},  from   the proofs that the various forms of computable presentation are equivalent. 
  
  Several constructions lead to new examples of t.d.l.c.\ groups with computable presentations. For instance,  after   defining  a computable presentation of  $\SL_{n+1}(\QQ_p)$ ($n \ge 2$, $p$ a prime) directly in \cref{prop: SL2}, we proceed to a computable presentation of $\GL_n(\QQ_p)$ via taking a closed subgroup, and then to $\PGL_n(\QQ_p)$ via taking the  quotient by the centre.

  \subsection{When is a computable presentation  unique?} \label{ss:auto}
Citing  Willis, \cite[Section~5]{Willis:17},  ``it is a truism that computation in a group depends on the description of the group".  
In the present article, we apply our notion of computable presentability of a t.d.l.c.~group to give a formal version of  this statement. We also give some examples  where the general  statement   fails. 
  Viewing a   computable Baire presentation as a description, we are interested in the question whether such a description is unique, in the sense that between any two of them there is a computable isomorphism. 
Adapting terminology for countable structures going back to  Mal'cev (see below),
   we will call such  a  group  \emph{autostable}.     If a t.d.l.c.\ group is  autostable,  then computation in the group can be seen as independent of its particular description.    

      \cref{thm:compCrit} reduces the problem of whether a t.d.l.c.\ group is autostable to the   similar problem in the  countable   setting of meet groupoids.   We apply it to show that     $(\QQ_p, +)$   is autostable, and so is $\ZZ \ltimes \QQ_p $. 
  Proving the autostability of these groups requires more effort than the reader would perhaps expect. For other groups,  such as $\SL_n(\QQ_p)$ for $n \ge 2$ and $\Aut(T_d)$, we leave open whether a computable   presentation is unique up to  computable isomorphism. 
 
\subsection{Some context}   \label{s:context} \

\n {\emph{Related work on autostability.}  A   countable structure is called {autostable} (or  {computably categorical})~\cite{ErGon,AshKn} if it  has  a computable copy, and such a copy is unique   up to computable isomorphism. 
 For example, any computable finitely generated  algebraic structure is autostable because any homomorphism $A \to B$ for such structures $A,B$  can be reconstructed from the images of generators of $A$, and is therefore computable.
In contrast, there is  a discrete 2-step nilpotent group with   {exactly two} computable presentations up to computable isomorphism~\cite{GonGon}.  For t.d.l.c.\ groups, in the discrete case our notion of autostability   reduces to the  established one.

A profinite \emph{abelian} group is autostable if and only if  its    Pontryagin dual is autostable (\cite{Pontr}; also see \cite[Thm.\ 9.5.7]{Downey.Melnikov:book}). Note that this dual is a discrete, torsion abelian group. Autostability of the latter type of groups is characterized in \cite{tor}. In this way one obtains   a characterization of autostability for profinite abelian groups.


   Pour El and Richards~\cite{PourElRich} gave  an example of a   Banach space with two computable presentations so that there is no  computable linear isometry between them.
Works such as~\cite{MelIso} and then \cite{Clanin.McNicholl.Stull.2019,Ilja,McNellp}   systematically study  autostability in separable spaces,  using tools of computable (discrete) algebra. 

\medskip

 \n {\emph{Other related work.}
  Our   work with Lupini~\cite{Lupini.etal:21} focusses on abelian locally compact groups. We introduce two notions of computable presentation for abelian t.d.l.c.\ groups that take into account   their specific structural properties: The first  is based on the fact that such groups are pro-countable, the other  on the fact that such a  group is an extension of a discrete group by a profinite group.  Both notions can be used to provide     examples of computable abelian t.d.l.c.~groups that are neither discrete nor compact.    The same work~\cite{Lupini.etal:21}    states that in the abelian  case,   the  notion of computable presentability given in the present paper is equivalent to these  notions, referring to the present paper. We prove  this in  Appendix~2.

 An approach to computability for Polish groups   was suggested in Melnikov and Montalb\'an~\cite{MeMo} and then developed in, e.g., Melnikov~\cite{Pontr}, Pauly, Seon and Ziegler~\cite{ArnoHaar} and the aforementioned Melnikov and Ng~\cite{Melnikov.Ng:23}.    In that approach, a Polish group is said to be computable if the underlying topological  space  is computably, completely metrized and the group operations are computable operators (functionals) on this space. By~\cite[Cor.\ 1.6]{Pontr}
   there is  computably metrized profinite group that does not possess a computable presentation   in the sense of~\cref{df:profgroups}.   However,   if we additionally assume that the underlying computably metrized space is `computably compact' (equivalently, the Haar measure is computable~\cite{ArnoHaar}),  then 
we can produce a computable presentation of the group in that sense; see  Downey and Melnikov~\cite{EffedSurvey} for a proof.

Nies and Schlicht~\cite[Section~4]{Nies.Schlicht:nd} study a Borel duality of classes of groups closed subgroups of~$\S$ with meet groupoids   in an axiomatic setting, which encompasses the class    of locally Roelcke precompact groups, much larger than the class of t.d.l.c.\ groups. It would be interesting to probe the computable content of this theory. The present authors~\cite[Section~6]{Melnikov.Nies:25} study the Chabauty space of a computably t.d.l.c.\ group~$G$  via its meet groupoid~$\+ W(G)$; in particular, they characterise the computably closed subgroups of $G$ via certain computable ideals of $\+ W(G)$. Ferov, Tornier and Willis~\cite{FerovTornierWillis2025} survey  algorithmic aspects of particular  t.d.l.c.\ groups. They  connect to the present paper, and also   discuss  practical implementations on computer algebra systems.

  \section{Computably locally compact subtrees of $\NN^*$}
\cref{def:comploccompact} of this section   introduces computably locally compact trees. This purely   computability-theoretic concept  will matter for the whole     paper.  For basics on computability theory see, e.g., the first two  chapters of~\cite{Soare:87}, or the first chapter of~\cite{Nies:book} which also contains notation on strings and trees. Our paper is mostly   consistent with the terminology of these two sources. They  also serve for basic concepts such as Turing programs, computable functions, as well as partial computable (or partial recursive) functions, which will be  needed from  Section~\ref{s:comp notions} onwards. In this section we  will   review   some    more specialized  concepts   related to computability.   

\begin{notation} {\rm Let $\NN^*$ denote the   set of strings with natural numbers as entries. We use letters $\sss, \tau, \rho$ etc.\ for elements of $\NN^*$.  The  set  $\NN^*$ can be seen as a directed tree: the empty string is the root, and the successor relation is given by appending a number at the end of a string.     One writes $\sss \preceq \tau$ to denote that $\sss $ is an initial segment of $\tau$, and $\sss \prec \tau$ to denote that $\sss$ is a proper initial segment. For $k \le |\tau|$, by $\tau \uhr k$ one denotes the initial segment of $\tau$ that has length $k$. We can also identify finite strings of length $n$+$1$ with partial functions $\NN \rightarrow \NN$ having finite support $\{0, \ldots, n\}$.
We then write $\tau_i$ instead of $\tau(i)$. By $\max (\tau)$ we denote  $\max \{\tau_i \colon \, i \le n\}$. 
Let  $h \colon \NN^* \to \NN$ be the canonical encoding given by $h(w)= \prod_{i< |w|} p_i^{w_i+1}$, where $w = (w_0, \ldots, w_{|w|-1})$, and  $p_i$ is the $i$-th prime number. By convention $h()=1$.}  \end{notation}

\begin{definition}[Strong indices for finite sets of strings]  For   a finite set $u \sub \NN^*$ let   $n_u=\sum_{\eta \in u} \tp {h(\eta) }$;  one says that $n_u$ is the  \emph{strong index} for $u$.     \end{definition}
  We will usually  identify a finite subset of $  \NN^* $ with its strong index.  
\begin{remark}
Why ``strong index"? In computability theory, by an index  one usually means a (code for a) Turing program that decides  a set or computes a function. Such an index can be obtained from a strong index as defined above. However, a  strong index   tells us  more about the finite set, such as its size.  This is not true for an index as a  Turing program. 
\end{remark}
Unless otherwise mentioned, by a (rooted) tree  we mean a nonempty subset~$T$ of $\NN^*$ such that $(\sss \in T \wedge \rho \prec \sss)  \to \rho \in T$.   For $k \in \NN$ we write \bc $T^{[\ge k]}= \{ \sss\in T \colon \, \sssl \ge k\}$. \ec
By  $[T]  $ one denotes  the set of  (infinite) paths of a tree $T$ starting at the root.  Our trees usually have no leaves, so $[T]$ is a closed set in Baire space $\NN^\NN$ equipped  with the usual product topology. Note that    $[T]$ is compact if and only if  each level of $T$ is finite, in other words, iff  $T$ is finitely branching. 
For $\sss \in T$, let    \bc $[\sss]_T = \{X \in [T] \colon \sss \prec X \}$.  \ec That is,  $ [\sss]_T $ is the cone  of  paths on $T$ that extend  $\sss$.  One  says that $T$ is computable if the set $\{h(\sss) \colon \, \sss \in T \}$ is computable.

%

\begin{definition}[computably  locally compact  trees] \label{def:comploccompact} Let $T$  be a computable subtree of~$ \NN^*$ without leaves.  We say that  $T$ is  \emph{computably  locally compact  (c.l.c.)} if  for some  $k \in \NN$    there is a computable     function $H \colon \NN^k \times \NN \to \NN $ such that \bc $\rho(i) \le H(\rho\uhr k, i)$ for each  $\rho \in T$  of length $>k$  and  each $i < |\rho|$.  \ec     In particular, if   $\sss \in T^{[\ge k]}$, then $ \{ i  \colon \sss\ape i \in T\}$ is finite. 
   \end{definition}

Frequently  we  will have $k=1$, so that only the root can have infinitely many successors. In this case the condition says that  there is a computable function  $H \colon \NN \times \NN \to \NN$ such that $\rho(i) \le H(\rho(0), i)$ for each positive $i< |\rho|$.

%

Given  a c.l.c.\ tree $T$, the compact open subsets of $[T]$ can be algorithmically encoded by natural numbers; the notation below     will be used throughout.
\begin{definition}[Code numbers for compact open sets] \label{defn: str index compact}  Suppose that a tree $T$   is  c.l.c.\   via~$k\in \NN$. For  a finite set $u\sub T^{\ge k}$,  let \bc $\+ K_u= \bigcup_{\eta\in u} [\eta]_T$.   \ec    By a  \emph{code number} for a compact open set $\+ K\sub [T]$ we mean the strong index for a finite set $u$ of strings such that  $\+ K = \+ K_u$.  \end{definition} 
Clearly, each  
    compact open subset  $\+ K$ of $[T]$ is of the   form $\+ K_u$ for some $u$. 
 Such a code number  is not unique (unless $\+ K$ is empty).  So we will need to  distinguish between the actual compact open set, and any of its code numbers.  
Note that     one  can   decide, given    $u \in 	 \NN $ as an input,  whether $u$ is a  code number.

     The following lemma shows that the basic set-theoretic relations and operations are decidable for sets of the form $\+ K_u$, similar to the case of finite subsets of $\NN$.
    \begin{lemma} \label{lem: comp index tree}  Suppose a tree $T$ is  c.l.c.\    via $k \in \NN$.  Given  code numbers  $u, w$,
    
 \bi \item[(i)] one can compute  code numbers for $\+ K_u \cup \+ K_w$ and  $\+ K_u \cap \+ K_w$;
    
   \item[(ii)] one can decide whether $\+ K_u \sub \+ K_w$.  So  one can,  given  a  code number  $u\in \NN$,     compute the minimal  code number $u^*\in \NN$ such that $\+ K_{u^*}=\+ K_u$. 
   \ei \end{lemma}
  \begin{proof}  
  \n (i) The case of  union   is trivial. For the   intersection operation,  it suffices to consider the case that $u$ and $w$ are singletons. For strings $\aaa, \beta \in T$, one has  $[\aaa]_T \cap  [\beta]_T = \ES$ if $\aaa ,  \beta$ are incompatible, and otherwise $[\aaa]_T \cap  [\beta]_T = [\gamma]_T$ where $\gamma $ is the longest common initial segment of $\aaa, \beta$. 
  
\n  (ii)   Let~$H$ be a computable binary function as in   \cref{def:comploccompact}.  It suffices to consider the case that  $u$ is a singleton. Suppose that  $\aaa \in T^{[\ge k]}$.
 The algorithm to  decide whether  $[\aaa]_T  \sub \+ K_w$ is as follows. Let $N$ be the maximum length of a string in $w$. Answer ``yes" if       for each $\beta \succeq \aaa$ of length $N$ such that $\beta(i) \le H(\aaa\uhr k, k)$  for each $i < N$, there is $\gamma \in w$ such that $\gamma \preceq \beta$. Otherwise, answer~``no".
\end{proof}

\begin{definition} \label{def:E} Given a c.l.c.\  tree $T$, let $E_T$ denote the set of \emph{minimal}  code numbers for compact open subsets of $[T]$.    By the foregoing lemma, $E_T$ is decidable. \end{definition}


\section{Defining computable t.d.l.c.\  groups via  closed    subgroups of $\S$} 
   This section, in particular Definition\ \ref{Def1},  spells out  Type S of computable presentations of t.d.l.c.\ groups,   informally described  in Section~\ref{s: types of presentations}.  
 \subsection{Computable  closed subgroups of $\S$} \label{ss:Sinf}
      We first provide   a  computable presentation of the group of permutations of $\NN$. It is  related  to the  computable presentation   in \cite[Def 1.2]{Greenberg.Melnikov.etal:18}, where the group is viewed  as a topological group with a computable compatible metric. However, 
 in our presentation, an element  of the group  is given as a   path on a tree that encodes a  pair $(h, h^{-1})$ where $h$ is a permutation of $\NN$. This enables us to define a  computable tree, denoted  $\mathit{Tree}(\S) $,  each path of which corresponds to a  permutation of $\NN$.

Suppose   strings $\sss_0, \sss_1 \in \NN^*$   both have  length $N$. By    $\sss_0 \oplus \sss_1$ we denote  the string   of length $2N$ 
that  alternates between $\sss_0$ and $\sss_1$. That is, \bc $(\sss_0 \oplus \sss_1)(2i+b) =\sss_b(i)$ for $i<N, b =0,1$.  \ec 
 Similarly, for  functions $f_0, f_1$ on $\NN$, we define a function  $f_0 \oplus f_1$ on $\NN$ by \bc $(f_0 \oplus f_1)(2i+b)= f_b(i)$. \ec
Informally, $\mathit{Tree}(\S)$ is the tree of strings such  that it is consistent that the map given by the entries at odd positions extends to  an  inverse of the map given by entries at even positions. 
Formally,  let  $\mathit{Tree}(\S) =$
 
   \bc  $ \{ \sss \oplus \tau \colon
       \sss, \tau \,  \text{are 1-1} \lland   \sss (\tau(k)) = k \lland \tau(\sss(i))= i \text{ whenever defined}\}$. \ec 

We view  $\S$ as a group   on the set of paths of $\mathit{Tree}(\S)$. Its  domain is  the set of   functions of the form $ h\oplus  h^{-1}$ where $h$ is a permutation of $\NN$;  if  $f= f_0 \oplus f_1$ and $g= g_0 \oplus g_1$  in $ \S$, we define $f^{-1} = f_1 \oplus f_0$ and 
   $g  f    = (g_0 \circ f_0)    \oplus (f_1 \circ g_1  )$.  We will verify in Fact~\ref{ex:S} below  that these  group operations are computable (in the sense of Definition~\ref{def:Comp fcn on tree}).
\begin{definition} We  say that a closed subgroup $C$ of $\S$ is \emph{computable} if its corresponding   tree, namely $\mathit{Tree}(C)=\{\eta  \in \mathit{Tree}(\S) \colon [\eta]_T \cap C \neq \ES \}$ is computable. 
\end{definition}  
\begin{remark}  It is well known  that the closed subgroups of $\S$ are precisely the automorphism groups  of structures $M$ with domain $\NN$. Suppose that $M$ is a computable structure, and  there is an algorithm to  decide whether a bijection between finite subsets of    $M$  (encoded by a strong index) can be extended to an automorphism.  Then  the automorphism group of $M$  is  computable.  To see this, one  uses that  a string $\eta=\sss \oplus \tau$ on $\mathit{Tree}(\S)$ determines the  finite injective map
 \begin{equation} \label{eqn:inj} \aaa_\eta=\{ \la i,k \ra \colon \, \sss(i)= k \lor \tau(k)= i\}\end{equation} 
 between finite subsets of $M$. This map  is extendible to an automorphism of $M$ if and only if  $\eta\in  \mathit{Tree}(C)$. 

 For instance, assuming  a computable bijection between $\QQ$ and $\NN$, the group $ C=\text{Aut}(\QQ, <)$ is computable:  By Cantor's back and forth argument,   a bijection between  finite  subsets of $\QQ$ can be extended to an automorphism of $C$   if and only if   it preserves the ordering. There is an algorithm to decide the latter condition.   \end{remark}

  \subsection{First definition of   computably t.d.l.c.\  groups}  
  
  
  If a   subgroup $C$ of $\S$ is locally compact, then there is an  $n \in \NN$ such that $C\cap V_n$ is compact, where $V_n$ is the group of elements of $\S$  fixing $0, \ldots, n-1$. Thus all the 1-orbits for the natural action of  $C\cap V_n$ on $\NN$ are finite, which implies that in $\Tree C$ each string of length at least $2n$ has only finitely many successors. This motivates the following definition.
  \begin{definition}
  \label{Def1} Let $G$ be a t.d.l.c.\ group. We say that $G$ is  \emph{computably t.d.l.c.} (via a closed subgroup of $\S$) if  there is a closed subgroup $C$ of $\S$ such that $G \cong C$, and      the   tree \bc $\mathit{Tree}(C)=\{\eta  \in  \mathit{Tree}({\S}) \colon [\eta] \cap C \neq \ES \}   $ \ec is c.l.c.\  in the sense of Definition\ \ref{def:comploccompact}.
\end{definition}
   In this context we will often  ignore the difference between $G$ and $C$. That is,  we  will assume that  $G$ itself is a closed subgroup of $\S$.


 Recall from the introduction  that  $\mathrm{Aut}(T_d)$ is the group of automorphism of the undirected tree   $T_d$  where each vertex has degree $d$.
 \begin{example}  \label{ex:Td}Let $d\ge 3$. The t.d.l.c.\  group $G=\text{Aut}(T_d)$ is computably t.d.l.c.\ via a closed subgroup of $\S$.   \end{example}

   \begin{proof} Via  an effective encoding  of the vertices of $T_d$ by the natural numbers, we can view $G$ itself as a closed subgroup of $\S$. A finite injection $\aaa$ on $T_d$ can be extended to an automorphism of $T_d$ iff   it preserves   distances, which is a decidable condition. Each $\eta \in\mathit{Tree}(\S)$ corresponds to   an injection on $T_d$ via (\ref{eqn:inj}). 
  So we can decide whether $[\eta]_{\Tree G } \neq \ES$.  

We show that $\mathit{Tree}(G)$ is c.l.c.\ via $k=1$. Note that   if $\sss\in \mathit{Tree}(G)$ maps $x\in T_d$ to $y\in T_d$, then every extension $\eta \in \mathit{Tree}(G) $ of $\sss$ maps   elements  in $T_d$ at distance~$n$ from $x$ to elements  in $T_d$ at distance $n$ from $y$, and conversely.  
 This yields a computable bound $H$ as required in   Definition~\ref{def:comploccompact}. \end{proof}

The following lemma shows  that, given  a    group $G$ as in \cref{Def1},  the group operations are algorithmic when applied to  its  compact open subsets.   %
  \begin{lemma} \label{lem:basic operations}     Suppose $G$ is  computably t.d.l.c.\  via  a closed subgroup of $\S$ (identified with $G$).   
     Write $T = \Tree G$, which by hypothesis is c.l.c.\ via some $k \in \NN^+$. Recall from  Definitions  \ref{defn: str index compact} and~\ref{def:E}  that  $\+ K_u$ denotes the   open subset of $[T]$  with code number $u$, where $u$ is a strong index for a finite subset of $T^{[\ge k]}$;    $ E_T\sub \NN$ denotes the computable set of        minimal code numbers for compact open subsets of~$[T]$.  %
\bi \item[(i)]   There is  a computable   function  $ I \colon E_T \to E_T$     such that  $\+ K_{I(u)} = (\+ K_u)^{-1}$  for each $u  \in E_T$.   
\item [(ii)]    There is  a computable   function  $M \colon E_T  \times E_T \to E_T$    such that $ \+ K_{M(u,v)}= \+ K_u   \+ K_v$ for each $u,v \in E_T$.  \ei  \end{lemma}
\begin{proof}    We will use Lemma~\ref{lem: comp index tree} without special  mention.
  For (i), let $I(u) $ be the least strong index for the set $ \{ \sss_1 \oplus \sss_0 \colon\, \sss_0\oplus \sss_1 \in u\}$. 
  For (ii), first note that since $\Tree G$ is c.l.c.\ via $k$, we can computably  replace each string $\sss $ in $u$ by its set of    extensions on $\Tree G$ of a given length $N \ge \sssl$. So we may assume    that all the  strings in $u\cup v$ have the same length. Hence  it suffices to define $M(u,v)$ in case that $u = \{\sss\}$ and $ v= \{ \tau\}$  where $\sssl= |\tau| =:n$.

   For such $\sss, \tau$ let $m = 1 + \max(\sigma, \tau)$; that is, $m-1$ is the maximum number occurring in any of the two strings.  
    For strings  $\gamma, \delta \in \NN^*$ such that $|\delta|\ge 1+\max (\gamma)$, by $\delta \cdot \gamma$ we denote the string  $\seq {\delta(\gamma(i))}_{i < |\gamma|}$.  We will verify that for each $f \in G$, 
   \begin{eqnarray} \label{eqn:fff} f \in [\tau]_T   [\sss]_T & \LR &  \ex \beta \succ \tau  \ex \aaa \succ \sss [\beta, \aaa \in T \land   | \beta| = 2m \land |\aaa|= 2 \max (\beta)+2     \\ & &   \hspace{3cm}  \lland  \beta_0 \cdot    \sss_0 \prec f_0 \land    \aaa_1  \cdot  \beta_1 \prec f_1], \nonumber \end{eqnarray} 
   where $\aaa = \aaa_0 \oplus \aaa_1$, $\beta = \beta_0 \oplus \beta_1$, and $f = f_0 \oplus f_1$ as usual.
   Given this,  we let  $M(u,v)$   be the least strong index for the set of strings $(\beta_0 \cdot    \sss_0\oplus  \aaa_1  \cdot  \tau_1)$ as above, which we can compute from $u$ and $v$ by the hypothesis on $T$. This will complete the proof of (ii) of  the claim.

If the left hand side  of (\ref{eqn:fff}) holds,  then  $f = h   g$ for some $g,h\in G$ such that  $\sss \prec g$ and $\tau \prec h$. Then the right hand side holds via $\beta = h \upharpoonright {2m}$ and $\aaa= g \upharpoonright (2\max \beta +2) $.

Now suppose that the right hand side of (\ref{eqn:fff})  holds. Since $\beta \in T$,  there is $h \in G$ such that $h \succ \beta$. Let $g = h^{-1}   f$. Then $g\in G$. Since $h \succ \tau$, it suffices to show that $g \succ \sss$. Note that by definition $f = f_0 \oplus f_1$ where $f_1= (f_0)^{-1}$, and similarly   $g= g_0 \oplus g_1 $ and $ h= h_0 \oplus h_1$. We have $g_0 = h_1 \circ f_0$ and $g_1 = f_1 \circ h_0$. 

We check  that $g_0 \succ \sss_0$ as follows: for each $i< n$ we have $\beta_0(\sss_0(i))= f_0(i)$ by hypothesis. Hence $h_0(\sss_0(i))= f_0(i)$, so $\sss_0(i)= h_1(f_0(i))= g_0(i)$. 

Next,  we  check  that $g_1 \succ \sss_1$: using $ \aaa_1  \cdot  \beta_1 \prec f_1$, for each $i< n$ we have 
\bc $ g_1(i) = f_1(h_0(i)) = f_1(\tau_0(i))= \aaa_1(\beta_1(\tau_0(i)))$. \ec  Since $\beta \in T$, $\beta_0 \succ \tau_0$ and $|\beta_1| > \max (\tau_0)$,  we have $\beta_1(\tau_0(i))=i$. So the value of the rightmost term is $\aaa_1(i)$, which equals $\sss_1(i)$. 
   \end{proof}

  %
   %
  
 
\section{Defining computably t.d.l.c.\  groups via meet groupoids}
  This section provides   detail on  the second type (Type M) of computable presentations of t.d.l.c.\ groups   described in Section~\ref{s: types of presentations}.
 \subsection{The meet groupoid of a t.d.l.c.\  group}  Intuitively, the notion of a  {groupoid} generalizes the notion of a group by allowing that the binary operation is partial.  A~groupoid   is  given by a domain $\+ W$ on which     a unary operation $(.)^{-1}$ and a partial binary operation, denoted  by ``$\cdot $", are defined. These operations satisfy the following conditions:
 \bi \item[(a)] associativity in the sense that $(A \cdot B)\cdot C= 
 A \cdot (B\cdot C)$,  with either both sides or no side defined (and so the parentheses can be omitted);  \item[(b)]  $A\cdot A^{-1}$ and $A^{-1}\cdot A$ are always defined; \item[(c)] if $A\cdot B$ is defined then $A\cdot B\cdot B^{-1}=A$ and $A^{-1}\cdot A\cdot  B =B$.\ei

It follows from (c) that a groupoid satisfies the left and right cancellation laws. One says that an element $U\in \+ W$ is \emph{idempotent} if $U\cdot U =U$. Clearly this implies that $U= U \cdot U^{-1}= U^{-1} \cdot U$ and so $U= U^{-1}$ by cancellation. Conversely, by (c) every element of the form $A\cdot A^{-1}$ or $A^{-1}\cdot A$ is idempotent.
 \begin{definition} \label{def:MeetGroupoid} A \emph{meet groupoid} is a groupoid  $(\+ W, \cdot , {(.)}^{-1})$ that is also a meet semilattice  $(\+ W, \cap  ,\ES)$ of which  $\ES$ is the  least element.     Writing $A \sub B \LR A\cap B = A$ and letting the operation $\cdot$ have preference over $\cap$,
it satisfies the conditions   \bi \item[(d)]   $\ES^{-1} = \ES = \ES \cdot \ES$,    and  $\ES \cdot A$ and $A \cdot \ES$ are undefined for each $A \neq \ES$,  
\item[(e)]   if $U,V$ are idempotents such that $U,V \neq \ES$, then   $U  \cap V \neq \ES$,  \ei

 \bi \item[(f)] $A \sub B \LR A^{-1} \sub B^{-1}$, and


\item[(g)]   if  $A_i\cdot B_i$ are defined ($i= 0,1$) and $A_0 \cap A_1 \neq \ES \neq B_0 \cap B_1$, then \bc $(A_0  \cap A_1)\cdot (B_0 \cap B_1) =  A_0 \cdot  B_0 \cap A_1 \cdot B_1 $. \ec   
 \ei
 Item (g) implies  that the groupoid operations are monotonic: if  $A_i\cdot B_i$ are defined ($i= 0,1$) and $A_0 \sub A_1, B_0 \sub B_1 $, then  $A_0 \cdot B_0 \sub A_1 \cdot B_1$.
 Also,  if $U$ and $V$ are idempotent,  then so is $U \cap V$ (this can also be verified based on  (a)-(f) alone).

 Given meet groupoids $\+ W_0, \+ W_1$, a bijection $h \colon \+ W_0 \to \+ W_1$ is an \emph{isomorphism} if it preserves the three operations. 
  Given a meet groupoid $\+ W$, the letters $A,B, C$ will  range over  elements of $\+ W$, and the letters  $U,V,W$ will range  over idempotents.  \end{definition}

We use set theoretic notation for the meet semilattice because  the ordering for the motivating  examples  of meet groupoids are given by set inclusion: 
\begin{definition} Let $G$ be a t.d.l.c.\  group.  We define a  meet groupoid $\+ W(G)$.  Its  domain consists of  the compact open  cosets in $G$ (i.e., cosets of compact open subgroups of $G$), as well as the empty set.  
 We define $A\cdot B$ to be the usual product $AB$ in case that  $A= B= \ES$, or $A$ is a left  coset of a subgroup $V$ and $B$ is a right  coset of $V$; otherwise $A \cdot B$ is undefined.  \end{definition}


%
 \begin{fact}   \label{fact:WGGW} $\+ W(G)$ is a meet groupoid  with the groupoid operations of~$\cdot$ and   inversion,  and the usual intersection operation $\cap$. \end{fact}
 \begin{proof} We first note that $\+ W(G)$  is closed under the operation $\cap$:    if $C$ and $D$ are right cosets of subgroups $U$ and $V$, respectively, and if   $g \in C \cap D$, then   $C \cap D= (U \cap V)g$.

Item  (a) is clear from the definition.  For (b), if $A = gU$ then $A^{-1}= Ug^{-1}$,  so $A \cdot A^{-1}$ is defined. The case of  $A^{-1} \cdot A$  is similar. For (c), let $A$ be a left coset of $U$ and $B$ a right coset of $U$. Then $B \cdot B^{-1}= U$, so $A\cdot U =A$. The other case is symmetric. Item (d) is again clear from the definition, (e) follows because idempotents are subgroups by the discussion before \cref{def:MeetGroupoid}, and (f) is immediate. 
 
 To verify  item (g), let $A_i$ be a right coset of a subgroup $U_i$ and a left coset of subgroup $V_i$,  so that  $B_i$ is  a right coset of $V_i$ by hypothesis. Then    $A_0 \cap A_1 $ is a  right  coset of $U_0 \cap U_1$, and $B_0 \cap B_1$ a right coset of $V_0 \cap V_1$, so the left hand side is defined and a right coset of $U_0 \cap U_1$.  The left hand side   is clearly contained in the right hand side, which  is   a right coset of $U_0 \cap U_1$ as well. So the two sides must be equal.
 \end{proof}
\n    We note that  
$\+ W(G)$ satisfies the    axioms of inductive groupoids   posited in Lawson~\cite[page~109]{Lawson:98}.    (See~\cite[Section~4]{LogicBlog:20} for more on an axiomatic approach to meet groupoids.)

We will apply  the usual group theoretic terminology  to  elements   of  an abstract meet groupoid $\+ W$. If $U$ is an  idempotent of $\+ W$ we call $U$ a \emph{subgroup},   if $AU = A$ we call $A$ a \emph{left  coset} of $U$, and if $UB= B$ we call $B$ a \emph{right coset} of $U$. Based on the axioms, one can verify that if $U \sub V$ for subgroups $U,V$, then the map $A \mapsto A^{-1}$ induces  a bijection between  the left cosets and the right cosets of $U$ contained in~$V$. 
 
 	\begin{remark} \label{rem:cat}  It is well-known~\cite{Higgins:71} that one can view  groupoids as  small categories in which every morphism has an inverse.  The elements of the groupoid are the morphisms of the category. The idempotent morphisms   correspond to the  objects of the category.  One has $A\colon U \to V$ where $U= A \cdot A^{-1}$ and $V= A^{-1} \cdot A$.
Thus, in  $\+ W(G)$, 
  $A\colon U \to V$ means that     $A$ is  a right coset of $U$ and a  left coset of $V$.    \end{remark}

\subsection{Second definition of computably    t.d.l.c.\  groups}

 \begin{definition}[Haar computable meet groupoids] \label{def:comp_meet_groupoid} A meet groupoid $\+ W$ is called \emph{Haar computable}~if 
\bi \item[(a)]  its  domain is a computable subset $D$ of $\NN$; 

\item[(b)] the   groupoid and meet operations    are computable in the sense of \cref{def:computable}; in particular, the relation $\{ \la x,y \ra\colon \, x,y \in D \lland x\cdot y \text{ is defined}\}$   is computable;
 
 \item[(c)]     the partial function with domain contained in  $D \times D$ sending  a pair of subgroups $U, V\in \+ W $ to $|U:U\cap V|$ is   computable.   \ei
    \end{definition}
  Here   $|U \colon U \cap V|$  is  defined abstractly as the number of left, or equivalently right, cosets of the nonzero idempotent  $U \cap V$ contained in $U$;  we implicitly require  that this number is always finite.  Note that by  (b), the partial order induced by the meet semilattice structure of $\+ W$  is computable. Also, (b) implies that being a subgroup is decidable when viewed as a property of elements of the domain  $D$; this is used in (c).   Condition (c)   corresponds to the computable bound  $H(\sss, i)$ required in (3) of   Definition~\ref{def:comploccompact}.         For ease of reading we will say that $n \in D$ \emph{denotes} a coset $A$, rather than saying that~$n$ ``is" a coset. 
   
%
 \begin{definition}[Computably t.d.l.c.\  groups via meet groupoids] \label{Def2}Let $G$ be a   t.d.l.c.\  group. We say that $G$ is \emph{computably t.d.l.c.}\  via a meet groupoid  if    $\+ W(G)$ has a Haar computable       copy  $\+ W$.  In this context, we call $\+ W$ a computable presentation of $G$ (in the sense of meet groupoids).   \end{definition}


%
 \begin{remark} \label{rem:Haar computable}  In this setting, Condition (c) of \cref{def:comp_meet_groupoid} is equivalent to  saying that every Haar measure $\mu$  on $G$   that assigns a rational number to some compact open subgroup  (and hence  is rational-valued)  is computable on $\+ W$, in the sense that the function assigning   to a compact open coset $ A$  the rational $\mu(A)$ is computable. Consider left Haar measures, say. First suppose that (c) holds. Given $A$, compute the  subgroup $V$ such that $A= A\cdot V$, i.e., $A$ is a left coset of $V$.   Compute $W = U \cap V$. We have  $\mu(A) = \mu(V)= \mu(U) \cdot  |V:W|/|U:W|$. 

Conversely, if the Haar measure is computable on $\+ W$, then (c) holds because $|U\colon V|= \mu(U)/\mu(V)$.
\end{remark}

   For discrete groups, the condition (c) can be dropped, as the proof of the following shows.
 \begin{example} \label{ex:discrete computable} A countable discrete group  $G$ is computably t.d.l.c.\ via a meet groupoid 
 $\LR$ $G$ has a computable copy in the usual sense of  \cref{compStr}. \end{example}
 \begin{proof} For the  implication $\LA$,  we may assume that  $G$ itself is computable; in particular, we may assume that its domain is a computable subset of  $\NN$.  Each compact coset in  $G$  is finite, and hence can be represented by a strong index for a finite set of natural numbers. Since the group operations are computable on the domain, this implies that   the meet groupoid of $G$  has a  computable copy. It is then trivially Haar computable. 
 
 For the implication $\RA$, let $\+ W$ be a Haar computable copy of $\+ W(G)$.  Since $G$ is discrete, $\+ W$ contains a least subgroup $U$.  The set of left cosets of $U$ is computable, and forms a group with the groupoid and inverse operations.  This yields the required computable copy of $G$. \end{proof}
 
%
 
By $\mathbb{Q}_p$ we denote the additive group of the $p$-adics. By the usual definition of semidirect products (\cite[p.\ 27]{Robinson:82}),   $\ZZ \ltimes \QQ_p$ is the  group defined on the Cartesian product $\ZZ \times \QQ_p$ via the binary operation  $\la z_1, \aaa_1 \ra \cdot  \la z_2, \aaa_2\ra = \la z_1 + z_2, p^{z_2} \aaa_1  + \aaa_2\ra $. 
This turns $\ZZ \ltimes \QQ_p$ into   a topological group with  the product topology. 


\begin{proposition} \label{ex:Qp} For any prime $p$,   the  additive group  $\QQ_p$ and   the group $\ZZ \ltimes \QQ_p$  are  computably t.d.l.c.\ via a meet groupoid. 
    \end{proposition}

\begin{proof} 
  We begin with  the additive group $\QQ_p$. Note that  its    open proper  subgroups  are  of the form $U_r:= p^{r}\ZZ_p$ for some  $r\in \ZZ$. Let $C_{p^\infty}$ denote the Pr\"ufer group $\ZZ[1/p]/\ZZ$, where $\ZZ[1/p]= \{ z p^{-k} \colon \, z \in \ZZ \lland k \in \NN\}$. 
 For each $r$ there  is a canonical  epimorphism  $\pi_r\colon  \QQ_p \to C_{p^\infty}$ with kernel $U_r$: 
   if  $\aaa= \sum_{i=-n}^\infty s_ip^i$ where    $0 \le s_i < p$, $n\in \NN$, we have
 \bc $\pi_r(\aaa) = \ZZ + \sum_{i=-n}^{r-1} s_ip^{i-r}$; \ec 
  here an empty sum is interpreted as $0$. (Informally, $\pi_r(\aaa)$ is obtained by taking the  ``tail" of $\aaa$ from the position $r-1$ onwards to the last position, and shifting it in order  to represent     an element of $C_{p^\infty}$.)   So   each compact open coset in $\QQ_p$  can be uniquely written in  the form $D_{r,a}= \pi_r^{-1}(a)$ for some  $r \in \ZZ$ and $a \in C_{p^\infty}$. The domain $S\sub \NN$ of the  Haar computable copy $\+ W$ of $\+ W(\QQ_p)$ consists of the natural numbers  encoding such pairs $\la r,a\ra$ according to some fixed encoding. They will be identified with the cosets they denote. 

The groupoid operations are computable because we have $D_{r,a}^{-1}= D_{r,-a}$, and $D_{r,a} \cdot D_{s,b}= D_{r, a+b}$ if $r=s$, and undefined otherwise.  
 It is easy to check that   $D_{r,a} \sub D_{s,b}$ iff $r \ge s$ and $p^{r-s}a=b$. So the inclusion relation is decidable.  We have  $D_{r,a} \cap D_{s,b}= \ES$ unless one of the sets  is contained in the other, so the meet operation is computable.   
 Finally,   for $r\le s$, we have $|U_r: U_s|= p^{s-r}$ which is computable.

Next, let  $G = \ZZ \ltimes \QQ_p$; we   build a  Haar computable copy $\+ V$ of $ \+ W (G)$. We will extend the listing $(D_{r,a})_{r \in \ZZ, a \in C_{p^\infty}}$ of compact open cosets in $\QQ_p$ given above.  For each   compact open subgroup of $G$, the projection onto $\ZZ$ is a compact open subgroup of $\ZZ$, and hence    trivial.  So the only compact open subgroups of $G$ are of the form~$U_r$.  Let $g\in G$ be the generator of $\ZZ$  such that $g^{-1} \aaa g = p \aaa$ for each $\aaa \in \QQ_p$ (where $\ZZ$ and $\QQ_p$ are thought of as canonically embedded into $G$). Each compact open coset of~$G$ has a unique  form $g^z D_{r,a}$ for some $z \in \ZZ$.  Formally speaking, the domain of the computable copy of $\+ W(G)$  consists of natural numbers encoding  the triples $\la z, r, a\ra$ corresponding to such cosets according to some fixed encoding; as before they will be identified with the cosets they denote. 

 To show that the  groupoid and meet    operations are   computable, note that we have     $gD_{r, a } = D_{r-1, a} g  $ for each $r \in \ZZ, a \in C_{p^\infty}$, and hence  
 $ g^z D_{r, a } = D_{r-z, a}  g^z $ for each $z \in \ZZ$.
Given two cosets $g^v D_{r,a} $ and $ g^w D_{s,b}= D_{s-w, b}g^w$,  their composition is   defined iff $r=s-w$, in which case the result is $g^{v+w} D_{s,a+b}$. 
The inverse of $g^z D_{r,a}$ is $D_{r, -a}g^{-z} = g^{-z}D_{r-z, -a}$. 

To decide  the inclusion relation, note that we have $g^z D_{r,a} \sub g^w D_{s,b}$ iff $z=w$ and $D_{r,a} \sub D_{s,b}$, and otherwise, they are disjoint. Using this, one can show that the meet operation is computable (by an argument that works in any computable meet groupoid $\+ V$): recalling the notation in \cref{rem:cat},  if $A_0, A_1 \in \+ V$,  $A_i \colon U_i \to V_i$, and $A_0,A_1$ are not disjoint, then  $A_0 \cap A_1$ is the unique $C \in \+ V$ such that $C \colon U_0 \cap U_1 \to V_0 \cap V_1$ and $C\sub A_0, A_1$.  Since $\+ W$ satisfies Condition (c) in \cref{def:comp_meet_groupoid}, and $\+V$  has no subgroups beyond the ones present in $ \+ W$, we conclude that $\+ V$ is Haar computable. 
\end{proof}


\section{Uniform equivalence of two   definitions of computably t.d.l.c.}  We  show that Definitions~\ref{Def1} and~\ref{Def2}  of computably  t.d.l.c.\  groups are equivalent in a computationally uniform way. This provides    a  first  evidence for the robustness of this class of t.d.l.c.\ groups.
 
  \begin{theorem}  \label{thm:main} \ \\ A  group   $G$ is computably t.d.l.c.\ via a closed subgroup of $\S$  
$\LR$   

\hfill    $G$ is computably t.d.l.c.\   via  a meet groupoid.

\n Moreover, from a presentation of~$G$ of one type, one can effectively obtain a presentation of $G$ of the other type.  \end{theorem} 
\begin{proof} ``$\RA$": Suppose that $G$ is computably t.d.l.c.\  as a closed subgroup of $\S$. To save on notation, we may assume that $G$ itself is a closed subgroup of $\S$ showing this.  We will obtain   a Haar computable copy of its meet groupoid $\+ W(G)$.

Recall from Definition\ \ref{def:E}  that  $E=E_{\Tree G}\sub \NN$ is   the  decidable  set of       minimal code numbers for compact open subsets of~$G$.   For $u \in E$   we  can decide whether $\+ K_u$ is a subgroup: this is the case precisely when  $\+ K_{M(u,u)}= \+ K_u$  and $\+ K_{I(u)} = \+ K_u$, where   the computable functions $I,M$ were defined in  Lemma~\ref{lem:basic operations}. We can also decide whether $B=\+ K_v$ is a coset: this is  the case   precisely when  $B   B^{-1}$ is a subgroup $U$ (in which case $B$ is its right coset). Equivalently,    $\+ K_{M(v,I(v))}$ is a subgroup, which is a decidable condition.  

The  domain $D$ of the computable copy  of $\+ W(G)$ is the subset of $ E$  consisting of the      code numbers for cosets. For   cosets $A= \+ K_u, B= \+ K_v \in \+ W(G)$  recall  that $A \cdot  B$ is defined iff $A$ is a left  coset of a subgroup $V$ such that  $B$ is a right  coset of~$V$. Equivalently,    $\+ K_{M(I(u), u)}= \+ K_{M(v,I(v))}$, which   is a decidable condition. 
So the groupoid operations and meet operation are computable.

It remains to show that,   $u,v\in D$ coding   subgroups $U,V$, one can compute $|U  : U \cap V|$.    
To do so, one  enumerates (code numbers in   $D$ of)     left cosets of $U \cap V$ contained in $U$,  until their union equals $U$. 
 (There is an algorithm to decide the latter condition   by Lemma~\ref{lem: comp index tree}(i).) Then one  outputs the number of cosets found.

%
%

 \medskip

\n   ``$\LA$": The basic idea is    that the left translation action of $G$ on the given Haar computable presentation $\+ W$ of its meet groupoid, namely  $\la g, A \ra \mapsto gA$ for   $g \in G,  A \in \+ W$,  yields a computable presentation of $G$ via a closed subgroup of~$\S$. The main work  is to show that the group $\wt G$ of permutations of $\+ W$ so obtained can be described   by a c.l.c.\  tree   (\cref{def:comploccompact}) whose paths can be viewed as permutations.  

To carry this out, we introduce an  operator $\+ G_\text{comp} $ from meet groupoids to permutation groups that  will be needed frequently later on. In  case     that $\+ W$ is a copy of  the  meet groupoid  of $G$, the operator  produces the  permutation group given by this left action of $G$. However, the operator    can be defined   based on the meet groupoid alone, and  thus can be  seen as a  dual of the operator $G \to \+ W(G)$    obtained in the first part of the proof.   

We emphasise     that the elements of $\S$ are not actually permutations, but paths $p$ on $ \Tree {\S}$ encoding   pairs consisting of a permutation and its inverse. Nonetheless, if   $A \in \+ W$ is denoted by $x$,  below  we   often  suggestively  write $p(A)$ for the element of $\+ W$ denoted by $p(2x)$.  
 \begin{definition} \label{def:Gof} Given a   meet groupoid $\+ W$ with domain $D\sub \NN$, let $\wt G= \+ G_\text{comp} (\+ W)$ be the closed subgroup of $ \S$ consisting of elements $p$   such that  $p(2x), p(2x+1) \in D$ for each $x \in D$,  $p(2x )= p(2x+1)= x $ for each  $x \not \in D$, and    \bi \item[(a)]   $p$ preserves the meet operation of $\+ W$ \item[(b)]  $p(A) \cdot B = p(A\cdot B) $ whenever $A \cdot B$ is defined.  \ei \end{definition}
  The formal conditions  expressing  (a) and (b) state that: \bi \item[-] for all $x,y \in D$, we have $p(2 (x \cap  y))= p(2x) \cap p(2y)$; \item[-]  if $x \cdot y$ is defined then so is $p(2x) \cdot y$, and $p(2x) \cdot y= p(2(x\cdot y))$. \ei  Note that such $p$ can be identified with   the automorphisms of the  structure obtained from $\+ W$ which, instead of composition, for each $B$ has a  partial unary operation $A \mapsto A \cdot B$.

Suppose now  that $\+ W$ is as in~\cref{Def2}. 
  Then there is an isomorphism of meet groupoids $  \+ W \to \+W(G)$, which below we will use to identify $\+ W$ and $\+ W(G)$.   We can make  the assumptions that $0\in D$,  and that $0$ denotes   a subgroup $U$  in $\+ W$,   without  affecting  the uniformity statement of the theorem: otherwise we can search   $\+ W$ for the least $n$ such that $n$ is a subgroup, and then work with a new copy of $\+ W(G)$ where the roles of~$0$ and $n$ are swapped.    

We note that for  each  subgroup $U\in \+ W(G)$, the set $B=p(U)$  is a left coset of~$U$ since  $B\cdot U = p(U)\cdot U= p(U\cdot U)= B$.  
Define a group homomorphism  $\Phi \colon G \to \wt G$ by letting   $\Phi(g)$ be   the element of $\S$ corresponding to the  left translation by~$g$, i.e. $A \mapsto gA$ where $A \in \+W(G)$.  (See the comment after \cref{def:Gof}.) Note that $\Phi$ is injective because the compact open subgroups form a neighbourhood basis of $1$:
if $g \neq 1$ then $g \not \in U$ for some compact open subgroup $U$, so that  $\Phi(g)(U) \neq U$.

\begin{claim} \label{cl:isomoG}  $\Phi \colon  G \to  \wt G$ is an isomorphism of topological groups. \end{claim} 
\n  To show  that $\Phi$ is onto, let $p \in \wt G$. Clearly the subgroups form a filter in $\+ W(G)$. Using (a) this implies that   $\{p(U)\colon U \in \+ W(G)  \text{ is a subgroup}\}$   is also a  filter on $\+ W(G)$. Since this filter  contains a compact set, there is an element $g$ in its intersection.
  Then $\Phi(g)= p$: since  the set $B=p(U)$  is a left coset of $U$,  it equals $gU$. So, if $A$ is a right coset of $U$, then $p(A)= p(U\cdot A)=  B\cdot A = gA$.
 
  To show that  $\Phi$ is continuous at $1$ (and hence continuous), note that a basis of neighbourhoods of the identity in $\wt G$ is given by the open sets
   \bc $\{ p\in \wt G \colon \forall i \le n \, [p(A_i) = A_i]\}$, \ec where $A_1, \ldots, A_n \in \+ W(G)$. Given such a set, suppose $A_i$ is a right coset of $U_i$, and let $U = \bigcap U_i$. If $g \in U$ then $gA_i = A_i$ for each $i$. 
By the open mapping theorem for Hausdorff  groups \cref{rem: open mapping} $\Phi$ is open. This verifies the claim.

To complete the proof of the implication ``$\LA$" of the theorem, we now    show  that 
$\Tree {\wt G} $ is c.l.c.\  as in Definition \ref{def:comploccompact}, so that  $\wt G$ is a computable presentation of $G$ via a closed subgroup of $\S$ as required. The following claim will  be used to show that   $\Tree {\wt G}$ is computable, using the assumption  that $\+ W$ is  a  computable meet groupoid.   As usual by $A,B$ we denote compact open cosets of $G$, identified here with elements of $D \sub \NN$, the domain of $\+ W$. 
\begin{claim}  \label{cl:compext}A finite injection $\aaa $ on $\NN$ can be  extended to some $p \in \wt G$ $\LR$

\n $\aaa(x)= x$  for each $x \not \in D$, and 

  $B\cdot A^{-1}$ is defined whenever $\aaa(A)= B$, and  $\bigcap  \{ B \cdot A^{-1} \colon \, \aaa(A) = B\} $ is non-empty. \end{claim}
\lapf  Let $g$ be an element of the intersection. Then $gA = B\cdot A^{-1} \cdot  A = B=\aaa(A)$ for each $A \in \dom (\aaa)$.

\rapf  Suppose $p\in \wt G $ extends $\aaa$.  By \cref{cl:isomoG}, there is $g\in G$ such that  $p = \Phi(g)$.  Then $gA = p(A) = B$ for each   $A,B$ such that $\aaa(A)= B$. Such  $A,B$ are   right cosets of the same subgroup.  Hence $B\cdot A^{-1}$ is defined, and clearly  $g$ is in the intersection. 
This establishes the   claim.
 
 \medskip
 
 Recall  from (\ref{eqn:inj}) in Subsection~\ref{ss:Sinf} that a string $ \sss \oplus \tau\in \Tree \S $ gives rise to a finite injection $\aaa_{ \sss \oplus \tau}$, defined by   $\aaa_{ \sss \oplus \tau}(r)=s $ iff $\sss(r)= s \,   \lor\,  \tau(s)=r$.    So   \[ S = \{\sss \oplus \tau \colon \,  
 \aaa_{ \sss \oplus \tau} \text{ can be extended  to some } p \in \wt G\}\]
 is a  subtree of $\Tree \S$ with no leaves that is computable because the condition in \cref{cl:compext} is decidable. Clearly $\wt G= [S]$, and  hence $S=\Tree {\wt G}$.

%
Claim~\ref{cl:fcuk} below will  verify  that  $S $  is  a c.l.c.\ tree in the sense of Definition\ \ref{def:comploccompact} by providing a computable bound $H$. The following lemma does the main work, and will also be used  later on, such as  in the proof of Prop.\ \ref{prop: comp isom} below.  Informally, it says that given  some subgroup $U \in \+ W$,  if one  declares that $p\in \wt G$ has  a   value $L\in \+ W$  at $U$, then one  can compute for any $F \in \+ W$ the  finite set of possible values of $p$ at $F$. 
\begin{lemma} \label{lem: sst} Suppose that   $U \in \+ W$ is a subgroup  and $L$ is a left coset of $U$. Let $F \in \+ W$. One can uniformly in    $U, L$ and $F$ compute a strong index for the  finite set  $\+ L=\{ p(F) \colon \, p \in  [S] \lland p(U) = L\}$.  
\end{lemma}
To see this, first one computes $V= F\cdot F^{-1} $,  so that $F$ is a right coset of the subgroup $V$. Next  one computes   $k=|U \colon U \cap V|$, the number of left cosets of $U \cap V$ in $U$.  Note that  $\+ L_0= \{ p(U \cap V) \colon p \in  [S] \lland p(U) = L \}$ is the set of left cosets of $U \cap V$ contained in $L$.  Clearly this set  has size $k$. By searching $\+ W$ until all of its elements have appeared, one can compute a strong index for this  set.
 Next one computes a strong index  for   the set $\+ L_1$ of left cosets $E$ of $V$ such that $C \subseteq E$ for some $C \in \+ L_0$ (this uses that, given  $C$, one can compute $E$). 
  Finally one outputs a strong {index}  for the set $\{ E F \colon \, E \in \+ L_1\}$, which equals $\+ L$.  This proves the lemma. 
 


For  the following  recall that the computable set $D\sub \NN$ is the domain of $\+ W$, and that $0 \in D$ denotes a subgroup $U$. If $a,b \in D$ denote left cosets $L_0,L_1$ of $U$, then the string  $(a,b) $ is extended by a path $p\in [S]$ iff $p(U)= L_0$ and $p^{-1}(U)= L_1$. 
   \begin{claim} \label{cl:fcuk} There is a computable binary   function $H$ such that,  if    $ \sss$ is the string  $(a,b)$ and $\sss \in S$  (so that  $a,b\in D$),    and $\rho \in S$ extends $\sss$, then  $\rho(i) \le H(\sss, i)$ for each $i < |\rho|$.  Thus, the tree $S$ is c.l.c.\ via the parameter $k=2$.  \end{claim}

\n     
To see this, let   $v=\lfloor i/2 \rfloor$. If $v  \not \in D$, let $H(\sss, i)= i$. Otherwise,    let $F$ be the coset denoted by $v$. If $i= 2v$, let $L$ be the coset denoted by $a$. If $i= 2v+1$, let $L$ be the coset denoted by $b$.  Applying  Lemma~\ref{lem: sst} to $U, L, F$, one can compute     $H(\sss, i) $ as the greatest number denoting an element of $\{ p(F) \colon \, p \in  [S] \lland p(U) = L \}$.   
   \end{proof}

\subsection*{Understanding the automorphism group via the meet groupoid}
The meet groupoid  $\+ W(G)$ might turn out to be a useful tool for studying $G$, independently of   algorithmic considerations.  For a general locally compact group $G$, the group $\Aut(G)$  becomes a Polish group via the Braconnier topology, given by the sub-basis of identity neighbourhoods of the form 
\bc $\mathfrak A(K,U) = \{ \aaa \in \Aut(G)  \colon \, \forall x \in K  [\aaa(x) \in Ux \lland \aaa^{-1}(x) \in Ux] \} $, \ec
where $K$ ranges over the compact subsets of $G$, and $U$ over the  identity neighbourhoods of $G$.  As noted in \cite[Appendix A]{Caprace.Monod:11}, $\Aut(G)$ with this topology is Polish (assuming that  $G$ is countably based). For a t.d.l.c.\ group $G$, the following  shows that      $\Aut(G)$   can be viewed as    the automorphism group of a countable structure.
\begin{prop} \label{prop: Braconnier} Let $G$ be a t.d.l.c.\ group. The group $\Aut(G)$ with the Braconnier topology is topologically isomorphic to $\Aut(\+ W(G))$, via the  map $\Gamma$  that sends  $\aaa \in \Aut(G)$ to its action on $\+ W(G)$, that is,  $B \mapsto \aaa(B)$.  \end{prop}
\begin{proof} It is clear that  $\Gamma$ is an injective  group homomorphism. 
		To show that  $\Gamma$ is continuous, consider an identity neighbourhood of $\Aut(\+ W(G))$, which we may assume to have  the form $\{\beta \colon \, \beta(A_i)= A_i, i = 1, \ldots,  n\}$ where $A_i \in \+ W(G)$.  Let $U = \bigcap_i A_i A_i^{-1}$ and $K= \bigcup A_i$. Then $\aaa \in \mathfrak A(K,U)$ implies $\aaa(A_i) = A_i$ for each $i$. 
	
	It remains to show that $\Gamma $ is onto: then, 
	since  $\Aut(G)$ and $\Aut(\+ W(G))$ are Polish, $\Gamma $ is open (see e.g. \cite[Th.\ 2.3.3]{Gao:09}), and hence a topological group isomorphism. For  a  direct argument  that $\Gamma$ is onto, see \cite[Prop 9.5]{LogicBlog:22}. Given the work already done, here it is easier, however, to derive ontoness using  the isomorphism $\Phi \colon G \to \wt G$ defined in the proof of the  implication ``$\RA $" of~\cref{thm:main}.  The map $\Phi$  can be defined in the absence of algorithmic considerations. (We can now choose  $D= \NN$, ignoring the trivial case that  $G$ is finite. Furthermore, in the argument below we may ignore the fact that the elements of $\wt G$ are not literally permutations of $\NN$.)

	Let $\Phi^* \colon \Aut (G) \to \Aut(\wt G)$ be the isomorphism induced by $\Phi \colon G \to \wt G$, namely, 
	\bc $\Phi^*(\aaa) = \Phi \circ \aaa \circ \Phi^{-1}$. \ec 
	Let $\Theta \colon \Aut(\+ W(G)) \to \Aut(\wt G)$ be the 1-1 map given by \bc $ \Theta(\beta)(p) = \beta \circ  p \circ \beta^{-1}$, \ec  for any $\beta \in  \Aut(\+ W(G))$ and $p \in \wt G$.  It is clear that $\Theta(\beta)(p) \in \wt G$, and $\Theta(\beta) $ is a continuous automorphism of $\wt G$. The following diagram summarizes the maps.
	\[\xymatrix{
		\Aut(\wt G)&  \\
		\Aut(G) \ar[u]^{ \Phi^*}      \ar[r]_{\Gamma}                     & \Aut(\+W)\ar[ul]^\Theta }\]
	To verify that $\Gamma$ is onto, we show that $ \Theta \circ \Gamma =\Phi^* $. 
	We use letters $p,q,r$ for elements of $\wt G$.  Let  $p \in \wt G$, $\aaa \in \Aut(G)$ and   $C\in \+ W(G)$ be arbitrary. 
	Firstly, note that 	$ \Theta (\Gamma (\aaa))$ is the map sending $p$ to $q$ where $q(C)= \aaa(p(\aaa^{-1}(C)))$.  
	Secondly, recall from the proof   of~\cref{thm:main} that $\{\Phi^{-1}(p)\}= \bigcap_U p(U)$ where $U$ ranges through the compact open subgroups of~$G$. So $\Phi^*(\aaa)(p)$ is the map $r$ sending $C$ to $[\bigcap_U \aaa(pU)]C$.
	
We verify  that $q=r$.   If $B$ is a left coset of a subgroup $W$,  we have $p(B)= [ \bigcap_U pU] B$:  both sides are left cosets of $W$, and we have  $p(B)= p(V)B \supseteq [ \bigcap_U pU] B$ where  $B$ is a right coset of a subgroup $V$. Letting $B= \aaa^{-1}(C)$ yields   \bc $p(\aaa^{-1}(C) )=[ \bigcap_U pU] \aaa^{-1}(C)$. \ec  Applying $\aaa$ to both sides of this equation   shows that      $q(C)=r(C)$.
	\end{proof}
	

\section{Computable functions on the set of paths of computable trees} \label{s:comp notions}

This section provides   preliminaries on  computability of  functions that are defined on the set of paths of computable trees without leaves. These preliminaries will be used in Section~\ref{s:Baire} to introduce    computable Baire presentations of  t.d.l.c.\ groups, as well as  in much   of the rest of the paper.
Most of the  content of  this section can either be seen as a  special   case of known results  in abstract computable topology, or can be   derived from such results. These results stem from the study of computably compact metric spaces; some of them need to be extrapolated to the locally compact setting. They can be found  in the recent surveys \cite{Iljazovic.Kihara:21,EffedSurvey}. However, with the reader in mind who has little background in computability or  computable topology, we prefer to provide intuition and elementary proofs for     the  computability notions and results that are needed later~on, rather than referring to such more general results.  

Let  $T$ be a computable subtree of $\NN^*$ without leaves.    To define that   a function which takes arguments from   the potentially  uncountable domain   $[T]$  is computable,   one descends to the countable domain of strings on $T$, where   the usual computability notions   work.  The first definition,  \ref{def:computable 1} below,   will apply   when we show  in~\cref{cor:delta} that  the modular function on a computable presentation of a  t.d.l.c.\ group  is computable.  As a further  example, in~\cref{fact:mg} we will show that given a computable presentation of a t.d.l.c.\ group via a closed subgroup of $\S$,   the function $m(g,V)$ related to the scale function, defined in the introduction,  is computable. For the notion of partial computable functions on $\NN$, see the remark after~\cref{def:computable}. We assume   some standard effective encoding of strings by numbers.

We define what it means to say that $\Phi: [T] \times \NN \to \NN$ is computable. The intuition    is  that  a  partial computable function $P_\Phi$     represents the behaviour of 
 a so-called  \emph{oracle  Turing machine} computing $\Phi$.
 Given $f \in [T]$,  the machine has the  list of the values $f(0), f(1), f(2), \ldots $ written on a special  ``oracle" tape, and attempts to find the value  $\Phi(f, w)$ via queries of the type ``what is  the value of $f(q)$?", where $q$ is determined during the computation. If  a string $\sss $ is an initial segment of $ f$ and $P_\Phi(\sss,w) = m$, then with the answers  given by $\sss$,  the machine can find this value $m$. Clearly any extension $\tau$ of $\sss$ will then yield the same value: this is Condition~(1). Condition (2) says that sufficiently many queries will lead to an answer. By $f\uhr n$ one denotes the length $n$  initial segment of $f$. 

\begin{definition}  \label{def:computable 1} (a) A function $\Phi: [T] \times \NN \to \NN$ is computable if there is a partial computable    function $P_\Phi $ defined on a subset of $T\times \NN$, with values in $\NN$, such that  \bi \item[(1)] if $\sss \prec \tau\in T$, then $P_\Phi (\sss, n) = k $   implies $P_\Phi (\tau, n) = k$; 
\item[(2)] If $f \in [T]$, then 
$\Phi(f, w)= k $ iff there is an $n$ such that  $P_\Phi(f\uhr n,w) =k$. \ei 

\n (b) A function $\Psi: [T]   \to \NN$ is computable if the function $\Phi(f, n)  = \Psi(f)$ (which ignores the number input)  is computable in the sense above. \end{definition}  By (2),     $\Phi $   is     continuous (where $\NN$ carries the discrete topology). 
 
  \begin{example} Let $T= \NN^*$. The function  $\Phi(f,n)= \sum_{i=0}^n f(i)$ is computable. The oracle Turing machine,  with $f$ ``written" on the oracle tape,  queries the values of  $f(i)$ for $i=0, \ldots,n$ one by one and adds them.  \end{example}

We next rephrase \cref{def:computable 1}  in a form that will matter when we discuss computability for  groups where the domain can be of the form $[T]$ for any computable tree without leaves. We will need to express that the group operations are computable. In the setting of closed subgroups of $\S$, this is  the case automatically; for the  inversion operation, it   is due to the particular presentation of $\S$ we chose (\cref{ex:S} below). 

 Given a computable tree $R$ without leaves, there is a computable tree without leaves $T= R^2$  such that $[T]$ can be canonically identified with $[R] \times [R]$:  namely, $T$ consists of the initial segments of strings $\sss_0 \oplus \sss_1$ where $\sss_i \in R$ have the same length.  Given this identification, a path $f$ on $T$ can be written as $\la f_0, f_1 \ra $ where $f_i \in [R]$. 

 \begin{definition}[Computable functions on the set of paths] \label{def:Comp fcn on tree}  Let $R,S,T$ be computable trees without leaves.  A function $\Phi\colon [T] \to [S] $ is called \emph{computable} if the function  $\wt \Phi: [T] \times \NN \to \NN$ given by $\wt \Phi(g, n)= \Phi(g)(n)$ is computable   in the sense of  \cref{def:computable 1}. 
A function $\Psi\colon [R] \times [R] \to [S] $ is   \emph{computable} if the function $\hat \Psi \colon [R^2] \to [S]$ is computable, where $\hat \Psi(\la f_0, f_1 \ra)= \Psi(f_0, f_1)$.    \end{definition}

  It is intuitively clear, and not hard to check from the definitions, that the composition of computable unary functions on sets of paths is again computable.  
  The topology on the  space $[T]$ is induced by a  complete metric:  for instance, for $f \neq g$, let $d(f,g) = 1/n$ where $n $ is least such that $f(n)\neq g(n)$. Taking the usual ``$\epsilon, \delta$" definition, one sees that   $\Phi\colon [T] \to [S] $ is continuous at $f$ if for each~$k$ there is~$n$ such that for each $g \in [T]$, if  $f\uhr n = g\uhr n$ then  $\Phi(f)\uhr k = \Phi(g) \uhr k$. The definitions above can be seen as algorithmic versions of continuity, where one can compute the output $\Phi(f)$ up to $k$ from a sufficiently long part of the input $f$.


\begin{notation} {\rm Suppose that  a function $P$ is as in \ref{def:computable 1} above, and suppose that for each $\sss\in T$, the value $P(\sss, n)$ is only defined for finitely many $n$. We write $P(\sss)= \rho$ if $\rho \in \NN^*$ is the string of maximal length such that  $P(\sss, n)= \rho(n)$ for each $n < |\rho|$.}
 \end{notation} 

Clearly   $P$ as a function on strings is monotonic: if $\sss \preceq \rho$ then $P(\sss) \preceq P(\rho)$. While  $P$   is not partial computable in general, we can think of $P(\sss) $ as the eventual  output of a finite process   depending on   computations $P(\sss, k)$ converging for larger and larger $k$. 
  \begin{remark} \label{rem:monot} Suppose instead we \emph{start} by defining a monotonic  partial computable function $P$ on strings  such that   $\Psi(f)= \bigcup_n \{ P(f\uhr n) \colon \, f\uhr n \in \dom  (P)\}$.  Then defining the partial computable function $P(\sss, n)$ as $P(\sss)(n)$ for any $n< P(\sss)$ shows that $\Psi$ is computable  in the sense of   \ref{def:computable 1}.  In the algebraic applications below such as \cref{prop: SL2}, we will usually proceed in this way. In fact,  the domain of $P$ will usually be  a computable set given by a simple combinatorial condition.  So the function $P$ is  computable in the sense of \cref{def:computable}. \end{remark}
  
%
%
Recall that in Section~\ref{ss:Sinf} we introduced a special way of presenting the elements of  $\S$ as paths $f = f_0 \oplus f_1$ on  a directed tree $\Tree \S$  that keep track of both the permutation $f_0$ and its inverse $f_1$.    For $f,g \in [\Tree \S]$, we defined the group operations of $\S$ by  $f^{-1} = f_1 \oplus f_0$ and 
   $g  f    = (g_0 \circ f_0)    \oplus (f_1 \circ g_1  )$. 
\begin{fact} \label{ex:S} The    group operations  of $\S$ are computable. \end{fact}
\begin{proof}  We  define partial computable monotonic  functions $P_1$ and $P_2$ on $\Tree \S$ and  $\Tree \S^2$, respectively, according to \cref{rem:monot}. 
For the inverse, let $P_1(\sss_0 \oplus \sss_1) = \sss_1 \oplus \sss_0$, where $|\sss_0|=|\sss_1|$. For the binary group operation, given strings $\tau=\tau_0 \oplus \tau_1$ and $\sss=\sss_0 \oplus \sss_1$    of the same, even  length, let $t$ be  greatest  such that for each $r<t$,   $ \rho_0(r):=\tau_0(\sss_0(r))$ and $\rho_1(r):=\sss_1(\tau_1(s))$ are defined. Let $P_2(\tau, \sss) = \rho_0\oplus \rho_1$,    a string of length~$2t$.      
   \end{proof}
 The following lemma is an algorithmic version of the fact  that each continuous function defined on a compact space is uniformly continuous.   This is  well-known in computable analysis. 
 Here we restrict ourselves to the setting of paths spaces $[K]$ that are  compact in an effective way (as given by the bound $H$ below). 
  \begin{lemma} \label{lem:lahm} Suppose that $K$ and $ S$ are computable trees without leaves. Suppose further  that there is a computable function $H$ such that $\sss(i) < H(i)$  for each $\sss \in K$ and $i < \sssl$.
 Let   $\Phi\colon [K]\to [S] $ be computable via a partial computable function~$P\colon K \times \NN \to \NN$. 
\bi \item[(i)] There is a computable function $g$ as follows:  for each $n$, 
\bc $\forall \rho \in K \forall i< n \,  [|\rho| = g(n) \to  P(\rho, i) \text{ is defined}]$. \ec
%
%
 Moreover, $g$ is obtained uniformly in $K$ and $h$.
 \item[(ii)] If $\Phi$ is a bijection then $\Phi^{-1}$ is computable, via a partial computable function that is obtained uniformly in $K,H,S$ and $P$. \ei \end{lemma}

 Intuitively, given the metric on $[K]$ discussed above,  the   function $g$ in (i)   computes the ``$\delta$" in the definition of uniform continuity from the ``$\epsilon$": if $\delta= 1/n$ we have $\epsilon = 1/g(n)$. In computability terms, this means that, to obtain $n+1$ output symbols,  we need at most~$g(n)$ input symbols. 
\begin{proof} (i)  Given $i \in \NN$, 
consider the  subtree of $K$ consisting of the  strings $\rho$   such that $P(\rho, i)$ is undefined. If  this subtree is infinite, then  there is a  path  $f \in [K]$ such that  $\Phi(f)(i)$ is undefined, contradiction. Thus for each $n$ a possible  value $g(n)$ as above exists. 
Using the hypotheses on $K$, given $n$ one can search for the least such value  and output it as   $g(n)$. 

\smallskip

\n (ii) To show  $\Phi^{-1}$ is computable, we define a partial computable function $Q \colon [S] \times \NN \to \NN$. Clearly $[S]$ is compact and  $\Phi^{-1}$ is continuous.   Hence    there  is a computable function $h$ with $h(t) \ge t$ for each $t$ as follows:  given $s \in \NN$,   for each $\rho \in K$ of length $g(h(s))$, the values  $P(\rho,i)$ for $i< h(s)$ together determine $\rho\uhr s$. For a string $\beta \in S$ of length $h(s)$ and $k<s$, define $Q(\beta,k)=  \rho(k)$ where  $\rho\in K$ is a string of length $g(h(s))$ with $P(\rho,i) = \beta(i)$ for each $i< h(s)$. It is easy to verify that $Q$ shows that $\Phi^{-1}$ is computable, and that $Q$ is obtained uniformly in the given data.  For a proof in a more general setting see \cite[Thm.\ 4.2.48]{Downey.Melnikov:book}.
\end{proof}
The rest of this section  discusses computable functions on the set of paths of  c.l.c.\ trees.  Given such a   tree $T$, recall from Definition\ \ref{defn: str index compact}  that by $\+ K_u$ we denote the   compact open subset of $[T]$  with code number $u\in \NN$. That is, $u$ is the strong index for a set of strings $\{\aaa_1, \ldots, \aaa_r \} \sub T$ such that $\+ K_u= \bigcup_{i\le r} [\aaa_i]_T$.  If there is more than one tree under discussion, we will write    $\+ K_u^T$ for the subset of $[T]$ with code number~$u$.  
First, we  establish a useful interaction between computable functions   $[T]\to [S]$ and  such sets.  This   generalizes Lemma~\ref{lem: comp index tree}(ii) where $\Phi$ is the identity function.
\begin{lemma} \label{lem:lahm2} Let $T$ and $S$ be c.l.c.\  trees. Suppose  a function $\Phi\colon [T]\to [S] $ is computable via a partial computable function $P_\Phi $. Given  code numbers $u, w$, one can decide whether $\Phi(\+ K^T_u) \sub \+ K^S_w$. 
\end{lemma} 
\begin{proof}   Suppose that  $u$ is a strong index for the set of strings $\{\aaa_1, \ldots, \aaa_r \} \sub T$, and $w$ is a strong index for the set of strings $\{\beta_1, \ldots, \beta_s \} \sub S$.  Let $K$ be the subtree of $T$  consisting of the prefixes, or extensions,  of some $\aaa_i$. Clearly one can uniformly obtain   a computable bound  $H$ for this $K$ as in Lemma~\ref{lem:lahm}. Let $n= \max_i |\beta_i|$. Let $N= g(n)$ be the length computed  from $n$ through  that Lemma. Then $\Phi(\+ K^T_u) \sub \+ K^S_w$ if and only if  for each $\aaa \in K$ of length $N$, there is an $i$ such that   $P_\Phi(\aaa) \succeq \beta_i$. By   Lemma~\ref{lem:lahm}, this condition is decidable. 
\end{proof} 


To prove \cref{prop: SL2} below, we will need a criterion on whether,  given  a computable subtree   $S$  of a c.l.c.\ tree $T$ (where $S$ potentially has   leaves),  the maximally pruned subtree  of $S$  with the same set of paths is computable.
 \begin{prop}  \label{prop: prune} Let $T$ be a c.l.c.\ tree such that only the root is infinitely branching. Let $S $ be a computable subtree of $T$, and suppose that there is a uniformly computable dense sequence $(f_i)\sN i$ in $[S]$. Then the tree $\wt S= \{\sss  \colon [\sss]_S \neq \ES\}$ is   decidable. (It follows that $\wt S$ is c.l.c. Of course, $[\wt S] = [S]$.)  \end{prop}
\begin{proof}   Given a   string $\sss \in T$,  if $\sss = \ES$ then $\sss \in \wt S$. Assuming $\sss \neq \ES$, we can compute the least   $t\in \NN$ such that either $\sss \prec f_t$, or $\rho \not \in S$ for each $\rho \in T$ of length $t$ such that $ \rho \succeq \sss$; the latter condition can be decided by the hypotheses on $T$. Clearly $\sss \in \wt S$ iff the former condition holds. 
\end{proof} 

\section{Defining computably t.d.l.c.\  groups via  Baire presentations} \label{s:Baire} 
   This section spells out  the third type of computable presentations of t.d.l.c.\ groups   described in Section~\ref{s: types of presentations}, Type B, which generalises Type S. We call them  \emph{computable Baire~presentations}. 
   
    It is well-known that each  $0$-dimensional Polish space  $X$ is homeomorphic to the path space $[T]$ for some tree $T \sub \NN^*$; see  \cite[I.7.8]{Kechris:95}. Clearly  $X$ is locally compact if and only if  for each $f \in [T]$ there is an $n$ such that the tree above $f\uhr n$ is finitely branching.  Considering the set of prefix minimal strings $\sss \in T$ such that $[\sss]_T$ is compact, it is easy to see that we may replace $T$ by a tree such that only the root can have infinitely many successors. So in our  effective setting, it is natural to work with a domain of the presentation that has   the form $[T]$ for a  c.l.c.\  tree~$T$, and   require that the group operations on $[T]$ be computable. (In fact, while  the previous types of computable presentation were group-specific, in the present setting the same approach would work for other types of algebraic structure defined on $[T]$.)   We will show in Thm.\ \ref{thm:main2} that the resulting notion of computably t.d.l.c.\ group   is  equivalent to the previous ones. 
 
Despite this equivalence, as presentations, computable Baire presentations   offer     more flexibility  than the first type (Type S), which   relied on   closed subgroups of~$\S$. This will be evidenced by our proofs  that some algebraic groups over local fields, such as  $\SL_n(\QQ_p)$,  are computable.

 \begin{definition}\label{def:main1}  
A \emph{computable Baire~presentation}  
    is a topological group  of  the form  $H= ([T], \Op, \Inv)$ such that
  \begin{enumerate} \item  $T$ is computably   locally compact as defined in~\ref{def:comploccompact};
  \item  $\Op\colon [T] \times [T] \to [T]$ and $\Inv \colon [T] \to [T]$ are computable.
    \end{enumerate}
    We say that a t.d.l.c.\ group $G$ is \emph{computably t.d.l.c.}  (via  a Baire presentation) if  $G \cong H$ for such a group $H$.    \end{definition}

We verify    next that for profinite groups and  for  countable discrete groups, our notion of  computable presentability (type B) for t.d.l.c.\ groups coincides with the notions established in the literature.   First we provide  the formal definition of a computable profinite presentation. 
 \begin{definition}[La Roche \cite{LaRoche:81}, Smith \cite{Smith:81}] \label{def:profinite}
A computable profinite presentation of a profinite group $G$ is a uniformly computable sequence of strong indices of finite groups $A_i$ and finite surjective maps 
$\phi_i:  A_i \rightarrow A_{i-1}$ such that $G = \varprojlim_i (A_i, \phi_i)$. 
\end{definition}

\begin{prop} \label{ex:prof Baire}  Given a computable profinite presentation of $G$, one can effectively obtain a computable Baire presentation of $G$. \end{prop}
\begin{proof} We use  the notation of \cref{def:profinite}. Let $n_i= | A_i|$. By the hypotheses one  can effectively identify the elements of $ A_i$ with $\{ 0, \ldots,  n_i-1\}$. Let  
\bc $T = \{ \sss \in \NN^* \colon \forall i< |\sss| [\sss(i) < n_i \lland  [ i>0 \to \phi_i(\sss(i)) = \sss(i-1)]]\}$. \ec 
It is clear  from the hypotheses that  $T$ is c.l.c.\ via $k=0$ (Definition~\ref{def:comploccompact}). To show that the binary group operation $\Op$ on $[T]$  is computable, for strings $\sss_0, \sss_1$ on $T$  of the same length $k+1$, let $P_2 (\sss_0, \sss_1)$ be the unique string  $\tau \in T$ of  length $k+1$ such that in $\+ A_k$ one has  $\sss_0(k) \sss_1(k) = \tau(k)$. The inversion operation $\Inv$ on $[T]$   is computable by a similar argument. 
\end{proof}
We will see in  Prop.~\ref{lem:proc}  that the   converse   holds as well. So,   for profinite groups our definition of  computably t.d.l.c.\ group coincides with  the existing one.  {We note that  Smith~\cite{Smith:81}  already obtained this equivalence,  using  a  different   terminology.} 
\begin{proposition} \label{compBaire discrete} A discrete group $G$ has a computable presentation  in the usual sense of  \cref{compStr} if and only if it has a  computable Baire presentation. \end{proposition}
\begin{proof} \rapf Let the  computable set $D\sub \NN$ be the  domain of the  computable presentation.  For the computable Baire presentation, we take the c.l.c.\  tree \bc $T_G = \{ r\ape 0^k \colon \, r \in D \lland k \in \NN\}$,  \ec and the operations canonically defined on $[T_G]$ via partial computable functions that only regard the first entry of the oracle strings.  

\lapf  By \cref{thm:main2} proved shortly below, $G$ has a computable presentation via a meet groupoid. It now  suffices to invoke \cref{ex:discrete computable}. \end{proof}
    By Fact~\ref{ex:S},  each computable presentation of $G$  as a closed subgroup of $\S$  (Definition~\ref{Def1}) is a computable Baire presentation.  \cref{lem:basic operations} showed that for   a computable presentation as a closed subgroup of $\S$, the group theoretic operations on compact open sets are algorithmic.  In the  more general  setting of computable Baire presentations,  a weak  version of    Lemma~\ref{lem:basic operations} still holds. 
Recall from Definition\   \ref{def:E} that  given a c.l.c.\ tree $T$, by  $E_T$ we denote  the set of minimal code numbers $u$ such that $\+ K_u$ is compact. 
\begin{lemma} \label{lem: decide inclusion}  Let $G$ be computably t.d.l.c.\ via a computable Baire presentation $([T], \Op, \Inv)$.

\n (i)  There is  a computable   function  $ I \colon E_T \to E_T$     such that for each $u  \in E_T$, one   has  $\+ K_{I(u)} = (\+ K_u)^{-1}$. 
 (ii) Given  $u,v,w \in E_T$,  one can decide whether $\+ K_u \+ K_v \sub \+ K_w$. 
\end{lemma}
\begin{proof} (i) By Lemma~\ref{lem:lahm2} one can decide whether $\+ K_u\sub (\+ K_w)^{-1}$. The equality  $\+ K_u= (\+ K_w)^{-1}$  is equivalent to $\+ K_u\sub  (\+ K_w)^{-1}\lland \+ K_w\sub  (\+ K_u)^{-1}$. So one lets $I(u)$ be the least index $v$ such that this equality holds.

\n (ii) Let $\wt T$ be the tree of  initial segments of strings of the form $\sss_0 \oplus \sss_1$, where $\sss_0, \sss_1 \in T$ have  the same length. Then $\wt T$ is a c.l.c.\  tree, $[\wt T]$  is  naturally homeomorphic to $[T] \times [T]$, and $\Op$ can be seen as a computable function $[\wt T ] \to [T]$. Now one applies Lemma~\ref{lem:lahm2}. \end{proof}

As discussed at the beginning of this   section, the following adds a further equivalent condition to Theorem~\ref{thm:main}.
   \begin{thm}  \label{thm:main2} \ \\ A group   $G$  is computably t.d.l.c.\ via a  Baire~presentation
$\LR$   

\hfill    $G$ is computably t.d.l.c.\   via a  meet groupoid. 

From a presentation of $G$ of one type, one  can effectively obtain a presentation of $G$ of  the other type.
\end{thm} 
 \begin{proof} \lapf This   follows from the corresponding  implication in Theorem~\ref{thm:main}.
%
%

 \rapf We build a Haar computable copy $\+ W$ of the meet groupoid $\+ W(G)$ as required  in Definition~\ref{Def2}. By  Lemma~\ref{lem: decide inclusion}, one can decide whether $u \in E_T$ is the code number (\cref{defn: str index compact}) of a subgroup.   Furthermore,  one can decide whether $B= \+ K_w $ is a left  coset of  a subgroup $U= \+ K_u$: this holds iff   $BU \sub B$ and $B^{-1} B \sub U$, and the  latter two conditions are decidable by Lemma~\ref{lem: decide inclusion}. Similarly, one can decide whether $B$ is a right coset of $U$.
 
It follows that the set $V= \{ w \in E_T \colon \, \+ K_w \text{ is a  coset}\}$ can be  obtained via   an  existential quantification over a computable binary relation; in other words, $V$ is recursively enumerable (r.e.). 
A basic fact from computability theory states  that    from a description of $V$ as an r.e.\ set one can uniformly obtain a computable set $\wt D \sub \NN^+$ and a computable bijection  $\theta  \colon \wt D \to V$.  (Here  $\wt D$ is the set of ``stages" at which new elements are enumerated into $V$;  one can assume  that at each stage $n \in \wt D$  exactly one element $\theta(n)$  is enumerated.)      Write $A_n = \+ K_{\theta(n)}$ for $n\in \wt D$, and $A_0  = \ES$.

The domain of   $\+ W$   is $D = \wt D \cup \{0\}$.     By  \cref{lem: comp index tree}  the intersection operation on $\+ W$ is computable, i.e., there is a computable binary function $c$ on $\NN$ such that $A_{c(n,k)} =A_n \cap A_k$. Given $n,k \in \NN-\{0\}$ one can decide whether $A_n$  is a right coset of the same  subgroup   that  $A_k$  is a left coset of. In that case,  one can compute the number  $r$  such that  $A_r= A_n\cdot A_k$: one uses that  $A_r$ is the unique coset  $C$ such that \bi \item[(a)] $A_nA_k \sub C$,  and   \item[(b)]$C$ is a right coset of the same subgroup that  $A_k$ is a right coset of. \ei
 As in the corresponding implication in the proof of Theorem~\ref{thm:main}, for subgroups $U,V$, one can compute  $|U \colon U \cap V|$ by finding in  $\+ W$ further  and further  distinct left cosets of $U \cap V$ contained in $U$,  until their union reaches $U$. The latter condition is decidable.
\end{proof}
\begin{notation} \label{def:Wcomp} Given a computable Baire presentation $G$, by $\+ W_\text{comp} (G)$ we denote the computable copy of $\+ W(G)$ obtained in the proof above. \end{notation}

\begin{prop} \label{prop: SL2} Let $p$ be a prime, and let $n \ge 2$.   Let $\QQ_p$ and $F_p((t))$ denote the rings of $p$-adic numbers, and Laurent series over $\mathbb F_p$, respectively. The   t.d.l.c.\ groups $\SL_n(\QQ_p)$ and  $\SL_n(\mathbb F_p((t)))$  have computable  Baire presentations. 
\end{prop}
\begin{proof}
We provide  the proof for the groups $\SL_n(\QQ_p)$, and then  indicate  the changes that are necessary for the groups $\SL_n(\mathbb F_p((t)))$. 
We begin by giving a  computable Baire presentation of $\QQ_p$ as a ring. 

 {Let $Q$ be the tree of strings $\sss\in \NN^*$ such that    all the entries, except possibly the first, are   among $\{0, \ldots, p-1\}$,  and
  $r0 \not \preceq \sss$ for each $r>0$.}  Clearly $Q$ is c.l.c.\ via $k=1$ (Definition~\ref{def:comploccompact}). 
 We think of a  string $r \ape \sss \in Q$ as denoting the rational    number  $p^{-r} n_\sss\in \ZZ[1/p]$, where  $n_\sss$ is  the natural number that  has  $\sss$ as a   $p$-ary expansion, written in reverse order: \bc $n_\sss = \sum_{i< \sssl} p^i \sss(i)$. \ec    The   condition that  $r0 \not \preceq \sss$   for each $r>0$ expresses    that   $p$ does not divide~$n_\sss$ unless $r=0$.   
 Let $\sss, \tau$  be strings  over $\{0, \ldots, p-1\}$   of the same length $\ell$. By $\sss + \tau$ we denote the string $\rho$  of the same length $\ell$  such that  $n_\rho= n_\sss + n_\tau \mod p^\ell$. By   $\sss \cdot \tau$ we denote the string $\rho$ of length $2\ell$ such that $n_\rho= n_\sss  n_\tau \mod p^{2\ell}$.

We now provide   partial computable functions $P, P_1, P_2, P_3$ according to \cref{rem:monot} that are defined on computable subsets of the trees  $Q$ or $Q^2$,  in order to show that various operations on $[Q]$ are computable according to Definition\ \ref{def:Comp fcn on tree}. For  a string of the form~$r\ape \sss $~let 

\medskip 
   $P(r\ape \sss)=( r-k) \ape \tau$,   
   \medskip
   
   \n  where $k\in \NN $ is maximal such that  $k \le r$  and 
$0^k\preceq \sss $, and $\tau $ is the rest, i.e.\ $0^k \tau =  \sss $. To show  the computability of the function $q \to -q$, let
   \medskip
   
    $P_1(r\ape \sss)= r\ape \tau$ where $|\tau|=\sssl$ and $\sss + \tau=0 \mod  p^{\sssl}$. 
    
    \medskip
    
 \n    To show that the addition operation  is computable, if $r, s \in \NN$ and $\sss, \tau$ are strings such that $r+ |\tau| = s + \sssl$, let

  $P_2(r\ape \sss, s\ape \tau) = \begin{cases}
    P (s \ape (0^{s-r} \sss + \tau) &\text{if $r \le s$}\\
   P (r \ape (\sss + 0^{r-s} \tau )&\text{otherwise.}
\end{cases}$
    \n

  \n    To show that the multiplication  operation  is computable, let

  $P_3(r\ape \sss, s\ape \tau) =  \begin{cases} P ((2s) \ape (0^{s-r} \sss \cdot \tau)) & \text{if }  r \le s\\  
     P ((2r) \ape (\sss \cdot  0^{r-s}  \tau )) & \text{otherwise. } \end{cases}$
  
  \n $P_2$ and $P_3$ are   the correct   functions because, say for $r\le s$, in $\ZZ[1/p]$  one has   $p^{-r} n_\sss + p^{-s}n_\tau= p^{-s}(n_{0^{s-r}\sss}+n_\tau)$ and 
  $p^{-r} n_\sss p^{-s}n_\tau= p^{-2s}n_{0^{s-r}\sss}n_\tau$.

We   next provide a computable  Baire presentation $([T], \Op, \Inv)$ of  $\SL_n(\QQ_p)$ as   in \cref{def:main1}.    Let $T$ be the c.l.c.\ tree that is an  $n^2$-fold ``power" of $Q$. More precisely,
  $T = \{ \sss \colon \forall i < n^2 \, [ \sss^i \in Q]\}$, where $\sss^i$ is the string of entries of $\sss$ in positions of the form $\ell n^2 + i$ for some  $\ell,i \in \NN$. Clearly, $[T]$ is c.l.c.\ via $k= n^2$, and $T$  can be naturally identified with the matrix algebra $M_n(\QQ_p)$. By the  computability of  the ring  operations on $\QQ_p$ as verified above, the matrix product is computable as a function $[T] \times [T]\to [T]$, and the   function $\det \colon [T] \to [Q]$ is computable.  
  
  Note that for any c.l.c.\ trees $T$ and $R$, any computable path $f$ of $R$, and any computable function $\Phi \colon [T]\to [R]$, there is a computable subtree $S $ of $T$ such that $[S]$ equals the pre-image $\Phi^{-1}(f)$.  (In the language of computable analysis,     the pre-image is effectively closed.) To see this, suppose that $\Phi$ is computable according to \cref{def:Comp fcn on tree} via a partial computable function $L$  taking pairs in  $T\times \NN$ to $\NN$. Let $S$ consists of the strings $\sss \in T$ of length $t$ such that $t$ steps of the attempted computations $L(\sss,k) $   don't yield a contradiction to the hypothesis that the output of the oracle Turing machine equals $f$. More formally, 
  \bc $S = \{ \sss \in T \colon \forall k < \sssl [ L_{\sssl}(\sss,k)  \text{ is defined }  \to L_{\sssl}(\sss,k)   = f(k)]\}$. \ec
  In our setting, applying  this to the function $\det \colon [T] \to [Q]$ and  the path  $f = 01000\ldots $  that denotes $1 \in \QQ_p$, we obtain a computable subtree $S$ of $T$ such that $[S]$ can be identified with $SL_n(\QQ_p)$.  
  
  It is well-known  that $SL_n(\ZZ[1/p])$ is dense in $SL_n(\QQ_p)$. This is a special case of strong approximation for algebraic groups (see \cite[Ch.\ 7]{Rapinchuk.Platonov:93}), but can also be seen in an elementary way using Gaussian elimination. The paths on $S$ corresponding to matrices in $SL_n(\ZZ[1/p])$ are precisely the ones that are    $0$ from some point on. Clearly there is a computable listing $(f_i)$ of the set of such   paths.   So by \cref{prop: prune} one  can  replace $S$ by a c.l.c.\ tree $\wt S$ such that $[\wt S] = [S]$.

 To obtain a computable Baire presentation based on $\wt S$, note that matrix multiplication on $[\wt S]$ is computable  as the restriction of matrix multiplication on $[T]$. To define the matrix inversion operation $\Inv$, we use the fact that   the inverse of a matrix with determinant $1$  equals its  adjugate matrix; the latter can be obtained  by computing determinants on minors.
  
Next we treat  $SL_n(\mathbb F_p((t)))$. To  give a Baire presentation of    $\mathbb F_p((t))$ as a ring,  
  we use the same tree $Q$ as above, but now think of a string $r\sss$ as representing the finite Laurent series
  $t^{-r} n_\sss\in \ZZ[1/p]$, where   $n_\sss = \sum_{i< \sssl} t^i \sss(i)$. 
  For strings $\sss, \tau$   over $\{0, \ldots, p-1\}$ and of the same length $\ell$, by $\sss + \tau$ we  denote the string $\rho$  of length $\ell$  such that  $n_\rho= n_\sss + n_\tau \mod p^\ell$. By $\sss \cdot \tau$ we denote the string $\rho$ of length $2\ell$ such that $n_\rho \equiv n_\sss n_\tau \mod t^{2\ell}$.  Now the ring operations are computable via the   partial computable functions defined as  for the case of  $\QQ_p$, but   with the new interpretation of the symbols ``$+$'' and  ``$\cdot$" as  operations on strings.  
  The remainder of the proof didn't use any particulars about the computable Baire presentation of the ring $\QQ_p$,  besides the fact that the path denoting $1$ is computable; this  carries over to the present case.   It is   known  that $SL_n(\mathbb F_p(t))$, the group of   matrices  over finite Laurent series with  determinant 1,   forms a    dense subgroup of  $SL_n(\mathbb F_p((t)))$, so from here on we can argue as before.
     \end{proof}

\section{More on   the equivalences of notions of computable presentation} \label{s:autost}
Theorem~\ref{thm:main} and its extension Theorem~\ref{thm:main2} showed that  our various notions of computable presentation of a t.d.l.c.\ group are equivalent. This section obtains extra information  from  the proofs   of these equivalences. The main result of the   section, 
Thm.\ \ref{prop: comp isom} will show  that,  given a computable Baire presentation $ ([T], \Op, \Inv)$ of a t.d.l.c.\ group $G$, one   can  effectively obtain a computable presentation  of $G$  via a closed subgroup $\wt G$ of $\S$, with an isomorphism $\Phi \colon [T] \to \wt G$ so that both  $\Phi$ and $\Phi^{-1}$ are  computable.  The presentation based on a  closed subgroup of $\S$ can be seen as an ``improved" computable Baire presentation. For instance, the group operations  on compact open subsets of $\wt G$ are fully computable by Lemma~\ref{lem:basic operations}, while in the general case we only have the weaker form Lemma~\ref{lem: decide inclusion}.  

The section then proceeds to    two applications of Thm.\ \ref{prop: comp isom}: computability of the modular function, and  
 computability of Cayley-Abels graphs.  
For a compactly generated t.d.l.c.\ group $G$, its Cayley-Abels graphs are   generalisations of the usual Cayley graphs   in the   setting of a (discrete) finitely generated group.  We will show that each Cayley-Abels graph can be interpreted  via first-order formulas with parameters in  the meet groupoid  of $G$. The interpretation is constructed in such a way that  if $G$ is computably t.d.l.c., the graphs are computable in a uniform way. 

In~Section~\ref{s:scale} we apply  Thm.\ \ref{prop: comp isom}  to obtain  an upper  bound on  the computational complexity of the scale function. 
As  a  further application, 
 in Section~\ref{s:autostable} we will obtain an equivalent criterion on   whether  a t.d.l.c.\ group $G$ has a unique computable Baire presentation up to computable group homeomorphism: any two Haar computable copies of its meet groupoid $\+ W(G)$ are computably isomorphic.  
 This is useful because uniqueness of a computable presentation is easier to show for countable structures.  In Theorem~\ref{thm:autostable} we will apply the  criterion to show  the  uniqueness of a computable Baire presentation  for the additive groups of the $p$-adic integers and the $p$-adic numbers, as well as for $\ZZ \ltimes \QQ_p$. 

Besides  the computability theoretic  concepts introduced  in  Section~\ref{s:comp notions}, we will need      a   fact on   continous, open functions on the set of paths of c.l.c.\  trees. It  deduces the  computability of such  a function from the hypothesis  that   its action on the compact open sets is computable in terms of their code numbers. Similar to the  results in Section~\ref{s:comp notions}, this is a special case of a result in the field of computable topology; see    \cite[Lemma 2.13]{EffedSurvey}. However, we prefer to present a short, elementary proof.
\begin{prop} \label{prop: compact open} Let $T$ and $ S$ be c.l.c.\  trees (see Definition\ \ref{def:comploccompact}), and let the function $\Phi\colon [T]\to [S]$ be   continuous and open. Suppose that   there is a computable function $f\colon \NN \to \NN$ such that for each code number  $u$ one has $\Phi(\+ K_u^T)=\+K_{f(u)}^S$.  Then

(i)   $\Phi$ is    computable; 

(ii) if $\Phi$ is a bijection then     $\Phi^{-1}$ is computable as well.  \end{prop}
\begin{proof} (i) To show $\Phi$ is computable, we define a partial computable function $P_\Phi$ on $T$ with values in $S$ according to \cref{rem:monot}. Given $\sss \in T$ such that $[\sss]_T$ is compact, compute  $u$ such that $\+K^T_u = [\sss]_T$. Let $\{\beta_1, \ldots, \beta_s \} \sub S$ be the finite set with strong index $f(u)$. Let $\eta=P_\Phi(\sss)$ be the longest common initial segment of the strings $\beta_i$. 
Note that with the standard ultrametric on $[S]$, the set $[\eta]_S$ is a closed ball centred at  any element  of $\+ K_{f(u)}^S$,  with  the same diameter as $\+ K_{f(u)}^S$. So if $\sss \preceq \tau \in T$ then $[P_\Phi(\sss)]_S \supseteq  [P_\Phi(\tau)]_S$. This shows that $P_\Phi$ is monotonic. Since $\Phi$ is continuous, one has \bc $\Phi(f)= \bigcup_n  
\{ P_\Phi (f\uhr n)\colon \, f\uhr n \in \dom  (P_\Phi)\}$ \ec as required.

\n (ii) By Lemma~\ref{lem: comp index tree} applied to $S$,  there is a computable  function  $g$ that is ``inverse" to $f$ in the sense  such that for each $v$ such that $\+ K_v^S$ is compact, one has $\Phi^{-1}(\+ K_v^S)=\+K_{g(v)}^T$. Now one applies (i) to   $\Phi^{-1}$ and $g$.
\end{proof}
 
\subsection{Computable action on the meet groupoid}
 Towards the main result   of this  section, let $G= ([T], \Op, \Inv)$ be a computable Baire presentation of a t.d.l.c.\ group.  Let $\+ W $ be the computable copy of $ \+ W(G)$ with domain~$D$  given by the proof of ``$\RA$" in Thm.\ \ref{thm:main2}.  Recall (\cref{def:Wcomp}) that we    write $\+ W= \+ W_{\text{comp}}(G)$.   %
Let $\wt G = \+ G_\text{comp} (\+ W)$ as in \cref{def:Gof}. Let $\Phi \colon G \to \wt G$  be given by the     action $(g,A) \mapsto gA$; see  near the end of   the proof of the implication ``$\LA$" of Thm.\ \ref{thm:main}.

\begin{thm} \label{prop: comp isom}  Let $G, \+ W$,  $\wt G$ and the homeomorphism $\Phi \colon G \to \wt G$   be as defined above.
\bi
\item[(i)]
  $\Phi$   is computable, with a  computable inverse. 
  \item[(ii)] The   action $[T] \times D \to D$, given by   $(g,A) \mapsto gA$,  is computable.


 \ei \end{thm}
%
\begin{proof} 



\n (i) Write $S = \Tree {\wt G}$. By   Prop~\ref{prop: compact open} it suffices to show that there is a computable function $f$ such that $\Phi(\+ K_u^T)=\+K_{f(u)}^S$ for each code number $u$.

We may assume that $\+ K_u^T$ is of the form $A \cap B$ where $A $ is a left coset of a subgroup~$U$ such that $B$ is a right coset of $U$: trivially a  group $G$ is the union of the  right cosets of any given subgroup; hence  such sets form a basis of the topology of $G$. And by Lemma~\ref{lem: comp index tree}, given a general compact set $\+ K_w^T$ one can effectively write it as a finite union of sets of this form. 

If $\+ K_u^T$ is of the form $A \cap B$ as above, we have \bc $\Phi(\+ K_u^T)= \{p \in \wt G \colon \, p(U)= A \lland p^{-1}(U)= B\}$. \ec  Given any $F \in \+ W$,  by Lemma~\ref{lem: sst}  we can  compute  bounds on   $p(F)$ and $p^{-1}(F)$ whenever $p \in \Phi(\+ K_u^T)$,   letting $L=A$ and $L=B$ respectively. Recall the set $E_T$ from  Definition\   \ref{def:E}, and recall from   the proof of ``$\RA$" in Thm.\ \ref{thm:main2}  that $\theta \colon D \to E_T$ is a 1-1  function with range  the minimal code numbers of compact open cosets in $G$; we write $A_n$ for the coset  with code number $\theta(n)$, and often identify  $A_n$ with $n$. Suppose that $U = A_n$.   Let $f(u)$  be a strong index for the finite set of strings $\beta\in S$ of length $2n+2$ such that   \bi  \item[(a)] $A_{\beta(2n)} = A$ and $A_{\beta(2n+1)} = B$,  
 \item[(b)]  for $r< 2n$ of the form $2i$,     $\beta(r) $ is less than the bound  on $p(F)$ given by Lemma~\ref{lem: sst} where $L=A$ and  $F= A_i$.
  \item[(c)]  for $r< 2n$ of the form $2i+1$,     $\beta(r) $ is less than the bound  on $p^{-1} (F)$ given by Lemma~\ref{lem: sst} where  $L= B$ and $F= A_i$.
 \ei
Then $\Phi(\+ K_u^T) = \+K^S_{f(u)}$ as required.

\n (ii) Formally, by    \cref{def:Comp fcn on tree},  the left   action is computable iff $\Phi$ is computable.   Informally speaking,   we use an   oracle Turing machine  that  has as an  oracle a path $g$ on $[T]$, and as an input an  $A \in \+ W$. If $A$ is  a left coset of a subgroup $V$, it outputs     the left  coset $B$  of~$V$ such that it can   find a string   $\sss \prec g$ with  $[\sss]_TA \sub B$.  
\end{proof} 
Note  (ii) implies that the right action is also computable, using that $Ag = (g^{-1}A^{-1})^{-1}$ and inversion is computable both in $G$ and in $\+ W$. 

We apply the theorem to obtain a computable version of the open mapping theorem for t.d.l.c.\ groups (\cref{rem: open mapping})  in the case of a  bijection. (This shows that the inverse of $\Phi$ in (i) of the theorem is in fact automatically computable.)
\begin{cor} \label{cor:inverse} Let $G,H$ be   t.d.l.c.\ groups given by computable Baire presentations, and let $\Psi \colon G \to H$ be a computable bijection that is a group homomorphism. Then $\Psi^{-1}$ is computable. \end{cor}
\begin{proof}  Since $\Psi$ is a continuous group isomorphism, it is open by the open mapping theorem.  Fix a compact open subgroup $U$ of $G$. Then $V= \Psi(U)$ is a compact open subgroup of $H$. Let $\seq{a_i}\sN i$ be a uniformly computable   sequence of left coset representatives of $U$ in $G$.  To obtain this, fix an effective numbering $\seq{\sss_k}\sN k$ of  $\Tree G$, and let  $ b_k $ be a uniformly   computable path of $\Tree G$  extending  $\sss_k$. Let $\seq{a_i}\sN i$ be the subsequence  of $\seq{b_k}\sN k$ obtained by deleting $b_k$ if $b_k U = b_\ell U$ for some $\ell < k$. Now  $\seq{\Psi(a_i)}\sN i$ is a uniformly computable sequence of left coset representatives for $V$ in $H$. By (ii) of the theorem, the sequences of (code numbers for) compact open cosets $K_i:= a_i U$ and $S_i:= \Psi(K_i)= \Psi(a_i)V$ are uniformly computable. 

By \cref{lem:lahm}(ii),  we have a ``local" computable  inverse $\Theta_i \colon S_i \to K_i$ of the restriction $\Psi_i \uhr {K_i}$, given by  uniformly partial computable functions $Q_i$ with arguments in $\Tree H \times \NN$, according to \cref{def:Comp fcn on tree}. To show that~$\Psi^{-1}$ is computable, intuitively, using  a path $h \in H$ as an oracle, compute the   $i$ such that  $h \in S_i$, and output $\Theta_i(h)$. More formally, define a  partial computable function $Q$ as follows: given $\sss \in \Tree H$ and $n \in \NN$, search for $i$ such that $[\sss]_{\Tree H} \sub S_i$. If~$i$ is found, simulate the computation for  $Q_i(\sss,n)$ and give the corresponding output. 
\end{proof} 

\subsection{Modular function, and Cayley-Abels graphs}
In Subsection~\ref{s:background} we discussed  the modular function $\Delta \colon G \to \RR^+$.  As our   second  application of \cref{prop: comp isom}, we show that  for any computable presentation, the modular function   is computable.   \begin{cor}  \label{cor:delta} Let $G$ be computably t.d.l.c.\ via  a Baire presentation $([T],  \Op, \Inv)$. Then the  modular function $\Delta \colon [T] \to \QQ^+$ is computable. \end{cor}
\begin{proof} We use the notation  of \cref{prop: comp isom}. Let $V \in \+ W$ be any subgroup.
Given $g \in [T]$,  by (ii) of the theorem compute $A=g V$. Compute $U = A \cdot A^{-1} \in \+ W$ such that  $A $ is a right coset of $U$, and hence $A = Ug$. For any left Haar measure $\mu$ on $G$, we have \bc $\Delta(g)= \mu(A)/\mu(U)=  \mu(V)/\mu(U)$.  \ec  By \cref{rem:Haar computable} we can choose $\mu$ computable; so this suffices to determine $\Delta(g)$.  
 \end{proof}

Our third application of the theorem is to show computability of the Cayley-Abels graphs, discussed in the introduction.   Let $G$ be a t.d.l.c.\ group that is   compactly generated, i.e., algebraically generated by a compact subset. Then there is a compact open subgroup $U$, and a     set $S = \{ s_1, \ldots, s_k\} \sub G$ such that $S = S^{-1}$ and $U \cup S$ algebraically generates $G$. The \emph{Cayley-Abels graph} $  \Gamma_{S,U}= (V_{S,U}, E_{S,U})$ of $G$ is given as follows. The  vertex set $V_{S,U}$ is the set $L(U) $ of left cosets of $U$, and the edge relation is   
\bc $E_{S,U} = \{ \la gU, gsU \ra \colon \, g \in G, s \in S\}$. \ec
Some background and      original references are given  in Section 5 of \cite{Willis:17}. For more detailed background see Part~4 of \cite{Wesolek:18}, or \cite[Section~2]{Kron.Moller:08}.  If $G$ is discrete (and hence finitely generated), then $\Gamma_{S,\{1\}}$ is the usual Cayley graph for the generating set~$S$.  Any two  Cayley-Abels graphs of $G$  with the shortest path distance   are quasi-isometric. Here a quasi-isometry between metric spaces $(X,d_X)$ and $(Y,d_Y)$  is a map $\psi \colon X\to Y$ such that for some $k,c \in \NN^+$, one has $\forall y \in Y \exists x \in X \, d_Y(\psi(x), y ) \le c$, and   $\forall x, x' \in X [ \frac 1 k d_X(x,x')-c  \le d_Y(\psi(x), \psi(x')) \le k d_X(x, x') +c] $.  Also see \cite[Definition~3]{Kron.Moller:08} or \cite{Wesolek:18}.

\begin{thm} \label{prop:CA graph} Suppose that $G$ is computably t.d.l.c.\ and compactly generated.   \bi \item[(i)] Each Cayley-Abels graph $\Gamma_{S,U}$ of $G$ has a computable copy $\+ L$. Given a Haar computable copy $\+ W$ of the meet groupoid $\+ W(G)$, one  can obtain $\+ L$ effectively from    $U \in \+ W$ and the left cosets $C_i = s_iU$, where $\{ s_1, \ldots, s_k\} $ is as above.   

\item[(ii)] If $\Gamma_{S',V}$ is another Cayley-Abels graph obtained as above, then $\Gamma_{S,U}$ and  $\Gamma_{S',V}$ are computably quasi-isometric. 

\item[(iii)] Given a computable Baire presentation of $G$ based on a tree $[T]$, let $\+ W= \+ W_{\text{comp}}(G)$  be the computable copy of its  meet groupoid  as in  \cref{def:Wcomp}. Then  the left action  $[T] \times \+ L  \to \+ L$ is also computable.   \ei \end{thm} 
\begin{proof} 
 (i) For the domain of the computable copy $\+ L$,   we take the computable set of left cosets of $U$.    We show that the edge relation is first-order definable from the parameters in such a way that it can be verified to be computable as well. 

Let $V_i = C_i \cdot   C_i^{-1}$ so that $C_i$ is a right coset of $V_i$. Let $V =  U \cap \bigcap_{1 \le i \le k} V_i$. To first-order  define $E_\Gamma$ in $\+ W$ with the given parameters, the idea is to replace the elements  $g$ in the definition of $E_\Gamma$ by left cosets $P$ of $V$, since they are sufficiently accurate approximations to $g$.
 It is easy to verify that $\la A, B \ra \in E_\Gamma$ $\LR$ 
 \bc $\ex i \le k \ex P \in L(V) \ex Q \in L(V_i) \, [ P \sub A \lland P \sub Q \lland B = Q \cdot C_i]$,  \ec
 where $L(U)$ denotes the set of left cosets of a subgroup $U$: For the implication ``$\LA$", let $g\in P$; then we have $A = gU$ and $B = gs_i U$.  For the implication ``$\RA$", given $A = gU$ and $B= gs_i U$, let $P\in L(V)$ such that $g \in P$. 
 
 We verify that the edge relation  $E_\Gamma$ is computable. Since $\+ W$ is Haar computable, by the usual enumeration argument we can obtain  a strong index for the set of  left cosets of $V$ contained in $A$. Given  $P$ in this set and $i \le k$, the left coset  $Q= Q_{P,i}$ of $V_i$ in the expression above is unique and   can be  determined effectively. So we can test whether $\la A, B \ra \in E_\Gamma$ by trying all $P$ and all $i\le k$ and checking whether $B = Q_{P,i} \cdot C_i$. 

 It is clear from the argument that we obtained $\+ L$ effectively from the parameters $U$ and $C_i$.  
 
 \smallskip
 
\n (ii)   First suppose that $V \sub U$. There is a computable map $\psi \colon L(U) \to L(V)$ such that $\psi(A) \sub A$ for each $A\in L(U)$. The proof of \cite[Thm.\ 2$^+$]{Kron.Moller:08} shows that $\psi \colon \Gamma_{S,U} \to \Gamma_{S',V}$ is a quasi-isometry. In the general  case, let $R\sub G$ be a finite symmetric set such that $(U \cap V) \cup R$ algebraically generates $G$. There are  computable quasi-isometries $\phi \colon \Gamma_{S,U} \to \Gamma_{R, U \cap V}$ and  $\psi \colon \Gamma_{S',V} \to \Gamma_{R, U \cap V}$ as above. There is a computable quasi-isometry $\theta: \Gamma_{R, U \cap V} \to  \Gamma_{S',V}$ inverting $\psi$: given a vertex $y\in L(U \cap V)$, let    $x= \theta(y)$ be a vertex in $L( V)$ such that  $\psi(x)$ is at distance at most $c$ from $y$, where $c$ is a constant for  $\psi$  as above.   Then $\theta \circ \phi$ is a quasi-isometry as required. 
  (iii)    follows immediately from  Theorem~\ref{prop: comp isom}(ii). 
\end{proof}

 \subsection{Algorithmic properties of the scale function} \label{s:scale}
 
In Subsection~\ref{s:background} we discussed  the scale function $s \colon G \to \NN^+$ for  a t.d.l.c.\ group~$G$,   introduced by  Willis~\cite{Willis:94}. Recall that  for a compact  open subgroup~$V$ of $G$ and an element $g\in G$ one defines $m(g,V) = |gVg^{-1}\colon V \cap gVg^{-1}|$,   and   
\bc $s(g)= \min \{ m(g,V)\colon \,  V \text{ is a compact open subgroup}\}$. \ec  
 Willis    proved  that the scale function is continuous,  where $\NN^+$ carries the discrete topology. He  introduced the relation that  a compact open subgroup $V$ is  \emph{tidy for~$g$}, and showed that this condition is equivalent to being   minimizing for $g$ in   the sense that $s(g)= m(g,V)$. M\"oller  \cite{Moeller:02} used graph theoretic methods to   show  that  $V$ is minimizing for $g$ if and only if  $ m(g^k, V) = m(g,V)^k$ for each $k\in \NN$. He also derived the ``spectral radius formula": for any compact open subgroup $U$, one has $s(g)= \lim_k m(g^k, U)^{1/k}$. 
 
 For  this section, fix a  computable Baire presentation $([T], \Op, \Inv)$  of a t.d.l.c.\ group $G$ as in Definition\ \ref{def:main1}.   Let $\+ W = \+ W_{\text{comp}}(G)$ be the   Haar computable copy of $ \+ W(G)$ given by \cref{def:Wcomp}. Recall that  the domain of~$\+ W$ is a computable set~$D$.  Via $\+ W$,  we can  identify compact open cosets of $G$ with natural numbers. 
 The following is immediate from  \cref{thm:main2} and   \cref{prop: comp isom}.
\begin{fact}  \label{fact:mg}The function $m \colon [T]  \times D \to D$ (defined to be $0$ if the second argument  is not a subgroup) is computable. \end{fact}

 
Here and in Section~\ref{s:scale noncomp}, we will  study whether   the scale function, seen as a function  $s\colon [T]\to \NN$,   is computable in the  sense of   \cref{def:computable 1}.  We note that neither M\"oller's  spectral radius formula, nor the tidying procedure of Willis (see again \cite{Willis:17}) allow  to compute the scale  in our sense.
The scale  is computable iff one can algorithmically decide whether a subgroup is minimizing:
\begin{fact} The scale function  on $[T]$ is computable {iff} the  following function $\Phi$ is computable in the sense of Definition\ \ref{def:computable 1}: if   $g \in [T] $ and $V$ is a compact open subgroup of $G$, then $\Phi(g,V)=1$ if $V$ is minimizing for $g$;  otherwise $\Phi(g,V)=0$.      \end{fact}  
\begin{proof} \rapf An  oracle Turing machine with oracle $g$ searches for the first  $V$ that is minimizing  for $g$, and outputs $m(g,V)$. 

\lapf For  oracle $g$, given input $V$ check whether $m(g,V) = s(g)$.  If so output $1$, otherwise $0$. \end{proof}

We next provide a fact restricting  the complexity of the scale function.  We     say that a function $\Psi: [T] \to \NN$ is \emph{computably approximable from above} if there is a computable function $\Theta: [T] \times \NN \to \NN$ such that $\Theta(f ,r) \ge \Theta(f ,r+1)$ for each $f\in[T],r\in \NN$, and   \bc $\Psi(f)= k $ iff  $\lim_r \Theta(f ,r) =k$.  \ec
 \begin{fact} The scale function   is   computably approximable from above. \end{fact} 
 \begin{proof} Let $\Theta(f,r)$ be the minimum value of $m(f,s)$ over all $s \le r$.  
 %
\end{proof}
The following example is well-known (\cite[Example 2]{Willis:17}); we include it to show that our framework is adequate as a general background for case-based approaches to computability for  t.d.l.c.\ groups  used in earlier works.
 \begin{example}[with Stephan Tornier] For $d \ge 3$, the scale function on $\Aut(T_d)$ in the computable presentation of \cref{ex:Td} is computable. \end{example}
 \begin{proof}
 An automorphism $g$ of $T_d$ has exactly one of three types (see  \cite{Figa.Nebbia:91}):  
 \begin{enumerate}
\item $g$ fixes a vertex $v$: then $s(g)= 1$ because $g$ preserves the stabilizer of $v$, which is a    compact open subgroup. 

\item $g$ inverts an edge: then $s(g)= 1$ because $g$ preserves     the set-wise stabilizer of the set of  endpoints of this edge.

\item $g$ translates along an axis (a subset of $T_d$ that is a homogeneous  tree of degree $2$): then $s(g) = (d-1)^\ell $ where  $\ell$ is the  length. 
To see this, for $\ell=1$   one uses as a minimizing subgroup  the compact open subgroup of automorphisms that fix two given adjacent vertices on the axis. For $\ell>1 $ one uses that $s(r^k)= s(r)^k$ for each $k$ and $r \in \Aut(T_d)$; see again~\cite{Willis:94}. 
\end{enumerate}
The oracle machine, with  a path corresponding to   $g \in \Aut(T_d)$ as an oracle, searches in parallel  for a witness to (1), a witness to (2), and a sufficiently long piece of the axis in (3) so that the shift becomes visible. It then outputs the corresponding value of the scale.
 \end{proof}

\section{Closure properties of the class of computably t.d.l.c.\ groups} \label{s:closure}
All computable presentations in this section will be Baire presentations (see \cref{def:main1}), and we will usually   view a t.d.l.c.\  group $G$  concretely as  a    computable Baire presentation. Extending the previous notation in the setting of   closed subgroups of $\S$, by $\Tree G$ we denote the c.l.c.\  tree underlying this computable Baire presentation.
The following is immediate. 
\begin{fact}[Computable closed subgroups]\label{fact:immediate} Let $G$ be a computably   t.d.l.c.\ group. Let $H$ be a closed subgroup of $G$,  and note that   $\Tree H$  is a subtree of $\Tree G$. Suppose that   $\Tree H$ is  computable, and hence c.l.c.    Then $H$ is computably t.d.l.c.\ via the Baire presentation based on $\Tree H$,  with   the operations of $G$  restricted  to~$H$. \end{fact}    
For instance, consider the closed subgroups  $U(F)$ of $\Aut(T_d)$ introduced by Burger and Mozes~\cite{Burger.Mozes:00}, where $d \ge 3$ and $F $ is a subgroup of  $S_d$. By \cref{ex:Td} together with the preceding fact, each group $U(F)$ is computably t.d.l.c.
For another  example, consider the computable Baire presentation of $\SL_2(\QQ_p)$   given by   \cref{prop: SL2}. Let $S$ be the c.l.c.\  subtree of $T$ whose  paths describe    matrices of the form  $\begin{pmatrix} r & 0\\0 & s \end{pmatrix}$ (so that $s= r^{-1}$). This yields   a computable Baire presentation of the  group $(\QQ_p^*, \cdot)$. 

\begin{proposition} \label{ex:GL2} For each prime $p$ and $n \ge 2$, the   group 
 $\GL_n(\QQ_p)$ is  computably t.d.l.c.
\end{proposition} 
\begin{proof} 
We   employ  the embedding $F\colon GL_n(\QQ_p) \to SL_{n+1}(\QQ_p)$ which extends  a matrix $A$ to the matrix $B$ where the new row and new column vanish except for the diagonal element (which    necessarily  equals $(\det A)^{-1}$).  Clearly there is a c.l.c.\ subtree $S$ of    the c.l.c.\  subtree of $T$  in  \cref{prop: SL2} for $n+1$ such that $[S] = \range (F)$. Now we apply \cref{fact:immediate}. \end{proof}

A further construction staying within the class of t.d.l.c.\ groups is the semidirect product based on a continuous action. In the effective setting, we use  actions that are computable in the sense of Section~\ref{s:comp notions}. For   computable actions in the more  general context of Polish groups, see  \cite{MeMo}.
\begin{prop}[Closure under computable semidirect products] \label{prop:sdprod}
Let $G, H$ be computably  t.d.l.c.\ groups. Suppose $\Gamma \colon G \times H \to H $ is a computable function (as in \cref{def:Comp fcn on tree})  that specifies an action of $G$ on $H$ via topological automorphisms. Then the topological semidirect product $L=G \ltimes_\Gamma H$ is computably t.d.l.c. \end{prop}

\begin{proof} Let $T$ be the tree obtained by   interspersing strings of the same length from the trees of $G$ and $H$, i.e.\ 
\bc $T= \{ \sss \oplus \tau \colon \, \sss \in \Tree G \lland \tau \in \Tree H\}$. \ec
It is clear that $T$ is a c.l.c.\ tree. Via the  natural bijection \bc $[T] \to [\Tree G] \times [ \Tree H]$, \ec  one  can write elements of $L$ in the form $\la g, h \ra$ where $g\in [\Tree G]$ and $h \in [\Tree H]$. 

By the standard  definition of a semidirect product (\cite[p.\ 27]{Robinson:82}), writing the operations for $G$ and $H$ in the usual group theoretic way, we have 
\begin{eqnarray*} \Op ( \la g_1, h_1 \ra, \la g_2, h_2\ra) &=& \la g_1 g_2, \Gamma(g_2, h_1) h_2\ra \\
\Inv( \la g, h \ra)  & = & \la g^{-1}, (\Gamma (g^{-1}, h))^{-1}\ra . \end{eqnarray*}
This shows that  $\Op$ and $\Inv$ are computable, and hence   yields a computable Baire presentation $([T], \Op, \Inv)$ for $L$. 
\end{proof} 

%

%
%
\begin{remark}  The foregoing proposition leads to a different proof that  $\GL_n(\QQ_p)$ is  computably t.d.l.c.\ (\cref{ex:GL2}).  To simplify notation we let $n=2$; it is not hard to generalise the argument below to a general $n$. As  mentioned after \cref{fact:immediate}, there is a computable Baire presentation of    $(\QQ_p^*, \cdot)$. We have $\GL_2(\QQ_p) = \QQ_p^* \ltimes_\Phi \SL_2(\QQ_p)$ via the inclusion embedding of   $\SL_2(\QQ_p)$, the embedding $q \to \begin{pmatrix} q & 0\\0 & 1\end{pmatrix}$ of $\QQ_p^*$, and the computable action $\Phi (q, \begin{pmatrix} a & b\\c & d \end{pmatrix}) = \begin{pmatrix} a & q^{-1} b\\qc & d \end{pmatrix}$.

Note that the two    computable Baire presentations of $GL_2(\QQ_p)$ obtained above are    computably isomorphic: one maps $(q, \begin{pmatrix} a & b\\c & d \end{pmatrix})$ to $\begin{pmatrix} qa & q b & 0\\c & d & 0 \\ 0 & 0 & q^{-1}\end{pmatrix}$. \end{remark}
 
 Given a sequence of t.d.l.c.\ groups $(G_i) \sNp i $,   the direct  product $\prod  \sNp i G_i$ is not t.d.l.c.\  in general. In \cite[Definition\ 2.3]{Wesolek:15} a local direct product  is described that retains the property of being t.d.l.c. This construction depends on the choices of    compact open subgroups $U_i$ of  $G_i$, for each~$i$: let $G= \bigoplus_i (G_i, U_i) $ consist of the elements   $f \in \prod_i G_i$ such that $f(i) \in U_i$ for sufficiently large $i$.  
 We have $G = \bigcup_{k\in \NN} H_k$ where $H_0= \prod_{i\in \NN^+} U_i$ and for $k>0$, $H_k= G_1 \times \ldots \times G_{k} \times \prod_{ i > k} U_i$.  The groups  $H_k$ are equipped with the product topology. A set $W \sub G$ is declared open if $W \cap H_k$ is open for each $k$. 
(In particular, $\prod_i U_i$ is a compact  open subgroup.)
 
  Fix a computable bijection $\la ., .\ra \colon \NN \times \NN^+ \to \NN$ such that $\la a,b\ra \ge \max (a,b)$. For a string $\sss \in \NN^*$ and $i>0$, by $\sss^{(i)}$ we denote the string $\tau$ of maximum length such that $\tau(k)= \sss(\la k, i\ra)$ for each $k < |\tau|$.  Similarly, for $f \colon \NN \to \NN$ we define $f^{(i)}$ to be the function such that $f^{(i)}(k)= f(\la k, i\ra) $ for each $k$.  Given uniformly computable  subtrees $B_i$ of $\NN^*$, by $B=\prod_i B_i$ we denote the computable tree $\{ \sss \colon \fa i  [ \sss^{(i)} \in B_i]\}$. Note that $[B]$ is canonically homeomorphic to $\prod_i [B_i]$ via $f \to (f^{(i)})\sN i$. So we can specify a path $f$ of $B$ by specifying all the $f^{(i)}$.

\begin{prop}[Local direct products] \label{prop:direct products} Let $(G_i)\sNp i$ be computable Baire presentations $(T_i, \Op_i, \Inv_i)$ uniformly in $i$, and suppose that there is $k \in \NN^+$ such that each $T_i$ is c.l.c.\ via $k$. Suppose further  that  $U_i$ is  a compact open subgroup of~$G_i$, uniformly in $i$.  Then $G=\bigoplus \sNp i  (G_i, U_i) $ is computably t.d.l.c. \end{prop} 


 \begin{proof}    The uniformity hypothesis on the $U_i$ means that  there is a computable function $q$ such that $\+ K_{q(i)}^{T_i}= U_i$. We use these data to build a computable Baire presentation $([T], \Op, \Inv)$ of  $G$. We  aim at defining uniformly c.l.c.\ trees $V_r$ such that as topological spaces, $[V_0]$ is homeomorphic to   $H_0$  defined above, and for $r>0$, $V_r$ is homeomorphic to $H_r - H_{r-1}$.  All these homeomorphism are canonically given, as can be seen during the construction of the $V_r$. 
Our Baire presentation of $G$   will~be   based on the c.l.c.\ tree
 \bc $T = \{ r \ape \sss \colon r \in \NN \lland \sss\in V_r\}$. \ec 
    It is easy to verify  that $[T]$ is homeomorphic to $G$, using that  the $H_r$ are open subgroups of $G$.

Towards  defining  the trees $V_r$,  let $R_i$ be the subtree of $T_i$ without leaves such that $[R_i]= U_i$, and let  $S_i$ be the subtree of $T_i$ such that $[S_i]= [T_i]-U_i$. 
 \begin{claim}  The trees $R_i, S_i$ are  c.l.c.\ trees uniformly in $i \in \NN^+$.  \end{claim}
 To check this, it suffices to show that these trees  are uniformly computable. Given $i$, let $F_i \sub \NN^*$ be  the finite set with strong index $q(i)$. Note that $R_i$ consists of the strings compatible with a string in $F_i$, which is a decidable condition uniformly in $i$.  To determine  whether $\tau \in S_i$, \bi \item[(i)] first check whether some prefix of $\tau $ is in $F$; if so, answer ``no". 
 \item[(ii)] Otherwise, using the conditions defining  c.l.c.\ trees    check whether $[\tau]_{T_i}$ is compact; if not answer ``yes". If so,  check whether $\tau$  has an extension longer that any string in $F_i$ that is not in $R_i$; if so answer ``yes", otherwise ``no".  \ei This shows the claim. 
\smallskip

 Now let $V_0= \prod_i R_i$, and for $r >0$ let 
 \bc $V_r= \prod_{1 \le i<r} T_i \times S_{r} \times \prod_{i> r} R_i$, \ec
 interpreted as subtrees of $\prod_i T_i$ in the obvious way. (So, $f$ is a  path of $V_r$ iff  $f^{(i)}$ is a  path of $T_i$  for $1\le i<r$,  $f^{(r)}$ is  \emph{not} in $U_{r}$, but   $f^{(i)}$ is in $U_i$ for $i>r$.) It is easy to check that the trees $V_r$ are uniformly c.l.c.

Uniformly in $r$, on $[V_r]$ we have a computable function $L_r$ given by  $L_r(f)^{(i)}= \Inv_i(f^{(i)})$. So there is a computable function $\Inv$ on $[T]$ given by  \bc $\Inv(r \ape f) = r\ape L_r(f)$. \ec
Next, for $r,s \in \NN$ and $r \ape f, s \ape g \in [T]$,  let \bc $\Op (r \ape f, s \ape g) = t \ape h$,  \ec
where $h$ is the function given by $h^{(i)} = \Op_i(f^{(i)}, g^{(i)})$, and $t \le \max (r,s)$ is the least number  such that $t=0$, or $t>0$ and $h^{(t)}\uhr { \max(F_{t})}\in S_{t}$. That is, we compute the binary group operation componentwise, and then check which tree $V_t$ the overall result is a path of; this   can be done because from $\max(r,s)$  on,  the component of the  result will be in the  relevant compact open subgroup.

It should be clear that,  via the homeomorphism of the spaces $[T]$ and $G$ outlined above, $([T], \Op, \Inv)$ is a computable Baire presentation of $G$.
 \end{proof}
 We note that the hypothesis that $T_i$ is c.l.c.\ via a fixed $k$ is one of notational convenience; what matters is that given $i$ we  can compute $\ell_i$ such that $T_i$ is c.l.c.\ via~$\ell_i$. This is so because, given a tree $T$ that is c.l.c.\ via $\ell$,  we can uniformly transform it into an equivalent tree $\wt T$ that is c.l.c.\ via $1$, by ``skipping" the levels $1, \dots, \ell-1$.

  \smallskip

The hardest result in this section will be tackled last:  being computably t.d.l.c.\ is preserved under taking quotients by computable normal closed subgroups. As an application we will  show that the  groups $PGL_n(\QQ_p)$ are computably t.d.l.c. For $n=2$ these groups have been the subject of much research; for instance,  they are homeomorphic to closed subgroups of $\Aut (T_{p+1})$ as shown by Serre \cite[Section~II.1]{Serre:80}.

  First we need some notation and preliminaries.  The variables  $\aaa, \beta$ etc.\ will range over  strings in $\NN^*$ without repetitions. 
The variables  $P, Q, R $  range over permutations of $\NN$.   Recall from Section~\ref{ss:Sinf} that in our setup the elements of $\S$ have the form $P\oplus P^{-1}$ where $P$ is a permutation of~$\NN$.  It appears that a crucial ``finitization" argument, \cref{cl:crucial}, is best shown in the setting of true permutations, rather than elements of $\S$ in our  setup. So we need two lemmas allowing us to pass back and forth between the two.
 
  We adopt the setting of Theorem~\ref{prop: comp isom}: let $G=([T], \Op, \Inv) $ be a computable Baire presentation,   let $\+ W$ be a Haar computable copy of  $\+ W(G)$  with domain $\NN$, and let  $\wt G$ be as detailed there; in particular,    $S= \Tree {\wt G}$ is  c.l.c.\ via $k=2$, so that   $[\sss]_S$ is compact for each string $\sss$ of length $\ge 2$.  
 Let  \bc $S_0= \{\ES\} \cup \{ \aaa \colon \, \ex \beta [ |\aaa| = |\beta| > 0 \lland \aaa \oplus \beta \in S\}$. \ec
 The first lemma verifies that $S_0$ is computable, and shows that from    a nonempty string in $S_0$ one can compute  the set of elements of $[S]$ with first component extending it. 
\begin{lemma} \label{lem: prelim1} (i) $S_0$ is a  computable tree. 

\n (ii) There is a computable function $B\colon S_0 - \{\ES\} \to \NN$ such that \bc $\+ K_{B(\aaa)}= \{ f \in [S] \colon  \aaa \prec f_0 \text{ where } f= f_0 \oplus f_1\}$. \ec
\end{lemma}
\begin{proof}  (i)  We first show that there is a computable bound $H(k)$ on  $\beta(k)$ that is uniformly obtained from $\aaa$.  Since $\aaa$ is nonempty, we have  $\aaa(A)= B$ for some $A,B \in \+ W$. Hence $f(AA^{-1})= BA^{-1}$ for any $f \in \wt G$ such that  $\aaa\prec f$.  By Lemma~\ref{lem: sst}  with $U=  AA^{-1} $,  $L = BA^{-1}$ and $F$ the coset denoted by $k$, we obtain  a bound $H(k)$ as required. This shows that from $\aaa$ one    can compute a finite set of possible candidates for $\beta$.  So $S_0$ is computable.

\smallskip

\n (ii)  Given a nonempty  $\aaa \in S_0$, by the argument above we can compute  a strong index $B(\aaa) $ for the set of strings of the form $ \aaa \oplus \beta  $ in $S$. \end{proof} 

The second lemma takes a   string  in $S^{[\ge 2]}$ and writes the paths of  $S$ extending it in terms of a finite set of strings in $S_0$.
\begin{lemma} \label{lem:prelim2} Given $\sss \in S$ such that $\sssl \ge 2$, one can compute (a strong index for) a finite set $F \sub S_0$ such that \bc $ [\sss]_S= \bigcup\{ \+ K_{B(\aaa)} \colon \, \aaa \in F\}$. \ec \end{lemma} 

\begin{proof} Let $n = \max (\sss)+1$. Let $F$ consist of the strings $\aaa \in S_0$ of length $n$ such that viewed as an injection, $\aaa$ extends the injection $\aaa_\sss$ associated with $\sss$ as in (\ref{eqn:inj}).    Since $\sssl \ge 2$,  by Lemma~\ref{lem: sst} we can compute a strong index for $F$.
If  $\aaa \in F$ and $P \succ \aaa$, then $P \oplus P^{-1} \in [\sss]_S$, because $P$ extends $\aaa_\sss$. Conversely, if $P \oplus P^{-1} \in [\sss]_S$, then $P \uhr n \in F$. \end{proof}

 \begin{thm}  \label{thm:closure normal} Let $G$ be computably t.d.l.c. Let $N$ be a closed normal subgroup of~$G$ such that $\Tree N$ is a computable  subtree of $\Tree G$. Then $G/N$ is computably t.d.l.c. \end{thm}
%
 \begin{proof} We     continue to adopt the setting of   \cref{prop: comp isom}.  Recall that  $S = \Tree { \wt  G}$, and $S_0= \{\ES\} \cup \{ \aaa \colon \, \ex \beta [ |\aaa| = |\beta| > 0 \lland \aaa \oplus \beta \in S\}$.     Let $M= \Phi(N)$ where $\Phi$ is the bicomputable  homeomorphism $G \to \wt G$ established in \cref{prop: comp isom}.  Note that  $\Tree {M}$ is a computable  subtree of $S$; given $\tau \in S$, one can search for a string $\sss \in \Tree G$ such that $P_\Phi(\sss) \succeq \tau$; then $\tau \in \Tree {M}$ iff $\sss \in \Tree N$.
 
 We will build a Haar computable copy $\+ V$ of $\+ W(\wt G/M)$. We use that each compact open subset of $\wt G/M$ has the form $M \+ K$ where $\+ K$ is a compact open subset of $\wt G$.
  In the first step,  we will show that the   pre-ordering ``$\+ K \sub M \+ L$" is decidable, where  $\+ K, \+ L$ range over       compact open sets.   In the second step,  we will use as the domain of $\+ V$ the least numbers  in  the  classes of the computable equivalence relation associated with this pre-ordering.

  We may assume that $\+ K= [\sss]_S$ for some string $\sss\in S$ of length at least $2$. So the following suffices for the first step:

 \begin{lemma} \label{lem:stupid} Given strings $\sss, \tau_1, \ldots, \tau_r \in S$ of length at least $2$, one can decide whether the inclusion $[\sss]_S \sub \bigcup_i M [\tau_i]_S$ holds. \end{lemma}

 \n To verify the lemma, let 
   $  M_0 = \{ P  \colon P  \oplus P^{-1} \in M\}$. For each $\aaa \in S_0$, let 
\bc  $\underline  \aaa = \{P \in [S_0] \colon \, \aaa \prec P\}$. \ec
Recall  here that $P$  is a permutation of $ \NN$, not merely a path of $S_0$ (which in general could fail to be onto). Let \bc $T_0 = \{ \aaa \colon \, \ul \aaa \cap M_0 \neq \ES\}$. \ec 
By Lemma~\ref{lem:prelim2}, it is sufficient to decide whether a  version of the inclusion holds that only refers to the permutations,   not directly to  their inverses.
 \begin{claim} \label{cl:crucial} For $\aaa, \beta_1, \ldots, \beta_k \in S_0$ one can decide whether $ \ul \aaa \sub \bigcup_i  M_0 \ul {\beta_i}$. \end{claim}
\n We may assume that $|\aaa| \ge m:= 1+  \max_i  \max (\beta_i)$.  We show that  $m$ is the maximum  height on $T_0$  relevant for this inclusion:  for each $i$, 
\begin{equation} \label{eq: fffuuuccc} \ex Q  \in   M_0 \,  [ \ul a \sub Q  \circ \ul {\beta_i} ] \LR  \ex \eta \in T_0  \, [  |\eta| =m \lland  \ul \aaa \sub \ul \eta \circ \ul {\beta_i} ].  \end{equation}
For  the implication ``$\RA$", simply let $\eta = Q \uhr m$. 
For the implication ``$\LA$", fix  a permutation $Q \succ \eta$ such that  $Q \in  M_0$. Given $P \in \underline \aaa$, we can choose $R \in \ul \eta$ and $R' \in \ul{\beta_i}$ such that   $P= R \circ R'$. Since $R, Q \succ \eta$ and $|\eta|= m> \max (\beta_i)$, we have $Q^{-1} \circ R \circ R' \succ \beta_i$. So $P \in M_0 \circ  \ul{\beta_i}$ via $Q$. This verifies  the equivalence (\ref{eq: fffuuuccc}).

Using  (\ref{eq: fffuuuccc}), if $ \ul \aaa \sub \bigcup_i  M_0 \circ \ul {\beta_i}$, then because $|\aaa|> \max_i |\beta_i|$, 
we have  $\ul \aaa \sub M_0 \circ \ul  \beta_i$ for some  single $i$.  If $\eta$ works on the right hand side of (\ref{eq: fffuuuccc}) then $\eta(\beta_i(0))= \aaa(0)$. So by   Lemma~\ref{lem: sst} again, there is a computable bound $H(k)$ on  $\eta(k)$ that is uniformly obtained from the values   $\aaa(0), \beta_1(0), \ldots, \beta_k(0)$. This shows that   one can compute a strong index for a finite set containing all  potential witnesses $\eta$ on the right hand side of~(\ref{eq: fffuuuccc}). For each such $\eta$, using Lemma~\ref{lem: prelim1}(ii), it is equivalent to decide whether $\+ K_{B(\aaa)}  \sub \+ K_{B(\eta)}   \+ K_{B( \beta_i)}$, which can be done using Lemma~\ref{lem: decide inclusion}(ii). This shows the claim and hence verifies \cref{lem:stupid}.
%

\smallskip

By the lemma (and the discussion preceding it), the  equivalence relation on $\NN$ given  by
 \bc $A \sim B$ if $A M = B M$ \ec 
  is computable; recall here that $\+ W$ has domain $\NN$. 
As the domain $\+ D$ of the computable copy $\+ V$ of $\+ W(\wt G/M)$, we use    the computable set of least elements of equivalence classes.

We think of an  element $A$ of $\+ D$ as denoting  the compact open coset $A M$ of $\wt G/ M$. Given $A,B \in \+ W $ we   have $(A M)^{-1} = A^{-1} M$ and $(AM ) (B M)= (AB) M$. In particular, one can decide whether $A M $, viewed as a subset of $\wt G/M$,  is a left   coset of   a subgroup of $\wt G/ M$ that $B M$ is a right coset of.   So by \cref{lem:basic operations} the groupoid operations  are computable on $D$.

For the computability of the meet operation, suppose  $A,B \in \+ W$ are given.   One has  $AN \cap BN = (A \cap BN) N$. Suppose $A $ is a left coset of the subgroup $U$ and $B$ is a left coset of the subgroup $V$. Then $A \cap BN$ is a left  coset of the compact open subgroup $U \cap VN$.
  Note  that 
\bc $ A \cap BN = \bigcup \{ L \colon \, L \subseteq A \lland L \text{ is left coset of } U \cap V \lland  L \sim B\} $.\ec
(The inclusion $\supseteq$ is trivial. For $\sub$, if $x \in A \cap BN$, then $ x \in L$ for some left coset of $U \cap V$. Since $x = bn$ for some $b\in B, n \in M$ we have $LN = BN$.) So one can compute $C \in \+ W$ such that   $C= A \cap BN$ using \cref{lem: comp index tree}. Then one outputs the element $C'$ of $\+ D$  such that $C' \sim C$.

Finally, to show that $\+ V$ is Haar computable, note that $|UN : UN \cap VN|= |U \colon U \cap VN|$. By the above (in the special case that $A$ and $B$ are subgroups) one can compute $U \cap VN \in \+ W$. So one  can compute the index using that $\+ W$ is Haar computable. 
 \end{proof}
%

 \begin{example} For each prime $p$ and each $n \ge 2$, the group $\PGL_n(\QQ_p)$ is computably t.d.l.c. \end{example} 
 \begin{proof} 
 We use   the   computable Baire presentation $(T,   \Op, \Inv) $ of $\GL_n(\QQ_p)$ obtained in \cref{ex:GL2}.  In this presentation,  the centre $N$ of $\GL_n(\QQ_p)$ is given by the diagonal $(n+1) \times (n+1)$ matrices such that the first $n$ entries of the diagonal agree. So clearly $\Tree N$ is a  computable  sub-tree  of the tree $S$ in \cref{ex:GL2}. Hence we can apply \cref{thm:closure normal}.
 \end{proof}

\section{Uniqueness of computable  presentations} \label{s:autostable} 
As discussed in Subsection~\ref{ss:auto}, in   computable  structure theory a countable structure is called  {autostable}, or computably categorical, if  it has a computable copy, and all  its computable copies are computably isomorphic.  We adapt this notion to the present setting. 

\begin{definition} A  computably t.d.l.c.\ group $G$ is called \emph{autostable} if for any two computable Baire presentations of $G$, based on    trees  $T, S\sub \NN^*$, there is a computable group homeomorphism $\Psi \colon [T] \to [S]$.   \end{definition}

Note that $\Psi^{-1}$ is also  computable  by \cref{cor:inverse}. %
For  abelian profinite groups, the notion of autostability used in \cite{Pontr} is equivalent to our definition. This follows from   the proofs of Prop.~\ref{prop:proctotdlc} and Prop.~\ref{lem:proc}, which show that in the  abelian case the correspondence between Baire presentations and   procountable  presentations is uniform and   witnessed by uniformly obtained group-isomorphisms between these presentations.  The first author  \cite[Cor.\ 1.11]{Pontr}     characterized  autostability for  abelian compact pro-$p$ groups  given by  computable procountable   presentations with effectively finite kernels:  such a group is autostable iff its Pontryagin - van Kampen dual is autostable.  
 For instance,    $(\ZZ_p, +)$ is autostable  because  its dual  is the Pr\"ufer group $C_{p^\infty}$, which is easily seen to be autostable as a countable structure (see the proof of \cref{thm:autostable} below).

We   provide a criterion for autostability, and evidence its usefulness  through various examples. 
\begin{criterion} \label{thm:compCrit} A  computably  t.d.l.c.\ group $G$  is autostable $\LR$  any two Haar computable copies of   its meet groupoid $\+ W(G)$ are computably isomorphic. \end{criterion}
We will only apply the implication ``$\LA$". However,  the converse implication  is interesting on its own right because it shows that our notion of autostability is independent of whether we use computable Baire presentation, or  computable presentations based on    meet groupoids. In fact, the proof of  the converse  implication  shows that we could also use presentations  based on closed subgroups of $\S$.

\begin{proof} 
We may assume  that $G$ itself is a computable Baire presentation.  As before,  let $\+ W = \+ W_{\text{comp}}(G)$ as in \cref{def:Wcomp}, and let   $\wt G= \+ G_\text{comp} (\+ W)$ as in \cref{def:Gof}.  Let $\Phi \colon G \to \wt G$   be the group homeomorphism given by  Theorem~\ref{prop: comp isom}.

\medskip

 \lapf   Let $H $ be a   computable Baire presentation such that $G \cong H$. Let $ \+ V= \+ W_{\text{comp}}(H)$, and let $ \wt H = \+ G_\text{comp} (\+ V)$. 
  Clearly $G \cong H$ implies $\+ W \cong \+ V$.
   Since $\+ V$ is Haar computable, by hypothesis there is a computable isomorphism $\theta \colon \+ W \to \+ V$; note that $\theta$ is a permutation of $\NN$, so $\theta^{-1}$ is also computable. 
We define a computable homeomorphism $\wt \theta \colon \wt G \to   \wt H$ as follows. Given $p= f\oplus f^{-1} \in \wt G$, let 
\begin{equation} \label{eqn:dual} \wt \theta(p) = (\theta \circ f \circ \theta^{-1} \oplus \theta \circ f^{-1} \circ \theta^{-1} ). \end{equation} 

Using that $\theta$ is an isomorphism of meet groupoids, it is easy to verify that   $\wt \theta$ is a homeomorphism. Clearly $\wt \theta$ is computable. The inverse of $\wt \theta$ is given by exchanging $\theta$ and $\theta^{-1}$ in the above, and hence  is computable as well.

 Let  $\Psi \colon H \to \wt H$ be the homeomorphisms given by  Theorem~\ref{prop: comp isom} with $H$ in place of $G$. We have a group homeomorphism  $ \Psi^{-1} \circ \wt \theta \circ \Phi \colon G \to H$, which is computable as a composition of computable maps.  Also, its inverse is computable because the inverse of $\wt \theta$ is computable. 
%
 \medskip
 
 \rapf   
 Suppose   $\+ V$ is a Haar computable meet groupoid such that $\+ W \cong \+ V$ via an isomorphism $\beta$. We need to show that $\+ W$ and  $ \+ V$ are computably isomorphic. To this end,    let  $\wt H = \+ G_\text{comp} (\+ V)$, which is a computable Baire presentation of a t.d.l.c.\ group (\cref{cl:fcuk}). Let   $\wt {\+ W} = \+ W_{\text{comp}}(\wt G)$ and $\wt {\+ V} = \+ W_{\text{comp}}(\wt H)$. Let $\aaa_\+ W \colon \+ W \to \wt {\+ W}$ and $\aaa_\+ V \colon \+ V \to \wt {\+ V}$   be the maps given by \bc $\aaa_\+ W(A_U)= \{ p \in \wt G \colon \, p_0(U)= A\}$ and $\aaa_\+ V(B_U)= \{ q \in \wt H \colon \, q_0(U)= B\}$,  \ec
 where the notation $A_U$ indicates that $A $ is a left coset of $U$, and as usual $p = p_0 \oplus p_1$, etc.  Informally, $\wt{\+ W}$ is the double dual of $\+ W$, and $\aaa_\+ W$ maps $\+ W$ into its double dual; a similar statement holds for $\wt {\+ V}$, $\+ V$ and $\aaa_\+ V$.  An argument similar to the one in the proof of \cref{lem: prelim1}(ii) shows that the maps  $\aaa_\+ W$ and $\aaa_\+ V$  are computable.

 Since  $\Phi\colon G \to \wt G$ is a group homeomorphism,  its dual $\wt \Phi \colon \+ W \to \wt {\+ W}$ is a meet groupoid  isomorphism, where  $\wt \Phi (A) = \{\Phi(g) \colon \, g \in A\}$. The  following argument  uses that $\+ W \cong \+ W(G)$ for the t.d.l.c.\ group $G$ (and hence cannot be carried out  for $\aaa_{\+ V}$).

 \begin{claim} \label{cl:aaaW}  $\aaa_\+ W$  is a meet groupoid  isomorphism. \end{claim}
It suffices to show  that $\aaa_\+ W = \wt \Phi$.  To see this, let $A_U\in \+ W$.  For the inclusion  $\wt \Phi(A) \sub \aaa_\+ W(A)$, if  $g \in A$ then clearly $  \Phi(g)(U)=gU = A$, so $\Phi(g) \in \aaa_{\+ W}(A)$. For the converse inclusion, suppose $p_0(U)= A$ for $p\in \wt G$,  and let $g = \Phi^{-1}(p) $. Then $g\in A$ because $\Phi(g)(U)= p_0(U)= A$. This verifies the claim.
 
 Let $\Gamma  \colon \wt{\+ W} \to \wt{\+ V}$ be the ``double dual" of the isomorphism $\beta\colon \+ W \to \+ V$; that is,  $\Gamma(A) = \{ \wt \beta (p) \colon \, p \in A\}$, where $\wt \beta $ is the dual of $\beta$ defined as in \cref{eqn:dual} with $\beta$ in place of $\theta$. Note that $\Gamma$ is an isomorphism of meet groupoids. 
 
 \begin{claim} The   diagram  $\xymatrix{
    \wt {\+ W }   \ar[r]^\Gamma & \wt{\+ V}  \\
    \+ W \ar[u]^{\aaa_\+ W}      \ar[r]_\beta                   & \+ V \ar[u]_{\aaa_\+ V} }$ commutes. 
    
    \n Hence  $\aaa_\+ V$ is an isomorphism of  meet groupoids. 
    \end{claim}
    To see this, if $A_U \in \+ W$, then
    \begin{eqnarray*} \Gamma(\aaa_\+ W(A)) & = &  \{ p^\beta  \colon p \in \wt G \lland \colon p_0(U)= A\} \\ 
&=&  \{ q \in \wt H\colon q_0(\beta(U)) =  \beta(A)\}  \\ &=& \aaa_\+ V(\beta(A)), \end{eqnarray*}
where $p^\beta:=\beta \circ f \circ \beta^{-1} \oplus \beta \circ f^{-1} \circ \beta^{-1}$, for $p= f \oplus f^{-1}$, similar to \cref{eqn:dual}.

 Since $G$ is autostable by hypothesis, and $\Phi \colon G \to \wt G$ is bicomputable, $\wt G$ is autostable. Since $\beta$ is an isomorphism, we have   $\wt G \cong \wt H$. Hence   there is a bicomputable  isomorphism $  \wt G \to \wt H$.  Inspecting the construction in the proof of the implication~``$\LA$" of~\cref{thm:main} shows that there is a computable isomorphism $\wt {\+ W} \to \wt {\+ V}$.  So  there is a computable isomorphism $ {\+ W} \to   {\+ V}$,  as required.
   \end{proof}

 \begin{remark}[\emph{Computable duality between t.d.l.c.\ groups and   meet groupoids}] \label{rem:duality} The post~\cite[Section~4]{LogicBlog:20}  axiomatizes the class $\mathbf M$ of countable meet groupoids  $\+ W$ that are isomorphic to $\+ W(G)$ for some t.d.l.c.\ group~$G$. Note that     \cref{def:Gof} of $\+ G_{\text{comp}}( \+ W)$  makes sense for any meet groupoid $\+ W$ with domain~$\NN$. Besides some basic algebraic axioms on meet groupoids (such as saying that different left cosets of the same subgroup are disjoint), one needs an axiom ensuring local compactness of $\+ G_{\text{comp}}( \+ W)$: there is a subgroup $K$ such that   each subgroup $U \sub K$ has only finitely many left cosets contained in $K$. Furthermore, one needs to say that the map $\aaa_{\+ W} \colon \+ W \to  \wt{\+ W}$ in \cref{cl:aaaW} is onto. In general,  this could fail:   consider the meet groupoid obtained from  a computable copy of $\+ W(\ZZ_p)$ by  deleting  all cosets of subgroup of the form $p^{2i+1} \ZZ_p$: its double dual $\+ W$ is isomorphic to  $\+ W(\ZZ_p)$.  
The required  ``completeness axiom" avoiding this situation, called CLC in \cite[Section~4]{LogicBlog:20}, states that if $N \in \+ W$ is normal  (i.e., each left coset of $N$ is also a right coset), and $\+ S \sub L(N)$ is finite and closed under inverse and product, then there is a subgroup $V \in \+ W$ such that $C \sub V \lra C \in \+ S$. (These axioms are sufficiently simple to imply that $\mathbf M$ is an arithmetical class.) Using the methods to prove \cref{thm:compCrit} one can proceed to showing that the operators $\+ W_{\text{comp}}$ and $\+ G_{\text{comp}}$,  restricted to the Haar computable meet groupoids in $\mathbf M$,  are inverses of each other. That is, composing one with the other  leads to a computable copy of the original structure that is computably isomorphic to it.  \end{remark}

\begin{theorem}  \label{thm:autostable} The computably t.d.l.c.\ groups   $\QQ_p$ and $\ZZ \ltimes \QQ_p$ are   autostable. \end{theorem}
\begin{proof}   
  In  \cref{ex:Qp} we obtained  a   Haar computable copy  $\+ W$ of  the meet groupoid $\+ W(\QQ_p)$. Recall that the elements of $\+ W$ are given as cosets  $D_{r,a}=\pi_r^{-1}(a)$ where $r\in \ZZ$, $\pi_r \colon \ZZ_p \to C_{p^\infty}$ is the canonical projection with kernel  $U_r =p^{r} \ZZ_p$,  and $a \in C_{p^\infty}$.

By the criterion above, it  suffices to show that any Haar computable copy $\wt {\+ W}$ of $\+ W(\QQ_p)$ is computably isomorphic to $\+ W$. 
  By hypothesis on $\wt {\+ W}$  there is an isomorphism $\Gamma \colon \+ W \to \wt {\+ W}$. Let $\wt U_r = \Gamma(U_r)$ for $r\in \ZZ$.  We will construct a \emph{computable} isomorphism $\Delta\colon \+ W \to \wt {\+ W}$  which agrees with $\Gamma $ on the set   $\{U_r\colon \, r\in \ZZ\}$. 
First we show that from $r$ one can   compute the subgroup $\wt U_r \in \wt {\+ W}$. 

\bi \item[(a)] If $\wt U_r$ has been determined, $r\ge 0$, compute  $\wt U_{r+1}$ by searching for  the unique subgroup in $\wt {\+ W}$   that has index $p$ in $\wt U_r$.

\item[(b)] 
If $\wt U_r$ has been determined, $r\le 0$, compute  $\wt U_{r-1}$ by searching for  the unique subgroup in $\wt {\+ W}$ such that $\wt U_r$ has index $p$ in it.  \ei

The shift homeomorphism  $S\colon \QQ_p \to \QQ_p$ is defined by $S(x)= p x$.  Note that $B \to S(B)$ is   an automorphism of the meet groupoid $\+ W$.  Using  the notation of~\cref{ex:Qp} (recalled above), for each $\aaa \in \QQ_p, r \in \ZZ$,  one has $\pi_{r+1}(S(\aaa))=\pi_r(\aaa)$, and hence  for each $a \in C_{p^\infty}$, \begin{equation} \label{eqn:SSS} S(D_{r,a})= D_{r+1, a}. \end{equation}   
We show that $S$  is definable    within $\+ W$ by an existential formula using    subgroups~$U_r$ as parameters. Recall that given a meet groupoid $\+ W$, by $L(U)$ we denote  the set of left cosets of a subgroup $U$. For $D \in L(U_r)$ we write   $D^k$ for $D \cdot \ldots \cdot D$ (with $k$~factors), noting that  this is defined,  and in $L(U_r)$.%
\begin{claim} Let $B \in L(U_r)$ and $ C\in L(U_{r+1})$. Then
 \bc $C = S(B) \LR  \ex D \in L(U_{r+1})\, [D \sub B \lland D^p = C]$. \ec \end{claim}
\n  \lapf If $x \in C$ then $x = py$ for some $y \in B$, so $x \in S(B)$.    So $C \sub S(B)$  and hence $C= S(B)$ given that $S(B)\in L(U_{r+1})$. 
 
\rapf Let $x\in C$, so $x = S(y)$ for some $y \in B$. Let $y \in D$ where $D \in L(U_{r+1})$. Then $D\sub B$. Since $D^p \cap C \neq \ES$,   these two (left) cosets of $U_{r+1}$ coincide.  This shows the claim.

  We  use this to show that the function   $\wt S=\Gamma \circ S \circ \Gamma^{-1} $ defined on $\wt {\+ W}$ is computable. Since $\Gamma( U_r)= \wt U_r$,  ($r \in \ZZ$),   $\wt S $ satisfies   the claim when replacing the $U_r$ by the $\wt U_r$.  Since the meet groupoid $\wt {\+ W}$ is computable,  given $B \in \wt {\+ W}$,     one can search $\wt {\+ W}$ for a witness $D\in L(\wt U_{r+1})$ as on the right hand side, and then output $C= \wt S(B)$. So the function  $\wt S$ is computable.

We build the computable isomorphism $\Delta \colon \+ W \to \wt {\+ W}$ in four phases.    The first three phases build a computable   isomorphism $L( U_0) \to L(\wt U_0)$, where $L(\wt U_0) \sub \wt {\+ W}$ denotes the group    of left cosets of $\wt U_0$. (This group is isomorphic to $C_{p^\infty}$, so this amounts to defining a computable isomorphism between two computable copies of $C_{p^\infty}$.)  The last phase extends this isomorphism to all of $ \+  W$, using that $\wt S$ is an automorphism of $\wt {\+ W}$. 

%

For $q \in \ZZ[1/p] $ we write  $[q]= \ZZ + q \in C_{p^\infty}$.   We     define $\wt D_{r,[q]}= \Delta(D_{r,a})$ for $r \in \ZZ, q \in \ZZ[1/p] $

\bi \item[(a)] Let $\wt D_{0, [p^{-1}]}$ be an element of order $p$ in $L(\wt U_0) $. 
\item[(b)] Recursively,  for  $m>0$ let  $\wt D_{0, [p^{-m}]}$ be an element of order $p^m$ in $L(\wt U_0) $ such that $(\wt D_{0, [p^{-m}]})^p = \wt D_{0, [p^{-m+1}]}$.   
 \item[(c)]  For $a = [kp^{-m}]$ where $0 \le k <p^m$ and $p  $ does not divide $k$,  let $\wt D_{0, a} = (\wt D_{0, [p^{-m}]})^k$.
\item[(d)] For $r \in \ZZ -\{0\}$ let  $\wt D_{r,a} =\wt S^r (\wt D_{0,a})$.\ei 
One can easily verify that $\Delta \colon \+ W \to \wt {\+ W}$ is  computable  and preserves the   meet groupoid operations.   To  verify that $\Delta$ is onto, let  $B \in \wt {\+ W}$. We have $B \in L(\wt U_r)$ for some $r$.   There is a least $m$ such that $B= (\wt D_{r, [p^{-m}]})^k$ for some $k< p^m$.  Then  $p$ does not divide $k$, so $B= \wt D_{r, [kp^{-m}]}$. 

 \medskip
 
 \n   \emph{We next   treat   the   case of $G=\ZZ \ltimes \QQ_p$.} Let $\+ V$ be the Haar computable copy of~$\+ W(G)$ obtained in \cref{ex:Qp}, and let $\wt {\+ V}$ be a further  Haar computable copy of~$\+ W(G)$. Using the notation of \cref{ex:Qp}, let 
 \bc $E_{z,r,a}= g^zD_{r,a}$ for each $z, r \in \ZZ, a \in C_{p^\infty}$.   \ec We list some properties of these elements of $\+ V$ that will  be needed shortly. Note that we can view $\+ W$ as embedded into $\+ V$ by identifying $\la r,a\ra$ with $\la 0,r,a\ra$. 
Also note   that $E_{z,r,a} \colon U_{r-z} \to U_r$ (using the category notation of  \cref{rem:cat}). Since $D_{r+1, a} \sub D_{r,pa}$, we have 
   \begin{equation} \label{eqn:EEE} E_{z, r+1, a } \sub E_{z,r,pa}. \end{equation} 
  Furthermore,   \begin{equation} \label{eqn:EEE2} E_{z, r, 0 }= g^{z} U_r= U_{r+z} g^{z}= (g^{-z} U_{r-z})^{-1}= (E_{-z, r-z, 0 })^{-1}. \end{equation}   
    
 By hypothesis on $\wt {\+ V}$, there is a  meet groupoid  isomorphism $\ol \Gamma \colon \+ V \to \wt {\+ V}$.     Since $G$ has no compact open subgroups besides the ones present in $\+ W(\QQ_p)$,   the  family  $(\wt U_r)_{r \in \ZZ}$, where  $\wt U_r = \ol \Gamma(U_r)$,      is   computable in $\wt {\+ V}$  by the same argument as before.   The set of elements  $A$  of $\wt {\+ V}$ that are a  left and a right coset of the same  subgroup is computable by checking whether $A^{-1} \cdot A= A \cdot A^{-1}$. The operations of $\wt {\+ V}$   induce a Haar computable meet groupoid   $\wt {\+ W}$ on this set.  Clearly  the restricted  map    $\Gamma =\ol \Gamma \mid  \+ W$  is an isomorphism $\+ W \to \wt {\+ W}$. So by the case of $\QQ_p$, there is a computable isomorphism $\Delta \colon  \+ W \to \wt {\+ W}$.  
%
 
  We will extend $\Delta$ to a computable  isomorphism $\ol \Delta \colon \+ V \to \wt {\+ V}$.   The following summarizes the setting: 
\[\xymatrix{
   \+ V   \ar[r]^{\ol \Gamma , \ol \Delta}& \wt{\+ V}  \\
    \+ W \ar[u]^{\sub}      \ar[r]_{\Gamma, \Delta}                     & \wt { \+ W} \ar[u]_{\sub} } \]%
In five phases we define a computable family $\wt E_{z,r,a}$ ($z, r \in \ZZ, a \in C_{p^\infty}$), and then let $\ol \Delta(E_{z,r,a})= \wt E_{z,r,a}$.     As before write $\wt D_{r,a} = \Delta( D_{r,a})$.

  \bi \item[(a)] Let $\wt E_{0,r,a} = \wt D_{r,a}$.
 Choose   $F_0:= \wt E_{-1,0,0} \colon \wt U_1 \to \wt U_0$
  \item[(b)] compute $F_r:=\wt E_{-1,r,0}\colon U_{r+1} \to U_r$ by recursion on $|r|$, where $r \in \ZZ$, in such a way  that $\wt F_{r+1} \sub \wt F_r$ for each $r \in \ZZ$; this is possible by (\ref{eqn:EEE}) and since $\+ V \cong \wt {\+ V}$ via $\ol \Gamma$. 
  \item[(c)]
   For $z < -1$, compute  $ \wt E_{z,r,0}\colon U_{r-z} \to U_r$ as follows: 
   \bc $\wt E_{z,r,0}= F_{r-z-1} \cdot   F_{r-z-2} \cdot \ldots \cdot   F_r$. \ec
  \item[(d)] For $z>0$ let $\wt E_{z,r,0}= (\wt E_{-z, r-z, 0 })^{-1}$; this is enforced by (\ref{eqn:EEE2}). 
  \item[(e)] Let $\wt E_{z,r,a}= \wt E_{z,r,0} \cdot \wt D_{r,a}$. 
  \ei   
One verifies that $\ol \Delta$  preserves the   meet groupoid operations (we omit the formal detail).  To show that $\ol \Delta$ is onto, suppose that $\wt E \in \+ V$ is given. Then $\wt E = \Gamma(E_{z,r,a})$ for some $z,r,a$. By (\ref{eqn:EEE}) we may assume that $z<0$.  Then   $  E_{z,r,0}=   \prod_{i=1}^{-z}E_{-1,r-z-i,0} $ as above. So, writing $F_s$ for $\wt E_{-1,s,0}$, we have  $\Gamma(E_{z,r,0})= \prod_{i=1}^{-z}F_{r-z-i} \wt D_{r-z-i, a_i}$ for some $a_i \in C_{p^\infty}$. 

Note  that  $\wt S(D)= F \cdot D \cdot F^{-1}$ for each $D \in L(\wt U_r)  \cap \wt {\+ W}$ and $F \colon \wt U_{r+1} \to  \wt U_r$. 
For, the analogous statement clearly holds in $\+ V$;  then one uses  that  $\wt S=\Gamma \circ S \circ \Gamma^{-1}  $, and that $\ol \Gamma \colon \+ V \to \wt {\+ V}$ is an isomorphism.   
Since $\wt D_{r+1,a}= \wt S(\wt D_{r,a})$, we may conclude that  $\wt D_{r+1, a}\cdot F = F \cdot D_{r,a}$ for each such $F$. We can use these ``quasi-commutation relations"  to simplify the expression $\prod_{i=1}^{-z}F_{r-z-i} \wt D_{r-z-i, a_i}$ to  $\wt E_{z,r,0} \wt D_{r,b}$ for some $b \in C_{p^\infty}$. Hence $\wt E = \wt E_{z,r,0} \wt D_{r,b} \wt D_{r,a}$. This  shows that $\wt E$ is in the range of $\ol \Delta$, as required.
\end{proof}

 \section*[Appendix 1: A noncomputable scale function]{Appendix 1 (with George Willis):  a non-computable scale function}
  \label{s:scale noncomp}
  \setcounter{section}{11}
  	Let $\mathbb F$ be some finite field; for notational simplicity we will assume that  $|\mathbb F|=2$.  We write $H$ for the additive group  of $\FF((t))$. The group 	 $\Aut(H)$ is equipped with the usual Braconnier topology recalled before \cref{prop: Braconnier}.    For $\pi \in \Aut(H)$ let  $s(\pi) = \min_V \{|\pi(V) \colon V \cap \pi(V)| $ where as usual $V$ ranges over compact open subgroups.  Willis \cite[Example  1]{Willis:17} showed that  $s \colon \Aut(H) \to \NN^+$   is not   continuous. He defined    automorphisms $\pi_t$  $(t \in \NN)$ and $\pi$ of $H  $ such that $\lim_t \pi_t = \pi$, $s(\pi_t)= 1$ and $s(\pi)=2$.   
  Starting from  his example we  show the following.

 \begin{theorem}[with George Willis] \label{th:scale noncomp}
 	There is a computable Baire presentation $(T, \Op, \Inv)$ of a t.d.l.c.\ group  $G$ and a computable sequence of paths $(\ol g_i)_{i \in \NN^+}$ on~$T$ such that the function $i \mapsto s(\ol g_i)$ is non-computable. Moreover,  $G$ is elementary in the sense of \cite{Wesolek:15}.
 \end{theorem}
 \begin{proof}
 	We begin with a brief outline. Let $\+ K\sub \NN $ denote the usual halting problem,  a standard recursively enumerable, undecidable set in computability theory.   Uniformly in $i\in \NN^+$ we  will   build a computably t.d.l.c.\ group $G_i$ with a  computable element $g_i$ so that  $s(g_i)= 2$ if $i \not \in \+ K$, and $s(g_i)=1$ else. Along with $G_i$,       uniformly in $i$ we  determine a compact open subgroup $U_i$. Then the  statement of the theorem holds for the    t.d.l.c.\ group  $G=\bigoplus \sNp i  (G_i, U_i) $,  with the computable  presentation according to     the proof of \cref{prop:direct products}, where $\ol g_i$ is the image of $g_i$ under the canonical embedding $G_i \to G$.

 	We let $G_i $ be the split extension $\ZZ \ltimes H$ for the action of $\ZZ$ corresponding to an automorphism $\hat \beta_i$ of $H$.
 To define $\hat \beta_i$ 	we   combine the technique of Willis with an argument  from computability theory.  Informally speaking,  if  a number $i$ enters    $\+ K$  at a stage~$t$,    the approximation of   $\hat  \beta_i$   changes  from ``following" $\pi$ to ``following" $\pi_t$.  We have $s_G(\ol g_i)= s_H(\hat  \beta_i)$ where  $g_i \in G_i$ is the element such that conjugation by $g_i $ induces the automorphism  $\ol \beta_i$ on $H$.

 	We proceed to the details. We  use the computable Baire presentation $(Q, \Op, \Inv)$ of $H$ given by  the proof of \cref{prop: SL2}.  Recall that   strings on $Q$ other than the root are of the form $r\sss$  where $r\in \NN$,  $\sss \in \{0,1\}^*$ is a binary string,  and
 	if  $r>0$ then~$\sss$ does not start with $0$.  We  think of $r\sss$ as denoting the formal Laurent polynomial $x^{-r} \sum_{0 \le k< \sssl} \sss(k) x^k$; note that $\sss$ is allowed to end in $0$.

 	For  $c\in \NN$, we say that a permutation $\alpha$ of $\ZZ$ is  \emph{$c$-bounded}   if $|\alpha(z)-z| \le c$ for each $z \in \ZZ$. In this case  the function $\hat \alpha$ defined on $H$  by \[\hat \alpha(\sum_{k \in \ZZ}  r_kx^k)= \sum_{k \in \ZZ}  r_{\alpha(k)}x^{\alpha(k)}\] 
 	is a continuous  automorphism of $H$. 
 	\begin{claim} \label{cl:otti}
 		Let $\alpha$ be a computable $c$-bounded  permutation  of $\ZZ$. Then 	$\hat \alpha \colon [Q] \to [Q]$ is computable,  uniformly in $c$ and a Turing machine index for computing $\alpha$.
 	\end{claim}
 	We verify this by defining a  monotonic computable function $L$ on the tree $Q$ according to   \cref{rem:monot}.  For $r\in \NN$ and a binary string $\sss$ such that   $\sssl \ge 2c+1$, we declare that  $L(r\sss)= s\tau$   where
 	\bi \item[(a)]  $-s$ is the minimum of $0$ and the values $\alpha(k-r)$, $k \le 2c+1$  such that $\sss(k) \neq 0$; if there is no such value (and hence  $r=0$)   let $s=0$;

 	\item[(b)]  $\tau$ is obtained by  searching for  the maximal  $\ell = |\tau| $    such that   for each $k< \ell$ the value \bc $\tau(k)= \sss(\alpha^{-1}(k-s)+r)$ \ec is defined. 
 	\ei
 	If $s>0$ then $\tau(0)= \sss(k)= 1$ where $k-r$ is the position at which  the minimal value is taken in (a), so indeed $s\tau\in Q$. Clearly $L$ is monotonic. Since $\aaa$ is c-bounded, given $r \in \NN$ and  $\sss$ with  $\sssl \ge 2c+1$,   we have $L(r\sss)= s\tau$ for some  $s\in \NN$ and $\tau$ such that    $|\tau| \ge \sssl -c$. Using this,  one  verifies that $L$ is as required. This shows the claim.


 	Since $\+ K$ is recursively enumerable,  $\+ K = \bigcup_t \+ K_t$ for a computable sequence  of strong indices for finite subsets of $\NN$. For technical reasons we  will assume that $\+ K_t- \+ K_{t-1} \neq \ES$  implies that  $t $ is even. We also let  $\+ K_t = \ES $  for  $t<0$. 
 	The following defines a computable function $\NN^+ \times \ZZ \to \ZZ$ via $\la i, t \ra \mapsto \beta_i(t)$:
 	\[ \beta_i(t)= \begin{cases}
 		t+2 &\text{if $t $ is even and $i \not \in \+ K_t$}\\
 		t-2  &\text{if $t $ is odd and $i \not \in \+ K_t$} \\
 		t+1&\text{if $t $ is even and $i   \in \+ K_t - \+ K_{t-1}$} \\
 		t &\text{if $i  \in \+ K_{t-1}$} 
 	\end{cases}\]
 	If $i \not \in \+ K$ then $\beta_i$ is the permutation of $\ZZ$ that adds 2 to even numbers, and subtracts 2 from odd numbers. So $s(\hat \beta_i)=2$. If   $t$ is least such that $i \in \+ K_t$ then the nontrivial cycle of  $\beta_i$ ``turns around" at position $t$.  Hence $\hat \beta_i$ leaves invariant the compact open subgroup consisting of the  Laurent series that begin at position $t+2$. So   $s(\hat \beta_i)=1$. 
 	
 	Now let $G_i = \ZZ \ltimes_{\gamma_i} H$, where $\gamma_i$ is the action $\ZZ \times H \to H$ given by $\gamma_i(z,h) = \ol \beta_i^{z}(h)$. Let $\ZZ$ have a computable Baire presentation according to \cref{compBaire discrete}. By   \cref{cl:otti}, $\gamma_i$   is computable. So  the proof of  \cref{prop:sdprod}  yields a  computable Baire presentation of $G_i$. Note that $U= \mathbb F[[t]]$ is a  compact open subgroup of each $G_i$. All the    assertions in this paragraph  hold uniformly in $i$.  
 	
 	Let $g_i$ be the generator of $\ZZ$ in $G_i$ whose conjugation action on $H$ induces  $\beta_i$.  Then $s_{G_i}(g_i) = s_H(\hat  \beta_i)$ because $G_i$ and $H$ have the same compact open subgroups.  
 	
 	Let $G= (T, \Op, \Inv) $ be the computable Baire presentation of the local direct product $G=\bigoplus \sNp i  (G_i, U) $ obtained in the proof of  \cref{prop:direct products}. It is clear that the sequence $(\ol g_i)$ of  paths of $T$ such that $\ol g_i$ represents $g_i$ viewed as an element of $G$ is uniformly computable.   It is easy to check that $s_G(\ol g_i) = s_{G_i}(g_i)$ for each $i$. (The canonical projection $G \to G_i$ is open, and   to each   compact open subgroup $V_i$ of $G_i$ corresponds   a compact open subgroup $W$ of $G$ where $W = \prod_{k \in \NN^+} V_k$ with $V_k= U$ for $k \neq i$.)  Thus $s_G(\ol g_i ) = 1 $  iff $i \in \+ K$, as required.
 \end{proof}

\section*{Appendix 2:  Abelian t.d.l.c.\   groups} \label{s: Lupini} 
\setcounter{section}{12}
Our work~\cite{Lupini.etal:21}, joint  with Lupini,    studies locally compact  \emph{abelian} groups with computable presentations. In  particular, we investigate the algorithmic content of Pontryagin -- van Kampen duality for such groups. Each abelian t.d.l.c.\ group is procountable, i.e., an   inverse limit of countable groups.  This enables us in~\cite{Lupini.etal:21} to   use a  notion of computable presentation for  t.d.l.c.\ abelian groups convenient for the given context.   This notion, which in the present  paper we will call a \emph{computably procountable presentation with  effectively finite kernels}, is  reviewed in  \cref{def:procountable} below.  The main purpose of this short section is to show that    the  definition of computably t.d.l.c.\  abelian groups in~\cite{Lupini.etal:21} is equivalent to the one given here. 
\subsection{Procountable groups} We review some concepts, mainly from \cite[Section~3]{Lupini.etal:21}.
Suppose   we are given a sequence of  groups  $(\+ A_i)\sN i$   such that each $A_i$ is   countable discrete. Suppose we are also given  epimorphisms   $\phi_i: \+ A_{i} \rightarrow \+ A_{i-1} $    for  each $i>0$.  Then $\varprojlim (\+ A_{i},\phi_i)$ can  be concretely defined as the closed subgroup of the topological group $\prod_{i\in \NN} A_i$ consisting of those $g$ such that $\phi_i(g(i))= g(i-1)$ for each $i>0$.
\begin{definition} \label{df:proc} A~topological group $G$ is called \emph{procountable} if   $G\cong \varprojlim (\+ A_{i},\phi_i)$ for some sequence $(\+ A_i,\phi_i)\sN i$ as above. 
\end{definition} 

The t.d.l.c.\ groups  that are pro-countable are precisely the SIN groups (where SIN stands for ``small invariant neighbourhoods); see \cite[Section~2.2]{Wesolek:15}.
\begin{remark} \label{rem: wn} Let  $G$ be a closed subgroup of $\S$,  let $\seq{N_i}\sN i$ be  a descending sequence of open normal subgroups of $G$ with trivial intersection,  and let  $\+ A_i= G/N_i$. Then  $G \cong \varprojlim_{i>0} (\+ A_{i},\phi_i)$ where the $\phi_i$ are the canonical maps. This is well-known; see  \cite[Lemma 2]{Malicki:14}. \end{remark}

	\subsection{Computably procountable groups} 
	Extending \cref{def:profinite}  of computable profinite presentations,   (2) below     allows the $\+ A_i$ to be discrete computable groups, while  retaining the condition that the kernels of the connecting maps be finite and given by strong indices.
	\begin{definition}[\cite{Lupini.etal:21}, Definition\ 3.4] \label{def:procountable}  \label{def:procountable_t.d.l.c.} \mbox{} \bi \item[(1)] A computable  presentation of a procountable group $G$ is a sequence $(\+ A_{i},\phi_i)_{ i \in \NN}$ of discrete groups $\+ A_i$ and epimorphisms $\phi_i: \+ A_{i} \rightarrow \+ A_{i-1} $ (for $i>0$) such that $G\cong \varprojlim (\+ A_{i},\phi_i)$,    each group $\+ A_i $  is uniformly computable as a discrete group, and the sequence of maps $(\phi_i)_{ i \in \NN^+}$ is uniformly computable. 
		
		\item[(2)] Suppose that 
		$\mathit  {ker\ } \phi_i$ is finite for each $i$, so that $G$ is locally compact by \cite[Fact~3.2]{Lupini.etal:21}. We say that   $(\+ A_{i},\phi_i)_{ i \in \NN}$ is a \emph{computably procountable presentation with  effectively finite kernels} if  in addition,   from $i$ one can compute a strong index for $\mathit  {ker\ } \phi_i$ as a subset of $A_i$. \ei\end{definition}

	\begin{fact} If $G$ is compact and has a computable procountable presentation with  effectively finite kernels, then this presentation is  a computable profinite presentation. \end{fact}
	%
	
	\begin{proof}    Since  $G$ is compact,   the $\+ A_i$ are finite. All we need is strong indices for the  $\+ A_i$ as groups. Since the $\+ A_i$ are  computable groups uniformly in $i$,    it suffices to compute the size   $|\+ A_i|$ uniformly in $i$. One  can do this recursively,  using that $|\+ A_i| = |\mathit {ker}\,  \phi_i | \times |\+ A_{i-1}|$. \end{proof}

\begin{prop}\label{prop:proctotdlc}  If  a t.d.l.c.\ group $G$ has a  procountable  presentation   $(\+ A_{i},\phi_i)_{ i \in \NN}$ with  effectively finite kernels, then $G$ has a computable Baire presentation.
\end{prop}
\begin{proof} This is a straightforward extension of the argument in the profinite setting,  ~\cref {ex:prof Baire}. For simplicity we may assume that $\+ A_0$ is infinite.  So we can effectively identify the elements of $\+ A_i$ with $\NN$, and hence view $\phi_i $ as a map $\NN \to \NN$. As before,    
\bc $T = \{ \sss \in \NN^* \colon \forall i< |\sss| [\sss(i) \lland  \phi_i(\sss(i)) = \sss(i-1) \text{ if } i>0 \}$. \ec 
It is clear  from the hypotheses that  $[\sss]_T$ is compact iff $\sss$ is a nonempty string,  and  hence that  $T$ is computably  locally compact via $k=1$. The rest is as before, using that the $\+ A_i$ are computable groups uniformly in $i$. \end{proof}

For abelian, as well as for  compact, t.d.l.c.\  groups, we can provide  a converse to the foregoing observation. The compact case  essentially  restates   \cite[Theorem~1]{Smith:81} for  the recursively profinite case.

\begin{prop}\label{lem:proc}
Let $G$ be an infinite  computably t.d.l.c.~group that is abelian or compact. 
Then $G$ has a computably procountable presentation with  effectively finite kernels. 
\end{prop}

\begin{proof}  
Suppose that the meet groupoid $\+ W(G)$ has a Haar  computable copy $\+ W$ as in Definition~\ref{Def2}. We may assume that its domain is all of $\NN$, and that   $0$ denotes a (compact open) subgroup. In the framework of that  copy one can compute a descending sequence $\seq{U_i}\sN i$ of compact open   subgroups of $G$, such that  $U_0$ is the group denoted by $0$ and,   for each  compact open $U$,  there is an $i$ with $U_i \sub U$.  In the abelian case, trivially each  $U_i$ is normal;   in the case that $G$ is compact, we effectively shrink the $U_i$  so that  they are also normal.  To do so, by  the hypothesis  that $\+ W$ is Haar computable,  we can compute from $i$ a strong index for the set   $\{B_1, \ldots, B_k\}$ of  distinct right  cosets  of $U_i$; now replace  $U_i$ by $\bigcap_{r=1}^k B_r^{-1} \cdot U_i \cdot B_r$. 

Let $\+ A_i$ be the automorphism group of the object $U_i$ in the groupoid  $\+ W$ viewed as a category;  that is, $\+ A_i$ is the set of cosets of $U_i$,  with the groupoid operations restricted to it. For $i >0$ let $\phi_i \colon \+ A_{i} \to \+ A_{i-1}$ be the map sending $B \in \+ A_i$ to the unique coset of $U_{i-1}$ that contains $B$. Since   $\+ W$ is  Haar computable,   the condition in (1)  of  Definition~\ref{def:procountable} holds. Also, the kernel of $\phi_i$ is the set of cosets of $U_i$ contained in $U_{i-1}$. By Definition~\ref{Def2}(b) we can compute the number of such cosets, so we can compute a strong index for the kernel. Hence  the condition in  (2) of Definition~\ref{def:procountable}   also holds.

We complete the proof by verifying the following claim.
\begin{claim}  $G \cong   \varprojlim_{i>0} (\+ A_{i},\phi_i)$. \end{claim}
As in  \cref{def:Gof}, let    $\wt  G $ be the closed subgroup of $ \S$   of elements $p$   that  preserve the inclusion relation  on $\+ W$ in both directions,  and satisfy  $p(A) \cdot B = p(A\cdot B) $ whenever $A \cdot B$ is defined. Recall that $\wt G \cong G$. So  it suffices to show that $\wt G \cong   \varprojlim (\+ A_{i},\phi_i)$. 

Let $\+ N_i$ be the stabilizer of $U_i$, which is a  compact open subgroup of $\wt G$. Since each $p\in \wt G$ preserves the inclusion relation, we have $\+ N_i \sub \+ N_{i-1}$ for $i >0$. By the choice of the $U_i$, the intersection of the $\+ N_i$ is trivial. Each $\+ N_i$ is normal: if $p \in \+ N_i$ and $q \in \wt G$, let $B= q(U_i)$, a left,  and hence also right, coset of $U_i$. Then \bc $q^{-1}pq(U_i)= q^{-1}p (U_iB)=q^{-1}(U_iB) = q^{-1}(B)U_i= U_i$. \ec So $q^{-1}pq \in \+ N_i$. 

For each $i$, the map $\wt G \to \+ A_i$ sending $p$ to $p(U_i)$ induces a group isomorphism $\wt G / \+ N_i \to \+ A_i$. For each $i>0 $ the natural map  $\aaa_i \colon \wt G/\+ N_i \to \wt G/\+ N_{i-1}$ induced by the identity on $\wt G$ corresponds to   $\phi_i$ via these isomorphisms. Thus   $\varprojlim_{i>0} (\wt G/\+ N_{i},\aaa_i)\cong  \varprojlim_{i>0} (\+ A_{i},\phi_i)$. The claim now follows in view of Remark~\ref{rem: wn}.
\end{proof}

  \n {\bf Acknowledgements.}  	The authors would like to thank Stephan Tornier and George Willis for very helpful discussions and insights.}

 \n {\bf Funding.} AM was supported by Rutherford Discovery Fellowship RDF-VUW1902 of  the Royal Society Te Apārangi. AN was supported by the Royal
 Society Te Apārangi under the standard Marsden grant UOA-1931.

\def\cprime{$'$} \def\cprime{$'$} \def\cprime{$'$} \def\cprime{$'$}
  \def\cprime{$'$} \def\cprime{$'$}

\end{document}